\documentclass[12pt, letterpaper, colorinlistoftodos]{amsart}
\usepackage{a4wide}
\usepackage{amssymb}
\usepackage{amsthm}
\usepackage{amsmath}
\usepackage{amscd}
\usepackage{verbatim}
\usepackage{bm}
\usepackage[mathscr]{eucal}
\usepackage[all]{xy}
\usepackage{todonotes}
\usepackage[pagebackref=true]{hyperref}
\usepackage{dsfont, stmaryrd}
\usepackage{mathrsfs}
\usepackage{tikz, tikz-cd, caption, tabu}
\usepackage{tkz-euclide}
\numberwithin{equation}{section}

\theoremstyle{plain}
\newtheorem{theorem}{Theorem}[section]
\newtheorem{corollary}[theorem]{Corollary}
\newtheorem{lemma}[theorem]{Lemma}
\newtheorem{proposition}[theorem]{Proposition}

\newtheorem{conjecture}[theorem]{Conjecture}
\theoremstyle{definition}

\newtheorem{remark}[theorem]{Remark}

\theoremstyle{remark}

\newcommand{\OO}{\mathcal O}
\newcommand{\A}{\mathbb{A}}
\newcommand{\R}{\mathbb{R}}

\newcommand{\Q}{\mathbb{Q}}
\newcommand{\Z}{\mathbb{Z}}

\newcommand{\ord}{\operatorname{ord}}

\newcommand{\GSpin}{\operatorname{GSpin}}
\newcommand{\Gspin}{\operatorname{GSpin}}

\newcommand{\SL}{\operatorname{SL}}
\newcommand{\PSL}{\operatorname{PSL}}
\newcommand{\cha}{\operatorname{Char}}

\newcommand{\fff}{\operatorname{if }}

\newcommand{\Ind}{\operatorname{Ind}}

\newcommand{\pmat}[4]{\begin{pmatrix}
                 #1 & #2\\
                 #3 & #4
\end{pmatrix}}
\newcommand{\smat}[4]{\left(\begin{smallmatrix}
                 #1 & #2\\
                 #3 & #4
\end{smallmatrix}\right)}
\newcommand{\kzxz}[4]{\left(\begin{smallmatrix} #1 & #2 \\ #3 & #4\end{smallmatrix}\right) }

\newcommand{\lp}{\left (}
\newcommand{\rp}{\right )}

\newcommand{\Fcr}{{\mathscr{F}}}
\newcommand{\Fc}{{\mathcal{F}}}

\newcommand{\Oc}{{\mathcal{O}}}
\newcommand{\Sc}{{\mathcal{S}}}
\newcommand{\Zb}{\mathbb{Z}}

\newcommand{\Qb}{\mathbb{Q}}
\newcommand{\SO}{{\mathrm{SO}}}
\newcommand{\GL}{{\mathrm{GL}}}
\newcommand{\sgn}{{\mathrm{sgn}}}

\newcommand{\df}{\mathfrak{d}}
\newcommand{\ef}{\mathfrak{e}}

\newcommand{\wf}{\mathfrak{w}}

\newcommand{\vf}{\mathfrak{v}}
\newcommand{\bvf}{\overline{\mathfrak{v}}}
\newcommand{\pf}{\mathfrak{p}}

\newcommand{\Nm}{{\mathrm{Nm}}}
\newcommand{\Ab}{\mathbb{A}}
\newcommand{\Db}{\mathbb{D}}
\newcommand{\Nb}{\mathbb{N}}
\newcommand{\Hb}{\mathbb{H}}
\newcommand{\Rb}{\mathbb{R}}
\newcommand{\Cb}{\mathbb{C}}
\newcommand{\Wb}{\mathbb{W}}
\newcommand{\Gm}{\mathbb{G}_m}
\newcommand{\ebf}{{\mathbf{e}}}

\newcommand{\tf}{\tilde{f}}

\newcommand{\tH}{{\tilde{H}}}

\newcommand{\tQ}{\tilde{Q}}

\newcommand{\tx}{\tilde{x}}

\newcommand{\tr}{\operatorname{Tr}}

\newcommand{\norm}{\operatorname{N}}
\newcommand{\Gal}{\operatorname{Gal}}
\newcommand{\Spin}{\operatorname{Spin}}

\newcommand{\taub}{{\boldsymbol \tau}}
\newcommand{\ub}{{\boldsymbol u}}
\newcommand{\vb}{{\boldsymbol v}}
\newcommand{\mut}{{t}}
\newcommand{\vc}{\varrho}
\newcommand{\ve}{\varepsilon}
\newcommand{\e}{\epsilon}

\newcommand{\tvc}{\tilde{\vc}}
\newcommand{\tI}{\tilde{\mathcal{I}}}
\newcommand{\Ic}{\mathcal{I}}

\newcommand{\vt}{\vartheta}
\renewcommand{\a}{a}
\newcommand{\ab}{\mathrm{ab}}
\newcommand{\RC}{\mathrm{RC}}
\newcommand{\tRC}{\widetilde{\mathrm{RC}}}
\newcommand{\slf}{{\mathfrak{sl}}}
\newcommand{\spf}{{\mathfrak{sp}}}
\newcommand{\half}{{\tfrac{1}{2}}}

\newcommand{\Ps}{\mathscr{P}}
\newcommand{\zf}{\mathfrak{z}}
\newcommand{\Ff}{\mathfrak{F}}
\newcommand{\Vo}{V_{\circ}}
\newcommand{\vo}{v_{\circ}}

\newcommand{\vol}{\mathrm{vol}}
\newcommand{\tc}{\tilde{c}}
\newcommand{\Ec}{\mathcal{E}}
\newcommand{\tW}{\tilde{W}}

\newcommand{\fac}{\frac{\pi}{3}}
\newcommand{\facinv}{\frac{3}{\pi}}
\newcommand{\TT}{T^\Delta}
\newcommand{\Qip}{\mathbb{Q}(\zeta_{p^\infty})}
\newcommand{\Qab}{\mathbb{Q}^{\mathrm{ab}}}
\newcommand{\reg}{\mathrm{reg}}
\newcommand{\Diff}{\mathrm{Diff}}
\newcommand{\V}{\mathrm{V}}
\newcommand{\tkappa}{\tilde{\kappa}}
\newcommand{\Vc}{\mathrm{V}_\circ}
\newcommand{\Lc}{L_\circ}
\newcommand{\vv}{\mathrm{vv}}

\newcommand{\prin}{\mathrm{prin}}

\newcommand{\oneb}{\mathds{1}}
\newcommand{\error}{\mathrm{error}}
\newcommand{\alv}{{\alpha^\vee}}

\AtBeginDocument{%
   \def\MR#1{}
}

\begin{document}
\title{Deformations of Theta Integrals and A Conjecture of Gross-Zagier}
\author[J.~Bruinier]{Jan H.~ Bruinier}
\author[Y.~Li]{Yingkun Li}
\address{Fachbereich Mathematik,
Technische Universit\"at Darmstadt, Schlossgartenstrasse 7, D--64289
Darmstadt, Germany}
\email{bruinier@mathematik.tu-darmstadt.de}
\email{li@mathematik.tu-darmstadt.de}

\author[T.~Yang]{Tonghai Yang}
\address{Department of Mathematics, University of Wisconsin Madison, Van Vleck Hall, Madison, WI 53706, USA}
\email{thyang@math.wisc.edu}
\thanks{The third author was partially supported by UW-Madison Kelley Mid-Career Award.  }

\begin{abstract}
  In this paper, we complete the proof of the conjecture of Gross and Zagier concerning algebraicity of higher Green functions at a single CM point on the product of modular curves.
  The new ingredient is an analogue of the incoherent Eisenstein series over a real quadratic field, which is constructed as the Doi-Naganuma theta lift of a deformed theta integral on hyperbolic 1-space.
\end{abstract}
\maketitle

 \makeatletter
 \providecommand\@dotsep{5}
 \def\listtodoname{List of Todos}
 \def\listoftodos{\@starttoc{tdo}\listtodoname}
 \makeatother

\tableofcontents
\section{Introduction}
Just over half a century ago, Doi and Naganuma discovered a Hecke-equivariant lifting map from a weight $k$ elliptic modular form to a weight $(k, k)$ Hilbert modular form for a real quadratic field $F$ \cite{DN70}.
This is a special case of cyclic base change \cite{JL70}, which has now become a basic and useful tool in the theory of automorphic forms and automorphic representations.
By the exceptional isogeny
\begin{equation}
  \label{eq:isog}
 \mathrm{O}(2, 2) \sim \mathrm{Res}_{F/\Qb} \SL_2,
\end{equation}
the Doi-Naganuma lifting is also an instance of a theta lifting from $\SL_2$ to $\mathrm{O}(2, 2)$ \cite{Kudla78}.

\subsection{A Problem posed by Gross and Zagier}
In the seminal paper \cite{GZ86}, Gross and Zagier proved their formula relating the central derivative of some Rankin-Selberg $L$-function attached to a weight 2 level $N$ newform $f$ and the N\'eron-Tate height pairing of $f$-isopytic components of Heegner points in the Jacobian of the modular curve $X_0(N)$.
This was extended in \cite{GKZ87} to describe the positions of these Heegner points in the Jacobian using Fourier coefficients of modular forms.
In the degenerate case $N = 1$, the Gross-Zagier formula yields a beautiful factorization formula of the norm of differences of singular moduli \cite{GZ85}.

To calculate the archimedean contribution to the height pairings, one requires the  automorphic Green function
\begin{equation}
  \label{eq:Gs}
  \begin{split}
    G^{\Gamma_0(N)}_s(z_1, z_2) &:= -2 \sum_{\gamma \in \Gamma_0(N)} Q_{s-1} \lp 1 + \frac{|z_1 - \gamma z_2|^2}{2 \Im(z_1) \Im(\gamma z_2)}\rp,~ \Re(s) > 1, \\
    Q_{s-1}(t) &:= \int^\infty_0 ( t + \sqrt{t^2 - 1} \cosh (u))^{-s} du
  \end{split}
\end{equation}
on $X_0(N) \times X_0(N)$.
It is an eigenfunction with respect to the Laplacians in $z_1$ and $z_2$ with eigenvalue $s(1-s)$.
The function vanishes when one of the $z_i$ approaches the cusps, and has a logarithmic  singularity along the diagonal.
In fact, these properties characterize it uniquely.
Using Hecke operators acting on either $z_1$ or $z_2$, we can define
\begin{equation}
  \label{eq:Gm}
  \begin{split}
    G^{\Gamma_0(N), m}_{s}(z_1, z_2)
    &:= \sum_{\gamma \in \Gamma_0(N) \backslash R_N,~ \det(\gamma) = m} G_s^{\Gamma_0(N)}(z_1, \gamma z_2)\\
    &= G^{\Gamma_0(N)}_s(z_1, z_2) \mid T_{m, z_1} = G^{\Gamma_0(N)}_s(z_1, z_2) \mid T_{m, z_2},
  \end{split}
\end{equation}
where $R_N := \{\smat{a}{b}{Nc}{d}: a, b, c, d \in \Zb\}$.
Then $G^{\Gamma_0(N), m}_s$ has a logarithmic singularity along the $m$-th Hecke correspondence $T_m \subset X_0(N)^2$ (see (1.2) in Chapter II of \cite{GZ86}).

For integral parameters $s = r + 1 \in \Nb_{\ge 2}$, these functions are called \textit{higher Green functions}.
In Section V.1 of \cite{GKZ87}, two problems about these functions were raised.
The first one was to give an interpretation of their values at Heegner points as archimedean contributions of certain higher weight height pairings.
This was answered by Zhang in \cite{Zhang97} (see also \cite{Xue10}), where the N\'eron-Tate height pairing of Heegner points is replaced by the arithmetic intersection of Heegner cycles on Kuga-Sato varieties.

The second problem dealt with the algebraicity of higher Green functions at a single CM point.
Let $M^{!, \infty}_{-2r}(\Gamma_0(N))$ be the space of weakly holomorphic modular forms for $\Gamma_0(N)$ of weight $-2r < 0$ with poles only at the cusp infinity (see \eqref{eq:M!!}).
Given $f = \sum_{m \gg -\infty} c(m) q^m \in M^{!, \infty}_{-2r}(\Gamma_0(N))$, we call the following linear combination of  higher Green functions
\begin{equation}
  \label{eq:Grf}
    G^{\Gamma_0(N)}_{r+1, f}(z_1, z_2) := \sum_{m \in \Nb} c(-m) m^{r} G^{\Gamma_0(N), m}_{r+1}(z_1, z_2)
\end{equation}
the \textit{principal higher Green function associated to $f$}.
Along the divisor
$$
Z_f := \sum_{m \ge 1,~ c(-m) \neq 0} T_m. %
$$
the function $G_{r+1, f}^{\Gamma_0(N)}$ has a logarithmic singularity.
Using Serre duality, this function is the same as the higher Green function defined via relations in Section V.4 of \cite{GZ86} (see Remark \ref{rmk:relation}).
We say that it is \textit{rational} when $f$ has rational Fourier coefficients at the cusp infinity.
Even though the theory of complex multiplication does not directly apply as in the case of automorphic Green functions, the value of a rational, principal higher Green function $G^{\Gamma_0(N)}_{r+1, f}$ at a single CM point on $X_0(N)\times X_0(N)$ should be algebraic in nature, predicted by the following conjecture (see e.g.\ \cite{Mellit08} and \cite{Via11}).

\begin{conjecture}
  \label{conj:GZ}
  Suppose $f \in M_{-2r}^{!, \infty}(\Gamma_0(N))$ has rational Fourier coefficients at the cusp infinity. Then for any CM point $(z_1, z_2) \in X_0(N)^2 \backslash Z_f$ with $z_j$ having discriminant $d_j < 0$, there exists $\alpha = \alpha(z_1, z_2) \in \overline{\Qb} \subset \Cb$ such that
  $$
G^{\Gamma_0(N)}_{r+1, f}(z_1, z_2) = |d_1d_2|^{-r/2} \log |\alpha|.
  $$
\end{conjecture}
Over the years, there have been a lot of partial results towards this conjecture.
When $d_1d_2$ is a perfect square, this conjecture was proved in \cite{Zhang97} conditional on the non-degeneracy of the height pairing of CM cycles.
Using regularized theta liftings, an analytic proof was given in \cite{Via11} with restrictions on $N, d_j$ and later in  full generality in \cite{BEY21}.
When $d_1d_2$ is not a perfect square, less was known before.
For $N = 1, z_1 = i$ and $r = 1$, Mellit proved the conjecture in his thesis \cite{Mellit08} using an algebraic approach.
When one averages over the full Galois orbit of the CM point $(z_1, z_2)$, the conjecture follows from \cite{GKZ87} for $r$ even.
More partial results are available when one averages over different Galois orbits \cite{Li18, BEY21} when $N = 1$.

Motivated by Conjecture \ref{conj:GZ}, the first and third author, together with S.\ Ehlen, considered its generalization to the setting of orthogonal Shimura varieties in \cite{BEY21}.
More precisely, let $\V$ be a rational quadratic space of signature $(n, 2)$ with $n \ge 1$, and $X_K$ be the Shimura variety associated to $\tH_\V := \GSpin(\V)$ and
a compact open subgroup $K \subset \tH_\V(\hat\Qb)$.
For a non-negative integer $r$ and a vector-valued harmonic Maass form $f$ of weight $1- n/2 -2r$, denote by $\Phi_{f}^r$ its regularized theta lift (see \cite{Bruinier02} or equation \eqref{eq:Phijint}).
This function is an eigenfunction of the Laplacian on $X_K$ and has a logarithmic  singularity along the special divisor $Z_f$ associated to $f$ (see \eqref{eq:Zf}).
We call it a \textit{higher Green function} on $X_K$, and say that it is \textit{principal}, resp.\ \textit{rational}, if $f$ is weakly holomorphic, resp.\ has rational principal part Fourier coefficients.
When $\V = M_2(\Qb)$ and $X_K = X_0(N)^2$, the function $\Phi_{f}^r$ becomes $G^{\Gamma_0(N)}_{r + 1, f}$ (see Corollary \ref{cor:1}).


For  a totally real field $F$ of degree $d$ and  an $F$-quadratic space $W = E$ with  $E/F$ a quadratic CM extension, suppose there is an isometric embedding
\begin{equation}
  \label{eq:isometry}
  W_\Qb := \mathrm{Res}_{F/\Qb}W \hookrightarrow \V,
\end{equation}
which in particular implies that $n+2 \ge 2d$.
Then we obtain a CM cycle $Z(W)$ on $X_K$ from a torus $T_W$ in $\tH_\V$ (see section \ref{subsec:CM} for details).
Note that $Z(W)$ is defined over $F$, and is the big CM cycle $Z(W, z_0^\pm)$ in \cite{BKY12}.
We denote $Z(W)_\Qb$
the union of the $F$-conjugates of $Z(W)$.
If $F$ is quadratic, we write $Z(W)_\Qb = Z(W) \cup Z(W)'$.
%

In \cite{Li23}, the second author studied the algebraicity of the \textit{difference} of a rational, principal $\Phi^r_f$ at two CM points in $Z(W)$, and was able to verify the analogue of Conjecture \ref{conj:GZ} in that setting.
This opens up the possibility of proving Conjecture \ref{conj:GZ} when one proves an algebraicity result for the averaged value $\Phi^r_f(Z(W))$.
In this paper, we complete this step by proving the following result complementary to \cite{Li23}.

\begin{theorem}[Algebraicity and Factorization]
  \label{thm:factor}
  Let $\Phi_{f}^r$ be a rational, principal higher Green function on $X_K$.
  Suppose that $E/\Qb$ is a biquadratic CM number field with the real quadratic field $F = \Qb(\sqrt{D})$, and $Z(W) \cap Z_f = \emptyset$.
   Then there exists a positive integer $\kappa$ and $a_1, a_2 \in F^\times$ such that
  \begin{equation}
    \label{eq:sumalg}
\Phi_{ f}^r(Z(W)) = \frac{1}{\kappa}\lp   \log \left|a_1\right| + \sqrt{D} \log\left|a_2\right| \rp.
\end{equation}
For any prime $\pf$ of $F$, the value
$  \kappa^{-1} \ord_\pf(a_j)$ is given in \eqref{eq:n2fac}.
When $n = 2$, we have $a_j = 1$ for $j \equiv r \bmod{2}$.
\end{theorem}

\begin{remark}
  \label{rmk:kappa}
  The denominator $\kappa$ appears as a consequence of our matching of sections (see Propositions \ref{prop:tI+} and \ref{prop:rational}), and only depends on $Z(W)$ and $r$ when $f$ has integral Fourier coefficients.
\end{remark}

\begin{remark}
  \label{rmk:r=0}
  Theorem \ref{thm:factor} also applies to the case $r = 0$ when $f$ has zero constant term, in which case
  $\Phi_f^0 = \Phi_f$
  is the  regularized Borcherds lift of $f$ and  we have $a_2 = 1$.
\end{remark}
    Combining Theorem \ref{thm:factor} with the main result in \cite{Li23}, we deduce the algebraicity of a rational, principal higher Green function at a single CM point when $E/\Qb$ is biquadratic, hence Conjecture \ref{conj:GZ} in particular.

\begin{theorem}
  \label{thm:main}
  In the setting of Theorem \ref{thm:factor}, there exists $\kappa \in \Nb$ and Galois equivariant maps $\alpha_1, \alpha_2: T_W(\hat\Qb) \to E^{\mathrm{ab}}$ such that
  $$
\Phi_{f}^r([z_0, h]) = \frac{1}{\kappa} \lp \log |\alpha_1(h)| + \sqrt{D} \log |\alpha_2(h)| \rp
$$
for all $[z_0, h] \in Z(W_{})$.
Furthermore for $n =2$, we can choose $\alpha_j(h) = 1$ for $j \equiv r \bmod{2}$, i.e.\ there exists a Galois-equivariant map $\alpha: T_W(\hat\Qb) \to E^{\mathrm{ab}}$ such that
  $$
\Phi^{r}_f([z_0, h]) = \frac{1}{\kappa\sqrt{D}^r }  \log |\alpha(h)|
$$
for all $h \in T_W(\hat\Qb)$.
In particular  Conjecture \ref{conj:GZ} is true.
\end{theorem}

\subsection{Comparison to Previous Works}
There has been an extensive literature on the CM-value of regularized theta lifts.
    When $r = 0, n = 2$ and $f$ is weakly holomorphic, the CM-value $\Phi_f(Z(W)_\Qb)$ was the subject of the classical work of Gross-Zagier on singular moduli \cite{GZ85}, and generalizations by the first and third author \cite{BY06}.
More generally for arbitrary $n$, totally real field $F$ and harmonic Maass form $f$, the value $\Phi_f(Z(W)_\Qb)$ is the archimedean contribution of the derivative of a Rankin-Selberg $L$-function involving the shadow $\xi(f)$ at $s = 0$ \cite{BY09, BKY12, AGHMP18}.

A crucial ingredient in these works is a real-analytic Hilbert Eisenstein series $E^{*}$ of parallel weight $1$ over $F$.
It is an \textit{incoherent Eisenstein series} in the sense of the Kudla program \cite{Kudla97}.
The arithmetic Siegel-Weil formula predicts that the Fourier coefficients of its derivative, $E^{*, \prime}$, are arithmetic degrees of special cycles \cite{HY11, HY12}, which are logarithms of rational numbers.

Using suitable weight 1 harmonic Maass forms in place of incoherent Eisenstein series, the first and third author, together with S.\ Ehlen, could prove the algebraicity result for higher Green function at a partially averaged CM cycle, and deduce the Gross-Zagier conjecture for $X_K = X_0(1)^2$ when the class group of one of the imaginary quadratic fields in $E$ is an elementary 2 group \cite[Theorem 1.2]{BEY21}.
However, the factorization of the ideal generated by the algebraic numbers are not explicitly given.




Our main result in Theorem \ref{thm:factor} goes far beyond these aforementioned works in an essential way by studying the regularized theta lifts at the partially averaged CM cycle $Z(W)$, which is in general only \textit{half} of $Z(W)_\Qb$ and a priori defined over $F$.
For $r = 0$ and $f$ weakly holomorphic, this means that $\Phi_f(Z(W))$ is the logarithm of a number in the real quadratic field $F$, and therefore cannot be related to the Fourier coefficients of incoherent Eisenstein series!

Furthermore, this partial average is quite different, yet more natural, than the one studied in \cite{BEY21}.
Instead of using the weight 1 harmonic Maass form loc.\ cit., which is an elliptic modular form,
we explicitly construct a Hilbert modular form $\tilde{\mathcal{I}}$, serving as a companion and substitute for the incoherent Eisenstein series, and obtain precise information concerning its Fourier coefficients.
This is the main innovation of the paper and allows us to prove the exact factorization formula for the ideal generated by the algebraic numbers in the spirit of \cite{GZ85}, which was not possible in \cite{BEY21}.
Most importantly, we are able to achieve this for \textit{arbitrary} open compact subgroup $K$, just as in \cite{Li23} for the difference of two CM-values, whereas the ingredients in \cite{BEY21} could only handle the level 1 case.
This enables us to prove Theorem \ref{thm:main} for arbitrary level $K$, which encompasses the case in Conjecture \ref{conj:GZ}.
In that sense, this paper is the complement to \cite{Li23}, both in results and methods, for biquadratic $E$.

Besides the analytic approach to Conjecture \ref{conj:GZ}, which originated from the work of Viazovska for $F = \Qb \oplus \Qb$ \cite{Via11}, there is also an algebraic approach in \cite{Zhang97, Mellit08}.
However, one must overcome serious obstacles to prove Theorem \ref{thm:main} via this approach.
For $F = \Qb \oplus \Qb$, one needs to assume in an essential way the non-degeneracy of the restriction of the  Gillet-Soul\'e height pairing, which is defined on Kuga-Sato varieties, to the subgroup of the Chow group spanned by CM cycles \cite[Theorem 5.2.2]{Zhang97}.
The non-degeneracy of this height pairing on a slightly larger subgroup is conjectured by Beilinson \cite{Bei87} and Bloch \cite{Bloch84} (see Conjecture 1.3.1 in \cite{Zhang97}).
For real quadratic $F$, one needs to find a substitute for the Kuga-Sato variety, construct canonical models, define suitable cycles and  arithmetic intersections such that the archimedean contribution is given by the CM-values of higher Green functions%
\footnote{An idea is to consider powers of the Kuga-Satake abelian scheme over an integral model of $X_K$, though the dimension of such an abelian scheme is $2^{n+1}$ and the fiber product becomes untractable quickly.}.
Assuming that the conjecture of Beilinson and Bloch holds in this case, one can then deduce the result in Theorem \ref{thm:main}.
For $n = 2$ and  concrete families of CM points, it is possible verify Conjecture \ref{conj:GZ} by explicit constructions of cycles and calculations (see \cite{Mellit08}).
In general, it is not clear at all how to construct suitable cycles, not to mention remove the non-degeneracy assumption.
On the other hand, it would be very interesting to see if Theorem \ref{thm:main}, which is proved via the analytic approach, can be used to prove the conjectural non-degeneracy when restricted to the above subgroup of the Chow group in \cite{Zhang97}.


\subsection{Proof Strategy}
\label{subsec:idea}
For simplicity, we focus on the case $n = 2$, from which the general case is not hard to derive (see section \ref{sec:ThmPfs} for details).
Applying the strategy in \cite{Kudla03} and the Rankin-Cohen operator, one can express $\Phi^r_f(Z(W)) + (-1)^r \Phi^r_f(Z(W)')$ as an $F$-linear combination of Fourier coefficients of the holomorphic part of $E^{*, \prime}$, which are logarithms of rational numbers.
This is a standard procedure involving the Siegel-Weil formula and Stokes' Theorem (see e.g.\ the proof of Theorem 3.5 in \cite{Li21a}).
A crucial property of the incoherent Eisenstein series is the following differential equation
(see \cite[Lemma 4.3]{BKY12})
\begin{equation}
  \label{eq:diffintro}
2(L_1 +L_2) E^{*, \prime}(g, 0, \Phi^{(1, 1)}) = E^*(g, 0, \Phi^{(1, -1)}) +E^*(g, 0, \Phi^{(-1, 1)}).
\end{equation}
Here $L_j$ are lowering operators in the $j$-th variable, and $\Phi^{(\e_1, \e_2)} = \Phi_f \otimes \Phi^{(\e_1, \e_2)}_\infty$ are Siegel-Weil sections in the degenerate principal series $I(0, \chi)$ with
 $\chi=\chi_{E/F}$ being the quadratic Hecke character of $F$ associated to $E/F$ (see section \ref{subsec:Eisenstein} for details).
In particular, $E^*(g, 0, \Phi^{(\e, -\e)})$ is a coherent Eisenstein series of weight $(\e, -\e)$ for $\e = \pm 1$.
To prove Theorem \ref{thm:factor}, it suffices to understand $\Phi^r_f(Z(W)) - (-1)^r \Phi^r_f(Z(W)')$, which means we need a substitute of $E^{*, \prime}$ on the left hand side of \eqref{eq:diffintro} such that the right hand side is $E^*(g, 0, \Phi^{(1, -1)}) - E^*(g, 0, \Phi^{(-1, 1)})$.

To obtain this minus sign, we apply the exceptional isogeny in \eqref{eq:isog} and view the coherent Eisenstein series as modular forms on the group $H_0 := \SO(V_0)$ for the quadratic space $V_0$  of signature $(2, 2)$ defined in section \ref{subsec:V}.
Since $E/\Qb$ is biquadratic, there is an odd character $\vc  = \vc_f \cdot \sgn$ of $[F^1]=F^1 \backslash \A_F^1$ such that $\chi = \vc \circ \Nm^-$ (see Remark \ref{rmk:biquad}).
By viewing $\vc$ as an automorphic form on $H_1 := \SO(V_1)$ for the quadratic space $V_1 = (F, \Nm)$, we can consider its theta lift
  following the  diagram
  \begin{equation}
    \label{eq:theta10}
    H_1 \stackrel{\theta_1}{\to} G \stackrel{\theta_0}{\to} H_0,
  \end{equation}
  where
  $G = \SL_2$.
  The first map lifts $\vc$ to a weight one holomorphic cusp form $\vartheta(g', \varphi^-_1, \vc)$ on $G$, which was first studied by Hecke.
    Here $\varphi_1^\pm$ is a Schwartz function on $V_1(\Ab)$ whose archimedean component $\varphi^\pm_{\infty}$ is the Schwartz function in $V_1(\Rb)$ defined in \eqref{eq:varphipm}.
    Then the second map lifts it to a coherent Eisenstein series, and is an instance of the Rallis tower property (\cite{Rallis})%
\footnote{By a change of integration order, we can also rewrite the map as
\begin{equation*}
  G \stackrel{\theta}{\to}
H
\stackrel{\text{Pullback}}{\longrightarrow}
H_0,
\end{equation*}
where $H = \SO(V)$ contains $H_0 \times H_1$.
}.
From this, $\theta_0 \circ \theta_1$ gives us the equation
\begin{equation}
  \label{eq:1}
  \begin{split}
 \Ic(g, \varphi^{(\pm 1, \mp 1)}, \vc)
&:= \int_{G(\Qb)) \backslash G(\Ab)}  \theta_0(g', g, \varphi_{0}^{(\pm 1, \mp 1)}) \vartheta(g', \varphi^-_1, \vc) dg'
\\
&  =       \frac{3}{\pi}
  E^*(g, 0, \Phi^{(\pm 1, \mp 1)}),
  \end{split}
\end{equation}
where $\theta_0$ is a theta kernel for the quadratic space $V_0$,
$\varphi^{(\pm 1, \mp 1)} = \varphi_0^{(\pm 1, \mp 1)} \otimes \varphi_1^-$  is a Schwartz function  on $V(\Ab)$ with $V := V_0 \oplus V_1$ and
$\Phi^{(\pm 1, \mp 1)} =  F_{\varphi, \vc}$ is the section defined in \eqref{eq:Fvarphi}.
Our first main result is Theorem \ref{thm:match}, which ensures that all coherent Eisenstein series can be realized as such lifts. This is reduced to the corresponding local problem and solved in section \ref{subsec:local-match-I}.

To construct $\tI$, we first modify the character $\vc$ to
the function $\tvc_C$ on $H_1(\Ab)$ defined in \eqref{eq:tvc}.
It is a preimage of $\vc$ under the first order invariant differential operator $t \frac{d}{dt}$ on $H_1(\Rb) \cong \Rb^\times$, and hence not a classical automorphic form on $H_1$.
We call its lift $\vartheta(g', \varphi^+_1, \tvc_C)$ to $G$ a \textit{deformed theta integral}, since the archimedean component of $\tvc_C$ is essentially $\log t$ and comes from the first term in the Laurent expansion of $t^s$ at $s = 0$.
This deformed theta integral was first studied in \cite{CL20}. 
It satisfies the following important property (see Theorem \ref{thm:modified})
\begin{equation}
  \label{eq:diff1}
L \vartheta(g', \varphi_1^+, \tvc_C) = \vartheta(g', \varphi_1^-, \vc) + \error.
\end{equation}
Here $L$ is the lowering operator on $G$, and $\error$ is the special value of the theta kernel on $V_1$.

We now define $\tI(g) := \Ic(g, \varphi^{(1, 1)}, \tvc_C)$ in \eqref{eq:tI} using the theta kernel $\theta_0(g', g, \varphi_0^{(1, 1)})$ with the archimedean component of $\varphi^{(1, 1)}_0$ being the Schwartz function $\varphi^{(1, 1)}_{0, \infty}$ defined in \eqref{eq:tv} (the integral is similar to (\ref{eq:1})).
A key observation  is that there is an identity between the actions of the universal enveloping algebras of $H_0$ and $G$, which gives in this special case  (see (\ref{eq:iotavarphi0}) and Lemma \ref{lemma:RC})
\begin{equation}
(L_1 +L_2) \theta_0(g', g, \varphi_0^{(1, 1)}) =  L \theta_0(g', g, \varphi_0^{(1, -1)})  -L \theta_0(g', g, \varphi_0^{(-1, 1)}),
\end{equation}
where $L_1, L_2$, resp.\ $L$, are differential operators
for the variable  $g \in H_0$, resp.\ $g' \in G$.
Putting these together, we see that $\tI$ satisfies the following property
(see the proof of Proposition \ref{prop:diffop} with $r = 0$ for details)
\begin{align*}
-  (L_1 +L_2) \tI(g)
  &= E^*(g, 0, \Phi^{(1, -1)}) - E^*(g, 0, \Phi^{(-1, 1)}) + \error'.
\end{align*}
Up to this the term $\error'$, which is a manifestation of the error term in \eqref{eq:diff1}, we have constructed the Hilbert modular form satisfying the  desired analogue of the differential equation \eqref{eq:diffintro}.

In addition to satisfying the differential equation, we still need to better understand the Fourier coefficients of $\tI$, and compare them to those of $E^{*, \prime}$.
This is done in section \ref{subsec:FEtI}, where we show that they are logarithms of algebraic numbers and give precise factorization information.
To achieve this, we introduce a new local section with an $s$-variable in \eqref{eq:Fvs} and match it with the standard section involving the $s$-variable up to an error of $O(s^m)$ for any positive integer $m$. This builds upon the results in section \ref{subsec:local-match-I} and is accomplished in Theorem \ref{thm:localmatchs}.
These new local sections are of independent interest, as they do not come from pullback of the standard section on $H \cong \SO(3, 3)$.
In fact, they do not even tensor together to form a global section with an $s$-variable.

Finally, we still need to handle the term arising from the error on the right hand side in \eqref{eq:diff1}.
This boils down to proving the rationality of a Millson theta lift, which is given in Proposition \ref{prop:millson}.
For this, we need the Fourier expansion of such a lift computed in \cite{ANS18}, and to choose the matching section with a suitable invariance property.
Proceeding essentially as in \cite{GKZ87} or \cite{BEY21}, with $E^{*, \prime}$ replaced by its sum with $\tI$, we complete the proof of Theorem \ref{thm:factor}.

}


\subsection{Outlook and Organization}
The factorization of the algebraic numbers appearing in the Fourier coefficients of $\tI$ are very closely related to the Fourier coefficients of $E^{*, \prime}$, which suggests that they should reflect the non-archimedean part of the arithmetic intersection between integral versions of $Z_f$ and $Z(W)$ defined over the ring of integers of $F$.
It would be very interesting to relate this arithmetic intersection to special values of derivatives of $L$-functions as in \cite{BKY12} by applying and refining the results in \cite{AGHMP18}.

It would be interesting to investigate the analogues of Theorems \ref{thm:factor} and \ref{thm:main} for other  CM, \'etale $\Qb$-algebras $E/\Qb$.
When $E/\Qb$ has degree 4, there are four cases
\begin{enumerate}
\item $E/\Qb$ is biquadratic,
\item $E/\Qb$ is a product of  imaginary quadratic fields,
 \item $E/\Qb$ is cyclic,
\item  $E/\Qb$ is a non-Galois, quartic extension.
\end{enumerate}
The CM points $Z(W)$ have a moduli interpretation as abelian surfaces with CM by the reflex CM algebra $E^\#$.
The present paper treats case (1).
In  cases (2) and (3), the reflex algebras $E^\#$ are quartic, abelian field extensions of $\Qb$, and the CM cycle $Z(W)$ is already defined over $\Qb$.
In the last case, $E^\#/\Qb$ is a quartic, non-Galois field, and $Z(W)$ is defined over a real quadratic field.
We plan to extend the ideas and techniques in this paper to prove the analogue of Theorem \ref{thm:main} in cases (2)-(4).
One difficulty that arises is that the quadratic space of signature $(3, 3)$ will have Witt rank less than 3, making it impossible to apply the Siegel-Weil formula to identify the theta integral with an Eisenstein series. Instead, one could try to add a twist to the theta integral (see \cite{Li16}), compute its Fourier expansion, and match it with that of an Eisenstein series.
When $E/\Qb$  is a field of degree greater than 4, the Hilbert Eisenstein series are over totally real fields of degree greater than 2, hence do not arise from theta integral of elliptic modular forms.
For such cases, one would need some new ideas.

In addition, there are other applications of these expected results.
For cases (2) and (3), by combining the analogue of Theorem \ref{thm:factor} and the idea in \cite{Li21a}, we hope to obtain non-existence result of genus 2 curves with CM Jacobian and having everywhere good reduction in certain families, generalizing the main result in \cite{HP17}.
In the last case, we expect a variation of our construction to lead to a proof of the factorization conjecture of CM-values of twisted Borcherds product in \cite{BY07}.

The paper is organized in the following way. Section \ref{sec:prelim} contains preliminaries. Most of these are standard, except for section \ref{subsec:modify}, which contains the adelic version of the results in \cite{CL20}.
Section \ref{sec:match} matches the coherent Eisenstein series with the Doi-Naganuma lift of Hecke's cusp form.
Section \ref{sec:DNMod} defines $\tI$ and studies its various properties.
Finally, we give the proofs of Theorems \ref{thm:factor} and \ref{thm:main} in the last section.

\noindent {\bf Acknowledgement}:
We thank Claudia Alfes, Chao Li, and Shaul Zemel for helpful comments on an earlier version.
We also thank Mingkuan Zhang and the anonymous referees for carefully reading through the draft and giving useful remarks to help clear confusions and improve the exposition.
J.B.\ and Y.L.\ were supported by the LOEWE research unit USAG, and by the Deutsche Forschungsgemeinschaft (DFG) through the Collaborative Research Centre TRR 326 ``Geometry and Arithmetic of Uniformized Structures'', project number 444845124.
T.Y.\ was partially supported by the Dorothy Gollmar Chair's Fund and Van Vleck research fund of the department of mathematics, UW-Madison.

\section{Preliminaries}
\label{sec:prelim}
In this section, we introduce some preliminary notions, most of which are standard from the literature. The only material not easily found in the literature are in sections \ref{sec:Hecke} and \ref{subsec:modify} concerning the weight one cusp forms of Hecke in the adelic language, which are translated from the results in \cite{CL20} in the classical language.

Let $\Nb$ denote the set of positive integers and $\Nb_0 := \Nb \cup \{0\}$.
For a number field $E$, let $\Ab_E$ be its ring of adeles, $\hat E$ the finite adeles, and $\Ab =\Ab_\Qb$ with $\psi = \psi_f \psi_\infty$ its usual additive character.
For an algebraic group $\mathrm{G}$ over $E$, denote $[\mathrm{G}] = \mathrm{G}(E)\backslash \mathrm{G}(\Ab_E)$.
As usual, let $G = \SL_2$ with standard Borel $B = MN \subset G$.
Denote also
$$
m(a) = \smat{a}{}{}{a^{-1}} \in M, n(b) := \smat{1}{b}{}{1} \in N,~
w = \smat{}{-1}{1}{} \in G,
$$
and
$$
T(R) = \{t(a) := \smat{a}{}{}{1}: a \in R^\times \}\subset \GL_2(R).
$$
Throughout the paper, $F$ will be a real quadratic field (unless stated otherwise).
Let $\prime \in \Gal(F/\Qb)$ be the non-trivial element.
It induces an automorphism of $\Ab_F, \Ab_F^\times$ and $F_p := F \otimes \Qb_p$ for each prime $p \le \infty$,
If $p$ is a finite prime that splits in $F$ (resp.\ is the infinite place), then $F$ has two embeddings into $\Qb_p$ (resp.\ $\Rb$), and $F_p$ is a 2-dimensional vector space over $\Qb_p$ (resp.\ $\Rb$).
For $\lambda \in F$, let $\lambda_1, \lambda_2$ denote the images under those embeddings.
We will also sometimes use $\lambda$ to represent the pair $(\lambda_1, \lambda_2)$ in $\Qb_p^2$ (resp.\ $\Rb^2$), and $\lambda'$ would represent $(\lambda_2, \lambda_1)$.
We have the incomplete Gamma function
$$
\Gamma(s, x) = \int^\infty_x t^{s-1}e^{-t} dt.
$$

\subsection{Differential Operators}
For a real-analytic function $f$ on $G(\Rb)$, the Lie algebra $\slf_2(\Cb)$ acts via
\begin{equation}
  \label{eq:Lieact}
  A(f)(g) := \partial_t f(g e^{tA}) \mid_{t = 0},~ A \in \slf_2(\Cb).
\end{equation}
We define the raising and lowering operator
\begin{equation}
  \label{eq:RL}
  R := \frac{1}{2} \pmat{1}{i}{i}{-1},~ L := \frac{1}{2} \pmat{1}{-i}{-i}{-1} .
\end{equation}
If $f$ is right $K_\infty$-equivariant of weight $k$, then we have
$$
\sqrt{v}^{-(k+2)} R(f)(g_\tau) = R_{\tau, k} (\sqrt{v}^{-k} f(g_\tau)),~
\sqrt{v}^{-(k-2)} L(f)(g_\tau) = L_{\tau, k} (\sqrt{v}^{-k} f(g_\tau)),
$$
where $R_{\tau, k}$ and $L_{\tau, k}$ are the usual raising and lowering operators given by
\begin{equation}
  \label{eq:RLtau}
    R_{\tau, k} := 2i \partial_\tau + \frac{k}{v},~
L_{\tau, k} := -2iv^2 \partial_{\overline{\tau}}.
\end{equation}
We say that $f$ is holomorphic, resp.\ anti-holomorphic, if $L(f) = 0$, resp.\ $R(f) = 0$.

For $ r\in \Nb_0$  and $k_1, k_2 \in \half\Zb$, define
\begin{equation}
  \label{eq:Qs}
  \begin{split}
      Q_{r, (k_1, k_2)}(X, Y) &:= \sum_{s = 0}^r \binom{r + k_1 - 1}{s}  \binom{r + k_2 - 1}{r - s} X^{r-s}(-Y)^s,\\
  \tQ_r(X, Y) &:= \frac{Q_{r, (1, 1)}(X, Y) (X + Y)}{X + (-1)^r Y}
  \end{split}
\end{equation}
in $\Qb[X, Y]$.
We omit $(k_1, k_2)$ from the notation when it is $(1, 1)$, in which case
$$
Q_r(X, Y) = (X + Y)^r P_r\lp \frac{ X - Y}{X+Y} \rp,
$$
with $P_r(x)$ the $r$-th Legendre polynomial given explicitly by
\begin{equation}
  \label{eq:Pr}
  P_r(x) = 2^{-r} \sum_{s = 0}^r \binom{r}{s}^2(x-1)^{r-s} (x+1)^s
  = (-1)^{r_0} \sum_{s = 0}^{r_0} \binom{r_0 - r - 1/2}{r_0 - s} \binom{r - r_0 - 1/2}{s} x^{r-2s},
\end{equation}
where $r_0 := \lfloor r/2 \rfloor$. The second identity comes from (3.133) on page 38 of \cite{Gould72} and direct calculation. We thank Zhiwei Sun for pointing us to this reference.

From the differential equation satisfied by $P_r$, we have
\begin{equation}
  \label{eq:Qdiff}
  (  \partial_X   \partial_Y)(Q_r(X, Y)(X+Y))
  = (r+1)(\partial_X +   \partial_Y)Q_r(X, Y).
\end{equation}
For $A \in \slf_2(\Cb)$, denote
\begin{equation}
  \label{eq:Aj}
A_1 = (A, 0),~ A_2 = (0, A)
\end{equation}
in $\slf_2(\Cb)^2$.
Then we have two operators
\begin{equation}
   \label{eq:RC}
   \RC_{r, (k_1, k_2)} := (-4\pi)^{-r} Q_{r, k_1, k_2}(R_1, R_2),~
   \tRC_r := (-4\pi)^{-r} \tQ_r(R_1, R_2)
 \end{equation}
 on  real-analytic functions on $G(\Rb)^2$.
 If $f: G(\Rb)^2 \to \Cb$ is holomorphic and right $K_\infty^2$-equivariant of weight $(k_1, k_2)$, then
 $\RC_{r, (k_1, k_2)}(f)^\Delta: G(\Rb) \to \Cb$ is holomorphic and  right $K_\infty$-equivariant of weight $k_1 + k_2 + 2r$, and the operator $\RC$ is usually called the Rankin-Cohen operator.
 Here $f^\Delta(g) := f(g^\Delta) = f(g, g)$ is the restriction of $f$ to the diagonal $G(\Rb) \subset G(\Rb)^2$.
 In fact, we have (see \cite[Proposition 19]{123})
 $$
 \RC_{r, (k_1, k_2)}(f)^\Delta(g_z) = (2\pi i)^{-r} \lp Q_r(\partial_{z_1}, \partial_{z_2}) f(g_{z_1}, g_{z_2})\rp \mid_{z_1 = z_2 = z}.
 $$
  For example, if $f(g_{z_1}, g_{z_2}) = \ebf(m_1 z_1 + m_2 z_2)$, then
 \begin{equation}
   \label{eq:RCexp}
   \RC_{r, (k_1, k_2)}(f)^\Delta(g_{z}) =  Q_{r, (k_1, k_2)}(m_1, m_2) \ebf((m_1 + m_2)z).
 \end{equation}
 From Lemma 2.2 in \cite{Li23}, we know that there are unique constants $c^{(r; k_1, k_2)}_{\ell} \in \Qb$ such that
 \begin{equation}
   \label{eq:Rtf}
   (4\pi)^{-r} (R_{1}^r f)^\Delta = \sum_{\ell = 0}^r c^{(r; k_1, k_2)}_\ell (4\pi)^{-r+\ell} R^{r-\ell}\RC_{\ell, (k_1, k_2)}(f)^\Delta
 \end{equation}
 whenever $k_1 + k_2 + 2r < 2$.

\subsection{Quadratic Space associated to Real Quadratic Field}

  \label{subsec:Va}
  Let $F = \Qb(\sqrt{D})$ be a real quadratic field, which becomes a $\Qb$-quadratic space of signature $(1, 1)$ with respect to the quadratic form
$$
Q_\a (\lambda):= \a  \Nm(\lambda)
,~ \lambda \in F
$$
for any $\a  \in \Qb^\times$.
We denote this quadratic space
by $V_\a $ and identify
\begin{equation}
  \label{eq:identify}
\begin{split}
\iota_\a:  V_\a (\Rb) = F \otimes_\Qb \Rb &\cong \Rb^2  \\
(\lambda_1, \lambda_2)  &\mapsto (\lambda_1, \sgn(\a )\lambda_2)\sqrt{|\a |}.
\end{split}
\end{equation}
This is an isometry, where the quadratic form on $\Rb^2$ is given by $Q((x_1, x_2)) = x_1x_2$.
The special orthogonal group $H_\a  := \mathrm{SO}({V_\a })$ satisfies
$$
H_\a (\Qb) \cong F^1 ( \cong F^\times/\Qb^\times),
$$
where $ F^1 := \{\mu \in F: \Nm(\mu) = 1\}$ acts
on $V$ via multiplication. 
Furthermore, we identify
$$
H_\a (\Rb) \cong \Rb^\times,
$$
where $t \in \Rb^\times$ acts on
$V_\a (\Rb) = F \otimes_\Qb \Rb \cong \Rb^2$
via $(x_1, x_2) \mapsto (t x_1, t^{-1}x_2)$.
Note that the invariant measure on $H(\Rb) \cong \Rb^\times$ is $\frac{dt}{|t|}$.
We have a character $\sgn : H_\a (\Rb) \cong \Rb^\times \to \{\pm 1\}$.
Denote
$$
H_\a (\Rb)^+ := \ker(\sgn) \cong \Rb_{> 0}
$$
its kernel, which is the connected component of $H_\a(\Rb)$ containing the identity.
We also denote
\begin{equation}
  \label{eq:HQ+}
  H_\a (\Qb)^+ := H_\a (\Rb)^+ \cap H_\a (\Qb)
\end{equation}
which is an index 2 subgroup of $H_\a (\Qb)$.

\begin{remark}
  \label{rmk:biquad}
  Let $\chi = \chi_{E/F}$ be a Hecke character associated to a quadratic extension $E/F$. Suppose $E/\Qb$ is biquadratic. Then $\chi\mid_{\Ab^\times}$ is trivial
  and $\chi$ factors through the map $\Nm^-: \Ab_F^\times \to H_\a (\Ab)$, i.e.\ there exists $\vc = \vc_{E/F}: H_\a (\Ab) \to \{\pm 1\}$ such that
\begin{equation}
  \label{eq:vc}
  \vc \circ \Nm^- = \chi.
\end{equation}
Note that $\vc$ is odd if and only if $E$ is totally imaginary.
We also denote the compact subgroup
\begin{equation}
  \label{eq:Kvc}
K_\vc := H_\a(\hat\Zb) \cap \mathrm{ker}(\vc).
\end{equation}
Note that $H_\a (\Qb)\backslash H_\a (\hat \Qb)/K_\vc$ is a finite set.

\end{remark}

\subsection{The Weil Representation and Theta Functions}
\label{subsec:WeilRep}

Let $(V, Q)$ be a rational quadratic space of signature $(p, q)$, and $H_V := \SO(V)$.
For a subfield $E \subset \Cb$, we denote $\Sc(\hat V; E)$, resp.\ $\Sc(V_p; E)$, to denote the space of Schwartz functions on $\hat V := V(\hat\Qb)$, resp.\ $V_p := V(\Qb_p)$, with values in $E$, which is a $E$-vector space. We omit $E$ from the notation if it is $\Qb$.
On the other hand, we write $\Sc(V(\Rb))$ and $\Sc(V(\Ab))$ for the space of Schwartz function on $V(\Rb)$ and $V(\Ab)$ valued in $\Cb$ respectively.

For a lattice\footnote{Lattices will be even and integral throughout the paper.} $L \subset V$, we denote $L^\vee \subset V$ its dual lattice, $\hat L := L \otimes \hat\Zb$, $\hat L^\vee := L^\vee \otimes \hat\Zb$ and
\begin{equation}
  \label{eq:Lmmu}
  L_{m, \mu} := \{\lambda \in L+ \mu: Q(\lambda) = m\}.
\end{equation}
for $m \in \Qb, \mu \in L^\vee/L$.
The finite dimensional $E$-subspace
\begin{equation}
  \label{eq:ScL}
  \Sc(L; E) :=  \{\phi \in \Sc(\hat V; E): \phi \text{ is $\hat L$-invariant  with support on } \hat L^\vee\} \subset \Sc(\hat V; E),
\end{equation}
 is spanned by $\{\phi_{L + \mu} : \mu \in \hat L^\vee/\hat L \cong L^\vee/L\}$ with
\begin{equation}
  \label{eq:phimu}
\phi_{L + \mu} := \cha(\hat L + \mu) \in \Sc(\hat V).
\end{equation}
For full sublattices $M \subset L \subset V$, it is clear that
\begin{equation}
  \label{eq:LM}
  \Sc(L; E) \subset \Sc(M; E) \subset \Sc(\hat V; E).
\end{equation}
As above, we also denote $\Sc(L) := \Sc(L; \Qb)$.
Furthermore, since
$$
\Sc(\hat V) = \bigcup_{L \subset V \text{ lattice}} \Sc(L),
$$
we have $\Sc(\hat V; E) = \Sc(\hat V) \otimes_\Qb E$ for any subfield $E \subset \Cb$.

Suppose $V = V_1 \oplus V_2$ is a decomposition of rational quadratic spaces.
For any lattice $L_i \subset V_i$, we have $\Sc(L_1 \oplus L_2; E) = \Sc(L_1; E) \otimes \Sc(L_2; E) \subset \Sc(\hat V_1; E) \otimes \Sc(\hat V_2; E)$ via the natural restriction map.
This also gives us
\begin{equation}
  \label{eq:Ssum}
  \Sc(\hat V; E) =
  \Sc(\hat V_1; E) \otimes \Sc(\hat V_2; E)
  = \bigoplus_{L_1 \subset V_1,~ L_2 \subset V_2 \text{ lattices}} \Sc(L_1; E) \otimes \Sc(L_2; E).
\end{equation}
Combining with equation \eqref{eq:LM}, we see that for any given $\phi \in \Sc(\hat V; E)$, we can find a lattice $L = L_1 \oplus L_2 \subset V$ such that $L_i \subset V_i$  and
\begin{equation}
  \label{eq:phiL}
  \phi \in \Sc(L; E) = \Sc(L_1; E) \otimes \Sc(L_2; E).
\end{equation}

Let $\tilde G(\Ab) := \mathrm{Mp}_2(\Ab)$ be the metaplectic cover of $G(\Ab)$.
The group  $\tilde G(\Ab) \times H_V(\Ab)$ acts on $\Sc(V(\Ab))$ via the Weil representation $\omega = \omega_{V, \psi}$ (see \cite[section 5]{Kudla94} for explicit formula).
For each prime $p \le \infty$, we also have the local Weil representation $\omega_p$ of $G(\Qb_p)$ acting on $\Sc(V(\Qb_p); \Cb)$.
Then $\omega_f := \otimes_{p < \infty} \omega_p$ gives a representation of $G(\hat \Qb)$ on $\Sc(\hat V; \Cb)$.

For any lattice $L \subset V$, the subspace $\Sc(L; \Cb)$ defined in \eqref{eq:ScL} is a $K_f$-invariant subspace with $K_f := G(\hat\Zb)$.
It has a unitary pairing $\langle, \rangle$ with the vector space $\Cb[L^\vee/L] := \oplus_{\mu \in L^\vee/L} \Cb \ef_\mu$ given by
\begin{equation}
  \label{eq:pair1}
  \langle \ef_{\mu} , \phi_{\mu'} \rangle :=  \langle \ef_{\mu} , \phi_{\mu'} \rangle_L :=
\begin{cases}
  1,& \mu = \mu',\\
  0,&\text{ otherwise.}
\end{cases}
\end{equation}
More generally if $L = L_1 \oplus L_2$, we have $L^\vee/L = L_1^\vee/L_1 \oplus L_2^\vee / L_2$ and $\Cb[L^\vee/L] \cong \Cb[L_1^\vee/L_1] \otimes \Cb[L_2^\vee/L_2]$.
Therefore, we can extend the pairing above to
\begin{equation}
  \label{eq:pair2}
  \langle \cdot,  \cdot\rangle_{L_2} : \Cb[L^\vee/L] \times \Sc(L_2; \Cb) \to \Cb[L_1^\vee/L_1],~
(    \vf_1 \otimes \vf_2 ,\phi) \mapsto    \langle\vf_2 ,\phi \rangle_{}\vf_1
\end{equation}
with $\vf_i \in \Cb[L_i^\vee/L_i]$ and $\phi \in \Sc(L_2; \Cb)$.

With respect to the perfect pairing in \eqref{eq:pair1}, the unitary dual of $\omega_f$ is the representation $\rho_L$ on $\Cb[L^\vee/L]$ given by
\begin{equation}
  \label{eq:rhoL}
  \begin{split}
    \rho_L(n(1))(\ef_\mu) &:= \ebf(Q(\mu)) \ef_\mu, \\
\rho_L(w)(\ef_\mu) &:= \frac{\ebf(-(p-q)/8)}{\sqrt{|L^\vee/L|}} \sum_{\mu \in L^\vee/L} \ebf(-(\mu, \mu')) \ef_{\mu'},
  \end{split}
\end{equation}
where  $(p, q)$ is the signature of  $V(\Rb)$.
This is the Weil representation on finite quadratic modules used by Borcherds in \cite{Borcherds98}.
If we identify $\Sc(L; \Cb)$ and $\Cb[L^\vee/L]$ with $\Cb^{|L^\vee/L|}$ via the bases $\{\phi_\mu: \mu \in \hat L^\vee/\hat L\}$ and $\{\ef_\mu: \mu \in  L^\vee/ L\}$ respectively, then $\omega_f = \overline{\rho_L} = \rho_L^{-1}$.
For full sublattices $M \subset L$, the following trace map
\begin{equation}
  \label{eq:trace}
  \tr^L_M: \Cb[L^\vee/L] \to \Cb[M^\vee/M],~ \ef_\mu \mapsto \frac{1}{[L:M]}\sum_{h \in M^\vee/M,~ h \equiv \mu \bmod{L}} \ef_h
\end{equation}
intertwines the Weil representation and is compatible with the inclusion in \eqref{eq:LM} in the sense that
\begin{equation}
  \label{eq:CLM}
\langle \tr^L_M(\vf), \phi \rangle_M = \langle \ef, \phi \rangle_L
\end{equation}
for any $\vf \in \Cb[L^\vee/L], \phi \in \Sc(L; \Cb)$.

The following result will be very useful for us later.
\begin{lemma}
  \label{lemma:omega-p-def}
For any prime $p$, the local Weil representation $\omega_p$ descends to $\Sc(V(\Qb_p); \Qip)$ with
\begin{equation}
  \label{eq:Qinfp}
  \Qip := \bigcup_{n \ge 1} \Qb(\zeta_{p^n}) \subset \Qb^{\mathrm{ab}}
\end{equation}
the maximal abelian extension of $\Qb$ ramified only at $p$.
\end{lemma}
\begin{proof}
  Via $L'/L = \oplus_p L'_p/L_p$ with $L_p := L \otimes \Zb_p$, we can write $\rho_L = \otimes_p \rho_{p}$ with $\rho_p$ the Weil representation associated to the finite quadratic module $L'_p/L_p$, and identify $\omega_p = \rho_{p}^{-1}$.
  It is well-known that any finite quadratic module can be written in the form $M'/M$ for an even integral lattice $M$ \cite{Ni79}.
  Therefore it suffices to prove the claim with $\omega_p$ replaced by $\rho_M$ for an even integral lattice $M$ with quadratic form valued in $\Zb[1/p]$.
  This follows then directly from formula \eqref{eq:rhoL} and Milgram's formula \cite[Corollary 4.2]{Borcherds98}.
\end{proof}

As usual, we let $H_{k, L}(\Gamma)$  denote the space of harmonic Maass forms valued in $\Cb[L^\vee/L]$ of weight $k \in \frac{1}{2}\Zb$ and representation $\rho_L$ on a congruence subgroup $\Gamma \subset \SL_2(\Zb)$ (see \cite[section 3]{BF04}).
It contains the subspace $M^!_{k, L}(\Gamma)$ of vector-valued weakly holomorphic modular forms.
Post-composing with $\tr^L_M$ in \eqref{eq:trace} induces a map
 $H_{k, L}(\Gamma) \to H_{k, M}(\Gamma)$,
 which preserves holomorphicity and rationality of holomorphic part Fourier coefficients.
 If $L^\vee/L$ is trivial (resp.\ $\Gamma = \SL_2(\Zb)$), then we drop $L$ (resp.\ $\Gamma$) from the above notations.
 Furthermore, we let
 \begin{equation}
   \label{eq:M!!}
   M^{!, \infty}_{k}(\Gamma):= \{f  \in M^!_k(\Gamma)  : f \text{
     is holomorphic away from the cusp }
   \infty \}.
 \end{equation}
 and denote for $f(\tau) = \sum_{m \in \Qb,~ \mu \in L^\vee/L} c(m, \mu) q^m \ef_\mu \in M^!_{k, L}$
 \begin{equation}
   \label{eq:prin}
  \prin(f) := \sum_{m \in \Qb_{< 0},~ \mu \in L^\vee/L} c(m, \mu) q^m \ef_\mu
 \end{equation}
the principal part of $f$.

In \cite[Theorem 4.3]{McGraw03}, McGraw extended
the representation $\rho_L$ to the metaplectic cover of $\GL_2$.
To simplify the notation, we recall it here for lattices with even rank, in which case this extension factors through $\GL_2$.
Using the short exact sequence
$$
1 \to \SL_2 \to \GL_2 \stackrel{\det}{\to} \Gm \to 1,$$
we can identify $\GL_2 \cong \SL_2 \rtimes T$.
Then $\omega_f$  extends to a $\Qb$-linear action of $\GL_2(\hat\Qb) = \GL_2(\Qb)\GL_2(\hat\Zb)$ on $\Sc(\hat V; \Qab)$ via
\begin{equation}
  \label{eq:weilGL2}
  (\omega_f(g, t(a))\phi)(x) := (\omega_f(g)\sigma_a(\phi))(x) =
\sigma_a(  (\omega_f(t(a)^{-1}gt(a))(\phi))(x))
\end{equation}
for $ g \in \SL_2(\hat\Qb), a \in \hat\Qb^\times, \phi \in \Sc(\hat V; \Qb^{\ab})$,
where $\sigma_a \in  \Gal(\Qb^\ab/\Qb)$ satisfies $\sigma_a(\psi_f(a')) = \psi_f(aa')$ for all $a , a' \in \hat\Qb^\times$ and acts on $\Sc(\hat V; \Qab)$ via its action on $\Qab$.
This gives us
\begin{equation}
  \label{eq:Galinv}
  \Sc(\hat V) = \Sc(\hat V; \Qab)^{T(\hat \Zb)}.
\end{equation}
For each $p < \infty$, this gives an extension of $\omega_p$ to a $\Qb$-linear action of $\GL_2(\Qb_p)$ on $\Sc(V_p; \Qip)$, which satisfies
\begin{equation}
  \label{eq:Galinvp}
  \Sc(V_p) = \Sc(V_p; \Qip)^{T(\Zb_p)}.
\end{equation}
For $\varphi \in \Sc(V(\Ab))$, we have the theta function
\begin{equation}
  \label{eq:theta}
  \theta_V(g, h, \varphi) := \sum_{x \in V(\Qb)} (\omega(g) \varphi)(h^{-1} x)
\end{equation}
for $(g, h) \in (G \times H_V)(\Ab)$ as usual.
For a lattice $L \subset V$, we also denote
\begin{equation}
  \label{eq:ThetaL}
  \Theta_L(\tau, h) := v^{-(p-q)/4}\sum_{\mu \in L^\vee/L} \theta_V(g_\tau, h, \phi_\mu\phi_\infty) \phi_\mu
\end{equation}
the vector-valued theta function with $\phi_\infty$ the Gaussian.

\subsection{CM Points and Higher Green Functions}
\label{subsec:CM}
We follow \cite{BKY12} and \cite{BEY21} to recall CM points and higher Green functions.
Let $(\V, Q)$ be a rational quadratic space of signature $(n, 2)$, and $\tH = \tH_\V := \mathrm{GSpin}(\V)$.
For an open compact subgroup $K \subset \tH(\hat\Qb)$, we have the associated  Shimura  variety $X_K$, whose $\Cb$-points are given by
$$
X_K(\Cb) = \tH(\Qb) \backslash(  \Db\times \tH(\hat\Qb) / K).
$$
Here $\Db = \Db^+ \sqcup \Db^-$ is the symmetric space associated to $\V(\Rb)$. For $m \in\Qb$ and $\varphi \in \Sc(\hat\V; \Cb)$, one can define the special divisor $Z(m, \varphi)$ on $X_K$ (see e.g.\ \cite[section 2]{BEY21}).

The CM points on $X_K$ can be described as follows.
For a totally real field $F$ of degree $d$ with real embeddings $\sigma_j, 1 \le j \le d$,  denote $\alpha_j := \sigma_j(\alpha)$ for $\alpha \in F$.
Suppose $\alpha_{j_0} < 0$ for some $j_0$ and $\alpha_j > 0$ when $j \neq j_0$.
Then a CM quadratic extension $E/F$ becomes an $F$-quadratic space $W = W_\alpha = E$ with respect to the quadratic form $Q_\alpha := \alpha \Nm_{E/F}$.
Suppose there is isometric embedding as in \eqref{eq:isometry}.
Then the subspace $W \otimes_{F, \sigma_{j_0}} \Rb \subset \V(\Rb)$ is a negative 2-plane, and determines a point $z_0^\pm \in \Db^\pm$ with a choice of orientation. For convenience, we denote
\begin{equation}
  \label{eq:z0}
z_0 := z_0^+.
\end{equation}
The group $\mathrm{Res}_{F/\Qb}(\SO(W))$ is contained in $\SO(\V)$, whose preimage in $\tH_\V$ is a torus denoted by $T_W$.
Note that $T_W(\Qb) = E^\times/F^1$.
We denote the CM cycle on $X_K$ associated to $T_W$ by
\begin{equation}
  \label{eq:ZW}
Z(W) := T_W(\Qb) \backslash (\{z_0^\pm\} \times T_W(\hat\Qb)/K_W) \subset X_K
\end{equation}
with $K_W := K \cap T_W(\hat\Qb)$.
It is defined over $F$, and its Galois conjugates are the CM cycles $Z(W(j))$ with $1 \le j \le d$, where $W(j)$ is the neighborhood $F$-quadratic spaces at $\sigma_j$ of admissible incoherent $\Ab_F$-quadratic space $\Wb$ associated to $W$  (see \cite{BY11, BKY12} for details).
Note that $W(j) = W_{\alpha(j)}$ for some $\alpha(j) \in F^\times$ and $\alpha(j_0) = \alpha$.
For totally positive $t \in F$, we define the ``Diff'' set
\begin{equation}
  \label{eq:diff}
  \mathrm{Diff}(W, t) := \{\pf: W_\pf \text{ does not represent }t\}
\end{equation}
following \cite{Kudla97}. Note that this set is finite and odd (see \cite[Proposition 2.7]{YY19}).

When $F$ is real quadratic (i.e.\ $d = 2$), for $\alpha \in F^\times$ with $\Nm(\alpha) < 0$, we set
  \begin{equation}
    \label{eq:alv}
   \alv := \alpha(2) \in F^\times.
  \end{equation}
  Then $(\alv)^\vee = \alpha$.
  Note that $\alv$ is not necessarily the Galois conjugate of $\alpha$!

Denote
\begin{equation}
  \label{eq:rho0}
  \sigma_0 := \frac{n}{4} - \frac{1}{2}.
\end{equation}
Let $L \subset \V$ be an even integral lattice such that $K$ stabilizes $\hat L$.
For $\mu \in L^\vee/L$ and $m \in \Zb + Q(\mu)$, we write $Z(m, \mu) := Z(m, \phi_\mu)$.
The automorphic Green function on $X_K \backslash Z(m, \mu)$ is defined by
\begin{equation}
  \label{eq:Phimmu}
  \Phi_{m, \mu}(z, h, s) := 2 \frac{\Gamma(s + \sigma_0 )}{\Gamma(2s)}
\sum_{ \lambda \in h(L_{m, \mu})}
\lp \frac{m}{Q(\lambda_{z^\perp})} \rp^{s + \sigma_0}
F\lp s + \sigma_0 , s - \sigma_0, 2s; \frac{m}{Q(\lambda_{z^\perp})} \rp,
\end{equation}
for $\mathrm{Re}(s) > \sigma_0 + 1$, where $F(a, b, c; z)$ is the Gauss hypergeometric function \cite[Chapter 15]{AS64}.
At $Z(m, \mu)$, the function $\Phi_{m, \mu}$ has logarithmic singularity.

At $s = \sigma_0 + 1 + r$ with $r \in \Nb$, the function $\Phi_{m, \mu}(z, h, s)$ is called a \textit{higher Green function}.
For a harmonic Maass form $f = \sum_{m, \mu} c(m, \mu) q^{-m} \phi_\mu + O(1)  \in H_{k - 2r, L}$ with $k:= -2\sigma_0$, define
\begin{equation}
  \label{eq:Phijf}
  \Phi_f^r(z, h) := r! \sum_{m > 0,~ \mu \in L'/L} c(m, \mu) m^r \Phi_{m, \mu} (z, h, \sigma_0 + 1 + r)
\end{equation}
to be the associated higher Green function.
Following from the work of Borcherds \cite{Borcherds98} and generalization by Bruinier \cite{Bruinier02} (also see \cite[Proposition 4.7]{BEY21}), the function $\Phi_f^r$ has the following integral representation
\begin{equation}
  \label{eq:Phijint}
  \begin{split}
    \Phi_f^r(z, h)
    &=
(4\pi)^{-r} \lim_{T \to \infty}  \int_{\Fc_T}^{} \langle R^r_\tau f(\tau),  \overline{\Theta_L(\tau, z, h)} \rangle d \mu(\tau)\\
&=  (-4\pi)^{-r} \lim_{T \to \infty}  \int_{\Fc_T}^{} \langle f(\tau), \overline{R^r_\tau \Theta_L(\tau, z, h)} \rangle d \mu(\tau),
  \end{split}
\end{equation}
where $\Fc_T$ is the truncated fundamental domain of $\SL_2(\Zb) \backslash \Hb$ at height $T> 1$ and $d\mu$ is the invariant measure given in \eqref{eq:dg}
It has logarithmic singularity along the special divisor
\begin{equation}
  \label{eq:Zf}
  Z_f := \sum_{m > 0,~ \mu \in L^\vee/L} c(-m, \mu) Z(m, \mu)
\end{equation}
on $X_K$.
Note that $[z_0, h] \in Z(W) \cap Z_f$ if and only if
\begin{equation}
  \label{eq:singint}
h(L_{m, \mu}) \cap z_0^\perp \neq \emptyset
\end{equation}
for some $m, \mu$ with $c(-m, \mu) \neq 0$.
\subsection{Product of Modular Curves as a Shimura Variety}
  \label{subsec:Gamma0N}
  We follow and slightly modify \cite[section 3]{YY19} to express $X_0(N) \times X_0(N)$ as $\mathrm{O}(2, 2)$ orthogonal Shimura variety.
    Consider $(\V, Q) = (M_2(\Qb), \det)$, and the lattice
  $$
L := \left\{ \pmat{a}{b}{Nc}{d}: a, b, c, d \in \Zb \right\} \subset \V
$$
for any $N \in \Nb$.
Then the dual lattice $L^\vee$ is given by
$$
L^\vee := \left\{ \pmat{a}{b/N}{c}{d}: a, b, c, d \in \Zb \right\} \subset \V
$$
and $L^\vee/L \cong (\Zb/N\Zb)^2$ is isomorphic to that of a scaled hyperbolic plane.

For $g_j \in \SL_2(\Qb)$ and $\Lambda \in \V(\Qb)$, the map
\begin{equation}
  \label{eq:SL22act}
\Lambda \mapsto g_1 \Lambda g_2^{-1}
\end{equation}
gives $\SL_2 \times \SL_2 \cong \Spin(\V)$
and identifies $H_\V$ as a subgroup of $\GL_2 \times \GL_2$ \cite[section 3.1]{YY19}.
Let $K(N) := K(\Gamma_0(N)) \subset \GL_2(\hat\Zb)$ be the open compact subgroup in \cite[section 3.1]{YY19} and $K := (K(N) \times K(N)) \cap H_\V(\hat \Qb)$.
Then the map
\begin{equation}
  \label{eq:w}
  w: \Hb^2 \to \Db^+,~ (z_1, z_2) \mapsto \Rb \Re\smat{z_1}{-z_1z_2}{1}{-z_2} + \Rb \Im\smat{z_1}{-z_1z_2}{1}{-z_2}
\end{equation}
induces an isomorphism
\begin{equation}
  \label{eq:XKisom}
  X_0(N) \times X_0(N) \cong X_K,~ (z_1, z_2) \mapsto [w(z_1, z_2), 1]
\end{equation}
with $X_K$ the Shimura variety for $H_\V$.

Under the map \eqref{eq:SL22act}, the inverse images of the discriminant kernel $\Gamma_L \subset \SO(L)$ are
$$
\Gamma^\Delta_0(N) := \left\{
(g_1, g_2) \in \Gamma_0(N)^2: g_1 g_2 \in \Gamma_1(N)
  \right\} \subset \Gamma_0(N)^2,
  $$
  which contains $ \Gamma_1(N) \times \Gamma_1(N)$ and is a normal subgroup of $\Gamma_0(N)^2$ satisfying
  $$
  \Gamma_0(N)^2 / \Gamma^\Delta_0(N) \cong (\Zb/N\Zb)^\times,~
  (\smat{a_1}{*}{*}{*}, \smat{a_2}{*}{*}{*}) \mapsto a_1a_2^{} \bmod{N}.
  $$
  The group $\SO(L^\vee/L) := \SO(L)/\Gamma_L \cong \Gamma_0(N)^2/\Gamma^\Delta_0(N) \cong (\Zb/N\Zb)^\times$ acts on $L^\vee/L \cong (\Zb/N\Zb)^2$ via
  $$
\alpha \cdot (b, c) :=  (\alpha b, \alpha^{-1} c),
  $$
  and the induced linear map on $\Cb[L^\vee/L]$ intertwines the Weil representation $\rho_L$.

  Now given $f \in M^!_{k}(\Gamma_0(N))$ for $k \in 2\Zb$, we can lift it to a vector-valued modular form in $M^!_{k, \rho_L}$ via the following map
  \begin{equation}
    \label{eq:vv}
    \vv:  M^!_{k}(\Gamma_0(N)) \to  M^!_{k, \rho_L},~
    f \mapsto \sum_{\gamma \in \SL_2(\Zb)/\Gamma_0(N)} (f \mid_{k} \gamma ) \rho_L(\gamma)^{-1} \cdot \ef_{0}.
  \end{equation}
  This map and its generalizations are well-studied (see e.g.\ \cite{Sch09}), whose properties are summarized in the following result.

  \begin{lemma}
    \label{lemma:sc-vv}
    When $k < 0$, we have
  \begin{equation}
    \label{eq:prinvv}
    \prin(\vv(f)) = \prin(f) \ef_{0}.
  \end{equation}
  for all  $f \in M^{!, \infty}_{k}(\Gamma_0(N))$, on which space
  the map $\vv$ is an isomorphism with the inverse given by
  \begin{equation}
    \label{eq:sc}
    g = \sum_{\mu \in L^\vee/L} g_\mu \ef_\mu \mapsto g_{0},
  \end{equation}
Furthermore, it preserves the rationality of the Fourier expansion at the cusp infinity.   \end{lemma}
  \begin{proof}
See Proposition 4.2 in \cite{Sch09} and Proposition 6.12, Corollary 6.14 in \cite{BHKRY20a}.
  \end{proof}

  As a consequence, we can relate the higher Green function $G^{\Gamma_0(N)}_{r+1, f}$   from the introduction to one on the Shimura variety $X_K$.
  \begin{corollary}
    \label{cor:1}
    Under the isomorphism \eqref{eq:XKisom}, we have
  \begin{equation}
    \label{eq:scG-vvG}
    G^{\Gamma_0(N)}_{r+1, f}(z_1, z_2)
    = - \Phi_{\vv(f)}^r([w(z_1, z_2), 1])
  \end{equation}
  with $G^{\Gamma_0(N)}_{r+1, f}$ the higher Green function defined in \eqref{eq:Grf}
  for $f \in M^{!, \infty}_{-2r}(\Gamma_0(N))$ with $r > 0$.
\end{corollary}

\begin{proof}
Under \eqref{eq:XKisom}, the divisor
  \begin{equation}
    \label{eq:Zm}
    Z(m, 0) =
    \Gamma_0(N)^2 \backslash
    \{(z_1, z_2) \in \Hb^2: (\smat{z_1}{-z_1z_2}{1}{-z_2} , x) = 0 \text{ for some } x \in L_{m, 0}\}
  \end{equation}
on $X_K$  is simply the $m$-th Hecke correspondence $T_m$ on $X_0(N) \times X_0(N)$.
Therefore, the two sides of \eqref{eq:scG-vvG} have logarithmic singularity along the same divisor.
Using Corollary 4.2 and Theorem 4.4 of \cite{BEY21}, we see that their difference is a smooth function in $L^2(X_0(N)^2)$ and an eigenfunction of the Laplacians in $z_1$ and $z_2$ with eigenvalue $r(1-r) < 0$. By fixing $z_2$, this difference is an eigenfunction of the Laplacian on $X_0(N)$ with negative eigenvalue, which vanishes identically. This holds for any $z_2 \in X_0(N)$, and we obtain \eqref{eq:scG-vvG}.
\end{proof}

\begin{remark}
  \label{rmk:relation}
  Following section V.4 of \cite{GZ86}, we call a set of integers $\{\lambda_m: m \in \Nb\}$ a \textit{relation for $S_{2-k}(\Gamma_0(N))$} if only finitely many $\lambda_m$ are non-zero and
  $$
\sum_{m \ge 1} \lambda_m a_m = 0
$$
for all $\sum_{m \ge 1} a_m q^m \in S_{2-k}(\Gamma_0(N))$.
Since $g_0 \in S_{2-k}(\Gamma_0(N))$ for all $g \in S_{2-k, L^-}$, we have
$$
\{f_P, g \} := \mathrm{CT}\lp\sum_{\mu \in L^\vee/L} f_{P, \mu} g_\mu\rp = 0
$$
for $f_P = \sum_{m \ge 1} \lambda_m q^{-m} \ef_0$. By Serre duality \cite{Borcherds99}, there exists $f \in M^!_{k, L}$ with $\prin(f) = f_P$.
\end{remark}

Suppose $E_1, E_2$ are imaginary quadratic fields such that $E = E_1E_2$ is biquadratic containing a real quadratic field $F$, i.e.\ $E_1 \neq E_2$.
Then for any CM points $z_j \in E_j$, the point $(z_1, z_2) \in \Hb^2$ is sent to $Z(W_\alpha) \cup Z(W_\alv) \subset X_K$ under the isomorphism in \eqref{eq:XKisom} (see section 3.2 in \cite{YY19} for details).

\subsection{Eisenstein Series}
\label{subsec:Eisenstein}
We recall coherent and incoherent Eisenstein series for the group $G = \SL_2$ following \cite{BKY12}.
Let $F$ be a totally real field of degree $d$ and discriminant $D_F$, $E/F$ be a quadratic CM extension with absolute discriminant $D_{E}$ and $\chi = \chi_{E/F} = \otimes_{v \le \infty} \chi_v$ the associated Hecke character.
For a standard section $\Phi \in I(s, \chi)$ with
\begin{equation}
  \label{eq:Ischi}
  I(s, \chi) := \Ind^{G(\Ab_F)}_{B(\Ab_F)}(|\cdot|^s \chi) = \bigotimes_{v \le \infty} I_v(s, \chi_v),~ I_v(s, \chi_v) := \Ind^{G(F_v)}_{B(F_v)}(|\cdot|_v^s \chi_v),
\end{equation}
we can form the Eisenstein series
$$
E^*(g, s, \Phi) := \Lambda(s + 1, \chi) E(g, s, \Phi),~
E(g, s, \Phi) := \sum_{\gamma \in B\backslash \SL_2(F)} \Phi(\gamma g, s).
$$
where $\Lambda(s, \chi)$ is the completed $L$-function for $\chi$ (see equation (4.6) in \cite{BKY12}).
When $\Phi = \otimes_v \Phi_v$, the Eisenstein series $E(g, s, \Phi)$ has the Fourier expansion
$$
E^*(g,s , \Phi) = E^*_0(g,s , \Phi) + \sum_{\mut \in F^\times} E^*_\mut(g,s , \Phi),~
$$
and  for $t \in  F^\times$,
$$
E^*_\mut(g, s, \Phi) =  \prod_v W^*_{\mut, v}(g, s, \Phi_v),
$$
where $W^*_{t, v}$ is the normalized local Whittaker function defined by
\begin{equation}
  \label{eq:W*s}
  \begin{split}
    W^*_{\mut, v}(g_v, s, \Phi_v) &:= |D_E/D_F|_v^{-(s+1)/2} L(s + 1, \chi_v) \int_{F_v} \Phi_v(w n(b) g_v, s) \psi_v(-\mut b) db    .
  \end{split}
\end{equation}
For simplicity, we denote
\begin{equation}
  \label{eq:W*0}
  W^*_{t, v}(\Phi_v) :=   W^*_{t, v}(1, 0, \Phi_v),~
  W^{*, \prime}_{t, v}(\Phi_v) :=   \partial_s W^*_{t, v}(1, s, \Phi_v) \mid_{s = 0}.
\end{equation}
We will be interested in the case when $\Phi$ is a Siegel-Weil section.

Given $\alpha \in F^\times$, we view $W_\alpha = E$ as an $F$-quadratic space with quadratic form $Q_\alpha(z) := \alpha z\bar z$.
Denote $\omega_\alpha$ the associated Weil representation.
We have an $\SL_2$-equivariant map $\lambda_\alpha:  \Sc(W_\alpha(\Ab_F)) \to I(0, \chi)$ given by
\begin{equation}
  \label{eq:SWsection}
\lambda_\alpha(\phi)  (g) = (\omega_\alpha(g) \phi)(0).
\end{equation}
At each place $v$, there are local versions of $\omega_\alpha$ and $\lambda_\alpha$ as well, denoted by $\omega_{\alpha, v}$ and $\lambda_{\alpha, v}$.
When $\Phi =  \lambda_\alpha(\phi)$, resp.\ $\Phi_v = \lambda_{\alpha, v}(\phi_v)$, we replace $\Phi$, resp.\ $\Phi_v$, from the notations above by $\phi$, resp.\ $\phi_v$.

Let $Z(W_\alpha)$ be the CM points on $X_K$ as in Section \ref{subsec:CM} and $W_{\alpha, \Qb} \subset \V$ the rational quadratic space as in \eqref{eq:isometry}.
A special case of the Siegel-Weil formula (see \cite[Theorem 4.5]{BKY12}) gives us
\begin{equation}
  \label{eq:SWCM}
  \theta(g, Z(W_\alpha), \phi) = C \cdot E(g^\Delta, 0, \phi)
\end{equation}
for any $\phi \in \Sc(W_{\alpha, \Qb}(\Ab)) = \Sc(W_\alpha(\Ab_F))$.
Here $C = \deg(Z(W_\alpha))/2$.

Suppose only the $j$-th real embedding of $\alpha$ is negative.
Denote $\oneb := (1, \dots, 1)$ and $\oneb(j) := (1, \dots, -1, \dots, 1)$ with $-1$ at the $j$-th slot.
The sections
$\Phi(j) = \lambda_{\alpha, f}(\phi) \otimes \Phi^{\oneb(j)}_{\infty}$ and
$\Phi = \lambda_{\alpha, f}(\phi) \otimes \Phi^{\oneb}_{\infty}$ are coherent and incoherent respectively.
For all $\phi \in \Sc(W_\alpha(\hat F); \Cb)$, the Eisenstein series $E^*(g, s, \Phi(j))$ is holomorphic of weight $\oneb(j)$ at $s = 0$ and
\begin{equation}
  \label{eq:cohE}
  E^*(\tau, \phi, \oneb(j)) :=  \Nm(v)^{-1/2} E^{*}(g_\tau, 0, \Phi)
\end{equation}
is called a \textit{coherent} Eisenstein series.
On the other hand, the Eisenstein series $E^*(g, s, \Phi)$ vanishes at $s = 0$ and its derivative
\begin{equation}
  \label{eq:tE}
  E^{*, \prime}(\tau, \phi) :=
  \partial_s \Nm(v)^{-1/2} E^{*}(g_\tau, s, \Phi) \mid_{s = 0}
\end{equation}
is called an \textit{incoherent} Eisenstein series, which is related to the coherent Eisenstein series via the differential equation \cite[Lemma 4.3]{BKY12}
\begin{equation}
  \label{eq:diffeq}
2L_{\tau_j} E^{*, \prime}(\tau, \phi) = E^*(\tau, \phi, \oneb(j))
\end{equation}
for all $1 \le j \le d$.
Furthermore, it has the Fourier expansion
  \begin{equation}
    \label{eq:tEFE}
  E^{*, \prime}(\tau, \phi)  =
  \Ec(\tau, \phi) + \phi(0)  \Lambda(0, \chi) \log \Nm(v) + \Ec^*(\tau, \phi),
\end{equation}
where
$\Ec^*(\tau, \phi)$ has exponential decay near the cusp infinity and
$$
\Ec(\tau, \phi) = a_0(\phi) + \sum_{t \in F,~ t \gg 0} a_t(\phi) \ebf(\tr(t \tau)).
$$
Here, $a_0(\phi)$ is an explicit constant (see (2.24) in \cite{YY19}) and
\begin{equation}
  \label{eq:at}
  \begin{split}
    a_t(\phi) &:=
    \begin{cases}
(-i)^d \tW_t(\phi) \log \Nm(\pf)    , & \text{ if } \mathrm{Diff}(W_\alpha, t) = \{\pf\},\\
      0,& \text{ otherwise.}
    \end{cases}
  \end{split}
\end{equation}
The coefficient $\tW_t(\phi)$ is  given by (see \cite[Proposition 2.7]{YY19})
\begin{equation}
  \label{eq:tW}
  \tW_t(\phi) = 2^d \frac{W^{*, \prime}_{t, \pf}(\phi_\pf)}{\log \Nm(\pf)} \prod_{v \nmid \pf \infty} W^*_{t, v}(\phi_v) \in \Qb(\phi)
\end{equation}
when $\mathrm{Diff}(W_\alpha, t) = \{\pf\}$

\subsection{Hecke's Cusp Form}
\label{sec:Hecke}

Denote $\omega_\a  = \omega_{V_\a , \psi}$ the Weil representation
and $\theta_\a := \theta_{V_\a}$ the theta function as in \eqref{eq:theta}.
For
 a bounded, integrable function $\rho: H_\a (\Qb) \backslash H_\a (\Ab)/K \to \Cb$,
consider the following theta lift
\begin{equation}
  \label{eq:vta}
\vartheta_\a (g, \varphi, \rho) := \int_{[H_\a ]} \theta_\a (g, h, \varphi) \rho(h) dh,
\end{equation}
The measure $dh$ is the product measure of the local measures $dh_p$, where $dh_p$ is normalized such that the maximal compact subgroup in $H_\a(\Qb_p)$ has volume 1.
Such integral was first considered by Hecke in \cite{Hecke26} when $\rho = \vc$ is an odd, continuous character, i.e.
\begin{equation}
  \label{eq:rhosgn}
\vc(h) =  \sgn(h_\infty) \vc_f(h_f)
\end{equation}
 with  $\vc_f$ a continuous character on $H_a(\hat \Qb)$ and
 $\varphi = \varphi^\pm_\infty \varphi_f$ with
\begin{equation}
  \label{eq:varphipm}
\varphi^\pm_\infty(x_1, x_2) := (x_1 \pm  x_2) e^{-\pi (x_1^2 + x_2^2)} \in \Sc(\Rb^2).  \end{equation}
Notice that $\varphi^\pm_\infty$ satisfies $\varphi^\pm_\infty(-x_1, -x_2) = - \varphi^\pm_\infty(x_1, x_2)$.

In this case, the $m$-th Fourier coefficient of $\vartheta_\a $ is given by
$$
W_m(\varphi, \vc )(g) := \int_{[N]} \vartheta_\a (n g, \varphi, \vc)
{\psi(-mn)}dn.
$$
for $m \in \Qb$.
To evaluate it, we apply the usual unfolding trick
\begin{align*}
  W_m(\varphi, \vc )(g)
  &= \int_{[N]}
    \int_{[H_\a ]}\sum_{\lambda \in V_\a (\Qb)} (\omega_\a (ng) \varphi)(h^{-1} \lambda)    \vc(h) dh {\psi(-mn)}dn\\
  &=     \int_{[H_\a ]}\sum_{\lambda \in V_{\a , m}(\Qb)}  (\omega_\a (g)\varphi)(h^{-1} \lambda)    \vc(h)dh\\
  &=
    \sum_{\lambda \in H_{\a }(\Qb) \backslash V_{\a , m}(\Qb)}
    \int_{H_{\a , \lambda}(\Qb) \backslash H_\a (\Ab)} (\omega_\a (g)\varphi)( h^{-1} \lambda)        \vc(h)dh.
\end{align*}
When $m = 0$, we have $\lambda = 0$ since $V_\a $ is anisotropic and $\varphi(0) = \varphi^\pm_\infty(0)\varphi_f(0) = 0$.
When $Q_\a (\lambda) = m \neq 0$, the group $H_{\a , \lambda}$ is trivial.
We can then write $g = g_\tau g_f$ with $g_\tau = n(u)m(\sqrt{v}) \in G(\Rb)$ and $g_f \in G(\hat \Qb)$ and obtain
$$
  W_m(\varphi, \vc )(g) =    \sum_{\lambda \in H_{\a }(\Qb) \backslash V_{\a , m}(\Qb)}
  \int_{\Rb^\times}  (\omega_\a (g_\tau)\varphi^\pm_\infty)
(t^{-1}(\iota_\a(\lambda)))
  \frac{dt}{t}
\int_{ H_\a (\hat \Qb)}        \vc_f(h) (\omega_\a (g_f)\varphi_f)(h^{-1}\lambda) dh.
$$
The group $H_\a (\Qb)$ acts on $V_{\a , m}(\Qb)$ transitively.
The archimedean integral can be evaluated as
\begin{align*}
  \int_{\Rb^\times}  (\omega_\a (g_\tau)\varphi^\pm_\infty)( t^{-1}(\iota_\a(\lambda)))\frac{dt}{t}
  &= 2 \ebf(m u) {\frac{v}{\sqrt\a}} \int_{0}^\infty (t^{-1} \lambda_1 \pm \sgn(\a)t\lambda_2) e^{-\pi \frac{v}{|\a|}(t^{-2}\lambda_1^2 + t^2\lambda_2^2)}        \frac{dt}{t}\\
&    = 2 v \ebf(m u)\sgn(\lambda_1) \sqrt{|m|} \int^\infty_0 (t^{-1} \pm \sgn(m) t)e^{-\pi v|m|(t^{-2} + t^2)} \frac{dt}{t}.
\end{align*}
This is 0 if $\pm \sgn(m) < 0$ by the change of variable $t \mapsto 1/t$, and can be otherwise evaluated using the lemma below.
\begin{lemma}
  For any $\beta > 0$, we have
  $$
\int^\infty_0 (t^{-1} + t) e^{-\beta\pi(t^{-2} + t^2)} \frac{dt}{t} = \frac{e^{-2\pi \beta}}{\sqrt{\beta}}.
  $$
\end{lemma}
Therefore, we have
\begin{equation}
  \label{eq:FCm}
W_m(\varphi^\pm, \vc )((g_\tau, g_f)) =
2\sqrt{v} \ebf(mu) \ebf(|m|iv)
\sgn(\lambda_1)
\int_{ H_\a (\hat \Qb)}        \vc_f(h) (\omega_\a (g_f)\varphi_f)(h^{-1}\lambda) dh,
\end{equation}
when $\pm m > 0$ and $\lambda \in V_{\a , m}(\Qb) $. Otherwise, it is 0.
The integral in \eqref{eq:FCm} can be evaluated locally.
Notice that it always converges as $\varphi_f(h^{-1} \lambda)$ has compact support as a function of $h \in H_\a (\hat \Qb)$.

\subsection{The Deformed Theta Integral}
\label{subsec:modify}
We recall the real-analytic modular form of weight one constructed in \cite{CL20} using the notations of section \ref{sec:Hecke}.
Let $\vc$ be an odd, continuous character as in \eqref{eq:rhosgn}, and $K_\vc \subset H_\a (\hat\Zb)$ the open compact subgroup defined in \eqref{eq:Kvc}.
The intersection $H_\a(\Qb)^+ \cap K_\vc$, where $H_\a(\Qb)^+ := H_\a(\Qb) \cap H_\a(\Rb)^+$, is a cyclic subgroup
\begin{equation}
  \label{eq:eK1}
\Gamma_{\vc} = \langle \ve_{\vc} \rangle,~ \ve_{\vc} > 1 > \ve_{\vc}' > 0
\end{equation}
of the totally positive units in $\Oc$.
Then we have
\begin{equation}
  \label{eq:H1decomp}
H_\a (\Ab) = \coprod_{\xi \in C} H_\a (\Qb) H_\a (\Rb)^+ K_\vc \xi,
\end{equation}
where $C \subset H_\a (\hat\Qb)$ is a finite subset of elements representing $H_\a (\Qb)^+ \backslash H_\a (\hat\Qb) / K_\vc$. So given $h = (h_f, h_\infty)$, we can find $\alpha  \in H_\a (\Qb), t \in H_\a (\Rb)^+, k_1 \in K_\vc, \xi \in C$ all depending on $h$ such that
\begin{equation}
  \label{eq:hdecomp}
h =   (\alpha k_1 \xi, \alpha t),
\end{equation}
though the choice is not unique.
This gives us the identification
\begin{equation}
  \label{eq:H1id}
H_\a (\Qb) \backslash H_\a (\Ab) / K_\vc
\cong
\coprod_{\xi \in C} \Gamma_{\vc} \backslash H_\a (\Rb)^+  \xi
\end{equation}
by sending $h = (\alpha k_1 \xi, \alpha t) \in H_\a (\Ab)$ as in \eqref{eq:hdecomp} to $t \in H_\a (\Rb)^+$ in the $\xi$-component.
Just like the decomposition \eqref{eq:H1decomp}, this isomorphism depends on the choice of the set of representatives $C$.
Similarly, we have
\begin{equation}
  \label{eq:C}
  H_\a (\hat\Qb) / K_\vc   \cong
  \coprod_{\xi \in C} \Gamma_{\vc} \backslash H_\a (\Qb)^+\xi.
\end{equation}

Using the Fourier coefficient $W_m$ in \eqref{eq:FCm} and the decomposition in \eqref{eq:C},
we can write
%
\begin{equation}
  \label{eq:vt1}
  \vt_\a(g, \varphi^\pm, \vc) =
 \vol(K_\vc)  \sum_{\xi \in C} \vc(\xi)
  \sum_{\beta \in \Gamma_{\vc} \backslash V_{\a}(\Qb),~ \pm Q_\a(\beta) > 0}
     (\omega_\a(g) \varphi^{0, \pm}))(\xi^{-1} \beta)   .
   \end{equation}
   for $g \in G(\hat\Qb)\times B(\Rb)$,
where $\varphi^\pm = \varphi_f \varphi_\infty^\pm$ with $\varphi_f \in \Sc(\hat{V}_\a)$ being $K_\varrho$-invariant and
\begin{equation}
  \label{eq:varphi0}
 \varphi^{0, \pm} = \varphi_f \varphi_\infty^{0, \pm},~ \varphi^{0, \pm}_\infty(x_1, x_2) := \sgn(x_1) e^{\mp 2 \pi x_1 x_2}.
\end{equation}
Although $\varphi^{0, \pm}_\infty$ is not a Schwartz function on $\Rb^2$, the sum above still converges absolutely.
For $g \in B(\Rb)$, the quantity $\omega_a(g) \varphi^{0, \pm}_\infty$ is defined
with the usual formula of the Weil representation, and $\vt_\a(g, \varphi^\pm, \vc)$ is right $\SO_2(\Rb)$-equivariant with weight $\pm 1$.
We also  have a left $G(\Qb)$-invariant function
\begin{equation}
\label{eq:Theta}
  \begin{split}
      \Theta_{\a}(g, \varphi^\pm, \vc)
    &:=
    \int_{H_\a(\Qb)\backslash H_\a(\hat\Qb)} \theta_\a(g, h, \varphi^\pm) \vc(h) dh
=
\vol( {K_\vc})    \sum_{\xi \in C} \vc(\xi) \theta_{\a}(g, \xi, \varphi^\pm)
  \end{split}
  \end{equation}
on $G(\Ab)$ for $\varphi \in \Sc(V_\a(\Ab))$, which is
independent of  the choice of $C$.

We now define a function $\lg_C :
H_\a (\Qb) \backslash H_\a (\Ab) / {K_\vc} \to [0, 1)$ by
\begin{equation}
  \label{eq:lg}
  \begin{split}
    \lg_C((\alpha k_1 \xi, \alpha t))
    &:= 2 \log \ve_{\vc}\cdot  \{\log t / \log \ve_{\vc}\},\\
    \{a\} &:= a - \lim_{\e \to 0} \frac{1}{2}\lp \lfloor a + \e \rfloor + \lfloor a - \e \rfloor \rp.
  \end{split}
\end{equation}
Note that $\{0\} = \frac{1}{2}$.
 Unlike  the function considered by Hecke, $\lg_C$ cannot be written as the product of functions on $H_\a (\hat\Qb)$ and $H_\a (\Rb)$.
 Denote
\begin{equation}
  \label{eq:tvc}
  \tvc_C := \lg_C \cdot \vc:
H_\a (\Qb) \backslash H_\a (\Ab) / {K_\vc} \to \Cb.
\end{equation}
Given $\varphi = \varphi_f \varphi_\infty \in \Sc(V_a(\Ab))$ for some ${K_\vc}$-invariant $\varphi_f \in \Sc(\hat V_a; \Cb)$, the deformed theta integral $\vt_\a (g, \varphi, \tvc_C)$, where $\vt_\a $ is defined in \eqref{eq:vta}, was studied in \cite{CL20},
To describe its Fourier expansion, denote
\begin{align*}
  \vt^*_{\a}(g, \varphi_f, \vc)
  &:= \int_{H_\a(\Qb)^+\backslash H_\a(\hat\Qb)}\vc(h)  \sum_{\substack{\beta \in \Gamma_{\vc} \backslash V_{\a}(\Qb)\\Q_\a(\beta) < 0}} (\omega_\a(g) \varphi^{0, *})(h^{-1} \beta) dh\\
  &=
-\vol({K_\vc}) \sqrt{v}  \sum_{\xi \in C} \vc(\xi)
  \sum_{\substack{\beta \in \Gamma_{\vc} \backslash V_{\a}(\Qb)\\Q_\a(\beta) < 0}} (\omega_\a(g_f) \varphi_f)(\xi^{-1} \beta)
  \sgn(\beta)    \ebf(Q_\a(\beta)\tau ) \Gamma(0, 4\pi |Q_\a(\beta)| v).
\end{align*}
for $g = (g_f, g_\tau) \in G(\Ab)$ with $\varphi^{0, *} = \varphi_f \varphi^{0, *}_\infty$ and
\begin{equation}
  \label{eq:varphi0*}
  \varphi^{0, *}_\infty(x_1, x_2) := - \sgn(x_1) e^{-2\pi x_1x_2} \Gamma(0, 4\pi |x_1x_2|).
\end{equation}
Note that $\vt^*_\a$  is not necessarily left-$G(\Qb)$ invariant.
But it is modular after applying the lowering operator as
$$
L \lp   \vt^*_{\a}(g, \varphi_f, \vc) \rp
= \vt_{\a}(g, \varphi^-, \vc).
$$
Similarly for $\xi \in H_\a(\hat\Qb)$, define
\begin{equation}
  \label{eq:theta*}
  \begin{split}
    \theta^*_\a&(g, \xi, \varphi_f)
    :=    \sum_{\lambda \in V_{\a}(\Qb)} (\omega_\a(g) \varphi^*)(\xi^{-1} \lambda)
=     \sum_{\lambda \in V_{\a}(\Qb)} (\omega_\a(g_f) \varphi_f)(\xi^{-1} \lambda) (\omega_\a(g_\tau)\varphi^*)(\lambda)    \\
    &=  - v   \sum_{\lambda \in V_{\a}(\Qb)} (\omega_\a(g_f) \varphi_f)(\xi^{-1} \lambda)
    \frac{\sgn(\lambda_1 - \sgn(\a)\lambda_2)}{\sqrt{\pi}}
\ebf(Q_\a(\lambda)\tau) \Gamma\lp \frac{1}{2}, \frac{\pi v}{|\a|}(\lambda_1 - \sgn(\a)\lambda_2)^2\rp.
  \end{split}
\end{equation}
Here, we have employed the rapidly decaying function
\begin{equation}
  \label{eq:varphi*}
  \varphi^*(x_1, x_2) := -e^{-2\pi x_1 x_2}
  \sgn(x_1-x_2)\Gamma\lp \frac{1}{2}, \pi (x_1-x_2)^2\rp,
\end{equation}
where $B(\Rb)\subset G(\Rb)$ acts via $\omega_a$ and $\SO_2(\Rb)$ acts with weight 1.
Also, we denote
\begin{equation}
  \label{eq:Theta*}
    \Theta^*_{\a, C}(g, \varphi_f, \vc)
    := \vol({K_\vc})   \sum_{\xi \in C} \vc(\xi) \theta^*_{\a}(g, \xi, \varphi_f),
  \end{equation}
  which depends on the choice of $C$ and   satisfies
  \begin{equation}
    \label{eq:LTheta*}
    L \Theta^*_{\a, C}(g, \varphi_f, \vc)
    = \Theta_{\a}(g, \varphi^-, \vc).
  \end{equation}
We recall some results.
\begin{theorem}
  \label{thm:modified}
  Let $\tvc_C$ be as in \eqref{eq:tvc} and $\varphi_f \in V_\a(\hat\Qb)$ a right-${K_\vc}$ invariant function.
  Then the integral $\vt_\a(g, \varphi^+, \tvc_C)$ defines a $G(\Qb)$-invariant function in $g \in G(\Ab)$ of weight 1 with respect to $\SO_2(\Rb)$.
  Furthermore, it  has the Fourier expansion
\begin{equation}
  \label{eq:vtFE}
  \begin{split}
    \vt_\a(g, \varphi^+, \tvc_C)
    &=
     {\vt^*_\a(g, \varphi_f, \vc)}
      +  \log \ve_{\vc} \Theta^*_{\a, C}(g,  \varphi_f, \vc)\\
&\quad + \vol({K_\vc})     \sum_{\substack{\xi \in C\\ \beta \in \Gamma_{\vc}\backslash V_\a(\Qb)\\ Q_\a(\beta) > 0}}
      \tvc_C(      (\xi, \sqrt{|\beta/\beta'|})) (\omega_a(g)\varphi^{0, +})(\xi^{-1} \beta),
  \end{split}
\end{equation}
where $\varphi^{0, +} = \varphi_f \varphi^{0, +}_\infty$ is defined in \eqref{eq:varphi0}, and satisfies
  \begin{equation}
    \label{eq:diffeq}
    L \vt_\a(g, \varphi^+, \tvc_C) =
    \vt_{\a}(g, \varphi^-, \vc) + \log \ve_{\vc} \Theta_{\a}(g, \varphi^-, \vc).
  \end{equation}
\end{theorem}

\begin{proof}
  This follows essentially from Proposition 5.5 in \cite{CL20}. For completeness, we include a different (and slightly shorter) proof here.
  As in the evaluation of $\vt_\a(g, \varphi, \vc)$ in \eqref{eq:vt1}, we have
  \begin{align*}
    W_m(\varphi, \tvc_C)(g_\tau)
    &=
\vol({K_\vc})  \ebf(mu) \sqrt{v}\sum_{\xi \in C} \sum_{\beta \in \Gamma_{\vc} \backslash V_{\a, m}(\Qb)} \varphi_f(\xi^{-1} \beta) J(\beta, v), \\
  J(\beta, v)
    &:=2\log \ve_{\vc}   \int^\infty_0 \varphi^+_\infty(t^{-1}\cdot( \iota_\a(\beta) \sqrt{v})) \left\{ \frac{\log t}{\log \ve_{\vc}} \right\} \frac{dt}{t}\\
    &=2\log \ve_{\vc} \sgn(m\beta)   \int^\infty_0 \varphi^{\sgn(m)}_\infty(\sqrt{|m|v} (t, t^{-1})) \left\{ \frac{\log t}{\log \ve_{\vc}} + \frac{1}{2} \frac{\log |\beta/\beta'|}{\log \ve_{\vc}}      \right\} \frac{dt}{t}.
  \end{align*}
  To verify \eqref{eq:diffeq}, we start with
  \begin{align*}
L_\tau &\ebf(m u) J(\beta, v)
    =
      \ebf( m \tau)          v^2 \partial_v       \lp e^{2\pi m v} J(\beta, v)      \rp\\
    &=       \ebf( m \tau)
      2\log \ve_{\vc} \sgn(m\beta)   \int^\infty_0
      v^2 \partial_v
 e^{2\pi m v}
      \varphi^{\sgn(m)}_\infty(\sqrt{|m|v} (t, t^{-1}))
      \left\{ \frac{\log t}{\log \ve_{\vc}} + \frac{1}{2} \frac{\log |\beta/\beta'|}{\log \ve_{\vc}}      \right\} \frac{dt}{t}\\
    &=
      v^{}  \ebf(mu ) \sgn(m\beta)
      \log \ve_{\vc}   \int^\infty_0
 \partial_t
      \varphi^{-\sgn(m)}_\infty(\sqrt{|m|v}(t, t^{-1}))
\left\{ \frac{\log t}{\log \ve_{\vc}} + \frac{1}{2} \frac{\log |\beta/\beta'|}{\log \ve_{\vc}}      \right\} dt\\
    &=
-      v  \ebf(mu )
\lp \sgn(m\beta)      \int^\infty_0
      \varphi^{-\sgn(m)}_\infty(\sqrt{|m|v}(t, t^{-1}))
      \frac{dt}{t}-
      \log \ve_{\vc}      \sum_{\ve \in \Gamma_{\vc}}\varphi^+_\infty(\iota_\a(\beta \ve)) \rp\\
    &= v \ebf(mu ) \lp  \delta_{m < 0} \sgn(\beta) +       \log \ve_{\vc}      \sum_{\ve \in \Gamma_{\vc}}\varphi^+_\infty(\iota_\a(\beta \ve)) \rp.
  \end{align*}
  Substitute this into the left hand side of \eqref{eq:diffeq} proves it.

  Now to calculate the Fourier expansion, it suffices prove the claim
  \begin{equation}
    \label{eq:claim}
    \lim_{v \to \infty} e^{2\pi m v} J(\beta, v)
    = \sgn(\beta) \lg_C((1, \sqrt{|\beta/\beta'|}).
  \end{equation}
  For each $\beta \in \Gamma_{\vc} \backslash V_\a(\Qb)$, we choose the unique representative $\beta_0 \in \Gamma_{\vc} \beta$ such that $|\beta_0/\beta_0'| \in [1, \ve_{\vc}^2)$.
  We can then write $J(\beta, v) =    J_1(\beta_0, v) + J_2(\beta_0, v)$,   where
  \begin{align*}
    J_1(\beta_0, v)
    &:=2 \int^\infty_0 \varphi^+_\infty(t^{-1}\cdot( \iota_\a(\beta_0) \sqrt{v})) \log t  \frac{dt}{t},\\
    J_2(\beta_0, v)
    &:=- 2\log \ve_{\vc}   \int^\infty_0 \varphi^+_\infty(t^{-1}\cdot( \iota_\a(\beta_0) \sqrt{v})) \left\lfloor \frac{\log t}{\log \ve_{\vc}} \right\rfloor \frac{dt}{t}.
  \end{align*}
  For $J_1$, we have
  \begin{align*}
    \lim_{v \to \infty} e^{2\pi m v} J_1(\beta_0, v)
    &=
      \log |\beta_0/\beta_0'|      \sgn(m\beta_0)
      \lim_{v \to \infty} e^{2\pi m v}
      \int^\infty_0 \varphi^{\sgn(m)}_\infty(\sqrt{|m|v} (t, t^{-1}))       \frac{dt}{t}\\
      &=       \log |\beta_0/\beta_0'|      \sgn(\beta_0) \delta_{m > 0}.
  \end{align*}
  For $J_2$, the limit vanishes unless $|\beta_0/\beta_0'| = 1$, in which case
    \begin{align*}
    \lim_{v \to \infty} e^{2\pi m v} J_2(\beta_0, v)
    &=
2\log \ve_{\vc} \sgn(m\beta_0)      \lim_{v \to \infty} e^{2\pi m v}
      \int^1_{\ve_{\vc}^{-1}} \varphi^{\sgn(m)}_\infty(\sqrt{|m|v} (t, t^{-1}))       \frac{dt}{t}\\
      &= \log \ve_{\vc} \sgn(\beta_0) \delta_{m > 0}.
  \end{align*}
Putting these together proves claim \eqref{eq:claim}.
\end{proof}
Finally, we record a result as a direct consequence of Theorem 4.5 in \cite{CL20} (see also section 5 in \cite{LS22}).
\begin{proposition}
  \label{prop:mixmock}
  For any $\varphi_f \in \Sc(\hat V_1; \Cb)$, there exists a real-analytic modular form $\tilde\Theta_{\a, C}(g, \varphi^-, \vc) = \tilde\Theta^+_{\a, C}(g, \varphi^-, \vc) + \tilde\Theta^*_{\a, C}(g, \varphi^-, \vc)$ such that $L \tilde\Theta_{\a, C} = \Theta_{\a}$ and
  $\sqrt{v}\tilde\Theta^+_{\a, C}(g_\tau, \varphi^-, \vc)$ is holomorphic in $\tau$ with Fourier coefficients in $\Qb(\varphi_f)$.
\end{proposition}

\section{
    Doi-Naganuma Lift of Hecke's Cusp Form}
\label{sec:match}
In this section, we are interested in computing the $\mathrm{O}(2, 2)$ theta lift of Hecke's cusp form from section \ref{sec:Hecke}, and realize it as coherent Hilbert Eisenstein series from \ref{subsec:Eisenstein} over real quadratic fields.
The main result of this section is the global matching Theorem \ref{thm:match}, where we show that any coherent Eisenstein series can be realized as such a theta lift.
This global statement follows from its local counterpart in Theorem \ref{thm:localmatch}, which is improved further in Theorem \ref{thm:localmatchs} to allow matching deformed local sections.
This last result will be crucial for us in proving the factorization result in Proposition \ref{prop:tI+} later.

\subsection{Quadratic Spaces}
\label{subsec:V}
Let $V_{\pm 1}$ be as in Section \ref{subsec:Va},   $\ell^+, \ell^-$ be isotropic lines such that $\ell^+ \oplus \ell^-$ is a hyperbolic plane and denote
\begin{equation}
  \label{eq:V}
  V := V_0 \oplus V_1,~
  V_0  := \ell^+ \oplus \ell^- \oplus V_{-1}.
\end{equation}
We can realize
\begin{align*}
  V_0(\Qb) &\cong \left\{
\Lambda \in M_2(F): \Lambda^t = \Lambda'
  \right\}\\
(a, b, \lambda) &\mapsto \pmat{a}{\lambda}{\lambda'}{b}
\end{align*}
with $ \det$ as the quadratic form, and furthermore write
\begin{equation}
  \label{eq:V0decomp}
  V_0 = V_{00} \oplus U_D,~ V_{00} := V_0 \cap M_2(\Qb),~ U_D := \sqrt{D}\Qb\pmat{0}{-1}{1}{0} \cong (\Qb, Q_D),
\end{equation}
where $Q_D(x) = Dx^2$.
So $V$ has Witt rank 3 and admits the isotropic decomposition
\begin{equation}
  \label{eq:Vpm}
V = V^+ + V^-,~
  V^\pm := \ell^\pm + (V_{-1} + V_1)^\pm,~
(V_{-1} + V_1)^\pm(\Qb) := \{(\lambda, \pm \lambda): \lambda \in F\}
\end{equation}
with $V^\pm$ maximal totally isotropic subspaces.
For a $\Qb$-algebra $R$, e.g.\ $R \in \{\Qb, \Qb_p, \Rb, \hat\Qb , \Ab\}$,
We will use
\begin{equation}
  \label{eq:Velt}
  (a, b, \lambda, \mu) \in V(R),~
  a, b \in R, \lambda \in R \otimes F \cong V_{-1}(R), \mu \in R \otimes F \cong V_{1}(R)
\end{equation}
to represent elements in $V(R)$.
Define elements $f_j^\pm \in V$ by
\begin{equation}
  \label{eq:fj}
  \begin{split}
  f_1^+ &:= (1, 0, 0, 0),~   f_1^- := (0, 1, 0, 0),~
  f_2^+ := (0, 0, 1/2, 1/2),~   f_2^- := (0, 0, 1/2, -1/2),\\
  f_3^+ &:= (0, 0, \sqrt{D}/2, \sqrt{D}/2),~   f_3^- := (0, 0, 1/(2\sqrt{D}), -1/(2\sqrt{D})).  \end{split}
\end{equation}
Then $\{f_j^\pm: j = 1, 2, 3\} \subset V^\pm$ is a $\Qb$-basis of $V^\pm$.
With respect to the ordered basis $(f_1^+, f_2^+, f_3^+, f_1^-, f_2^-, f_3^-)$, the Gram matrix of $Q$ is $\kzxz {0} {I_3} {I_3} {0}$.
For $i = 1, 2, 3$, the following linear transformations
\begin{equation}
  \label{eq:wi}
  w_i(f_j^\pm) :=
  \begin{cases}
    f_j^\mp, & \text{ if } i = j,\\
    f_j^\pm, & \text{ otherwise.}
  \end{cases}
\end{equation}
are easily checked to be in $\mathrm{O}(V)$.
The unimodular lattice
\begin{equation}
  \label{eq:VZ}
  V_\Zb := \{(a, b, \lambda, \mu ) \in V(\Qb) \cap (\Zb^2 \times (\df^{-1})^2): \lambda - \mu \in \Oc_F\} \subset V
\end{equation}
provides $V$ with an integral structure. Similarly for $?\in \{00, 0, 1\}$, the lattice $V_{?, \Zb} := V_\Zb \cap V_? $ in $V_?$ gives it with an integral structure.

For $? \in \{00, 0, 1, -1, \emptyset\}$. we write
\begin{equation}
  \label{eq:?s}
 \tH_? := \Gspin(V_?),~ H_? := \SO(V_?),
\end{equation}
which are subgroups of $\tH$ and $H$ respectively by acting trivially on $V^\perp_?$, and have the following exact sequences
\begin{equation}
  \label{eq:GSpin}
1 \to \mathrm{G}_m \to  \tH_? \to H_? \to 1.
\end{equation}
For any commutative ring $R$, we have explicitly
\begin{equation}
  \label{eq:GSpin}
\iota:  \GSpin(V_{0, \Zb})(R) \cong \{\gamma \in \GL_2(\Oc \otimes_\Zb R): \det(\gamma) \in R^\times\},
\end{equation}
via the action of $\gamma \in \GL_2(\Oc \otimes_\Zb R)$ on $V_{\Zb, 0}(R)$
\begin{equation}
  \label{eq:act}
\Lambda \mapsto \det(\gamma)^{-1} \gamma \Lambda (\gamma')^t.
\end{equation}
For any $\Qb$-algebra $R$, we also have $\tH_?(R) = \Gspin(V_{?, \Zb})(R)$ for $? \in \{00, 0, 1, \emptyset\}$.
Therefore through $\iota$, we have
\begin{equation}
  \label{eq:gpid}
G_0 :=  \mathrm{Spin}(V_0) \cong G_F :=  \mathrm{Res}_{F/\Qb}(G),~ G_{00} := \mathrm{Spin}(V_{00}) \cong G,~ H_{00} \cong \mathrm{PGL}_2,
\end{equation}
and will  represent elements in $H_0$ by their preimages in $G_F$.
Denote $T_0 := \iota^{-1}(T) \subset \tH_0$
Then the relations among these groups can be visualized in the following diagram
\begin{equation}
  \label{eq:diag1}
  \begin{tikzcd}[row sep=scriptsize, column sep=tiny]
  & G_{00} \arrow[dl, "\cong"'] \arrow[rr, hook] \arrow[dd]& & G_0 \arrow[dl, "\cong"'] \arrow[rr, hook] \arrow[dd] & & \tH_0 \arrow[dl, hook'] \arrow[dd, hook]  & & T_0 \arrow[ll, hook'] \arrow[dl, "\cong"']  \\
\SL_2 \arrow[rr, crossing over, hook] & &  G_F \arrow[rr, crossing over, hook]  & & \GL_{2 / F}   & & T_{} \arrow[ll, crossing over, hook']\\
& H_{00}  \arrow[rr, hook]& & H_0 \arrow[rr, hook] & & H_0 & &
\end{tikzcd}
\end{equation}
Here the horizontal and vertical arrows are natural inclusions and surjections of algebraic groups respectively, and the diagonal arrows are induced by $\iota$.
Let $B_F \subset G_F$ be the standard parabolic subgroup, and $B_0 := \iota^{-1}(B_F) \subset G_0$.
They can be visualized as
\begin{equation}
  \label{eq:diag2}
  \begin{tikzcd}[row sep=scriptsize, column sep=tiny]
  & B_0 \arrow[dl, "\cong"'] \arrow[rr, hook] & & G_0 \arrow[dl, "\cong"'] \\
B_F \arrow[rr, crossing over, hook] & &  G_F
\end{tikzcd}
\end{equation}
which gives us
\begin{equation}
  \label{eq:BG}
 B_0(\Qb) \backslash G_0(\Qb)\cong  B_F(\Qb) \backslash G_F(\Qb) = B(F)\backslash G(F)
\end{equation}
via $\iota$.
We also denote
\begin{equation}
  \label{eq:TT}
\TT \subset T \times T_0 \subset \GL_2 \times \tH_0
\end{equation}
the diagonal, which will play a crucial role in the local matching result in section \ref{subsec:local-match-I}.

Now let $P \subset H$ be the Siegel parabolic stabilizing $V^+$, whose Levi factor is isomorphic to $\GL(V^+)$.
Then $P_0 := P \cap H_0 \subset H$ is the subgroup stabilizing the line $\ell^+$
and acting trivially on $V_1$.
The preimage of $P_0 H_{-1} \subset H_0$ in $\tH_0$ is given by $B_0 T_0$.
Combining with \eqref{eq:BG}, we obtain
\begin{equation}
  \label{eq:BG1}
  (P_0 H_{-1})(\Qb) \backslash H_0(\Qb) =
  (B_0T_0)(\Qb) \backslash \tH_0(\Qb) =
  B_0(\Qb) \backslash G_0(\Qb) \cong B(F)\backslash G(F).
\end{equation}
For $\alpha \in F^\times, \beta \in F$, we then have $m(\alpha), n(\beta) \in G_0(\Qb) \subset \tH(\Qb)$.
It is easy to check that
\begin{align*}
  (  \omega(m(\alpha)) \varphi)(a, b, \nu, \lambda)
  &=  \varphi(a/{\alpha\alpha'}, \alpha\alpha' b, \alpha'\nu/\alpha, \lambda),\\
  (  \omega(n(\beta)) \varphi)(a, b, \nu, \lambda)
  &=  \varphi(a - \beta\nu - \beta'\nu' + \beta\beta'b,  \beta, \nu - \beta b, \lambda)
\end{align*}
for a Schwartz function $\varphi \in \Sc(V(\Ab))$.

\subsection{Theta Integral}
\label{subsec:thetaint}
Let $\theta_0(g, g_1, \varphi_0)$ denote the theta function on $[G \times H_0]$ associated to $\varphi_0 \in \Sc(V_0(\Ab))$.
{Suppose $\varphi_{0, \infty}$ is in the polynomial Fock space $\mathbb{S}(V_0(\Rb))$ (see section \ref{subsec:L}).
Using $\mathbb{S}(V_0(\Rb)) = \mathbb{S}(V_{00}(\Rb)) \otimes \mathbb{S}(U_D(\Rb))$, we can then restrict $\theta_0$ to the subgroup $[G \times H_{00}]$, view it as a function on $[G \times G_{00}]$,
  and write
  \begin{equation}
    \label{eq:theta0decomp}
  \theta_0(g, g_{00}, \varphi_0) =
  \sum_{j \in J}  \theta_{00}(g, g_{00}, \varphi_{00, j})
  \theta_D(g, \varphi_{D, j}),
  \end{equation}
  where $\varphi_0= \sum_{j \in J} \varphi_{00, j} \varphi_{D, j}$ with $\varphi_{00, j} \in \Sc(V_{00}(\Ab))$ and $\varphi_{D, j} \in \Sc(U_D(\Ab))$.
}

We now define
\begin{equation}
  \label{eq:I0}
I_0(h_0, \varphi_0, f) := \int_{[G]} \theta_0(g, h_0, \varphi_0) f(g) dg
\end{equation}
for $f$ a bounded, integrable function on $[G]$.
Note that the measure $dg$ is normalized so that $[G]$ has volume 1.
In particular for a right $G(\hat\Zb)$-invariant function $\phi$ on $[G]$, we have
\begin{equation}
  \label{eq:dg}
  \int_{[G]} \phi(g) dg = \frac{3}{\pi} \int_{\SL_2(\Zb)\backslash \Hb} \phi(g_\tau) d\mu(\tau),~ d\mu(\tau) := \frac{dudv}{v^2}.
\end{equation}
When $f(g) = \vt_1(g, \varphi_1, \rho)$
for a bounded, integrable function $\rho$ on $H_1(\Qb)\backslash H_1(\Ab)$, the integral $I_0$ above becomes
\begin{equation}
  \label{eq:Ic}
  \begin{split}
      \Ic(h_0, \varphi, \rho) &:= \int_{[H_1]}
I((h_0, h_1), \varphi)
\rho(h_1) dh_1 = I_0(h_0, \varphi_0, \vartheta_1(\cdot, \varphi_1, \rho)),\\
I(h, \varphi) &:=
\int_{[G]} \theta(g, h, \varphi) dg
  \end{split}
\end{equation}
with $\varphi = \varphi_0 \otimes \varphi_1$.

For our purpose, $\rho = \vc$ will be an odd, continuous character as in \eqref{eq:rhosgn}, %
and $\varphi = \varphi_f \varphi^{ ( \e, - \e)}_\infty$ for $\e = \pm 1$ with
\begin{equation}
  \label{eq:varphiinf}
  \begin{split}
      \varphi^{(\e, - \e)}_\infty &:= \varphi^{(\e, -\e)}_{0, \infty} \otimes \varphi^-_\infty,\\
  \varphi^{(\e, -\e)}_{0, \infty}(a, b, \nu_1, \nu_2)
  &:=
  (\e  i  (a+b) +  (\nu_1 - \nu_2))
  e^{-\pi (a^2 + b^2 + \nu_1^2 + \nu_2^2)},
  \end{split}
\end{equation}
and $\varphi^\pm_\infty$ defined in \eqref{eq:varphipm}.
Here we have identified $V(\Rb) = \Rb^2 \oplus V_{1}(\Rb) \oplus V_{-1}(\Rb) \cong (\Rb^2)^{\oplus 3}$ via \eqref{eq:identify}.
For any $\theta \in \Rb, \e = \pm 1$, we have
$$
\omega(\kappa(\theta))\varphi^{(\e, -\e)}_\infty = \varphi^{(\e, -\e)}_\infty,~
\kappa(\theta) := \smat{\cos \theta}{\sin \theta}{-\sin \theta}{\cos\theta} \in \SO_2(\Rb) \subset G(\Rb),
$$
where $\omega$ is the Weil representation of $G(\Rb)$ on $V(\Rb)$.
On the other hand, for $h(\theta) = (\kappa(\theta), 1), h'(\theta) = (1, \kappa(\theta)) \in H_0(\Rb)$ with any $\theta \in \Rb$, it is easy to check that
$$
\omega(h(\theta))\varphi^{(\e, -\e)}_\infty = e^{\e i\theta} \varphi^{(\e, -\e)}_\infty,~
\omega(h'(\theta))\varphi^{(\e, -\e)}_\infty = e^{-\e i\theta} \varphi^{(\e, -\e)}_\infty.
$$
So $\varphi^{(\e, -\e)}_\infty$ is equivariant of weight $(\e , -\e)$ with respect to the connected component $\SO_2(\Rb) \times \SO_2(\Rb)$ of the maximal compact of $H_0(\Rb)$.
Later, we will also consider the following integral
\begin{equation}
  \label{eq:If}
  \Ic_f(h_0,\varphi, \vc) := \int_{H_1(\Qb)\backslash H_1(\hat\Qb)} \vc(h_1) \int_{[G]} \theta(g, (h_0, h_1), \varphi) dg dh_1,
\end{equation}
which is  similar to $\Ic(h_0, \varphi, \vc)$ and well-defined as
$$
\vc(-h_1)\theta(g, (h_0, -h_1), \varphi) = \vc(h_1)\theta(g, (h_0, h_1), \varphi)$$
for all $g, h_0, h_1$ and $\varphi_f \in \Sc(\hat V)$.
When $\varphi = \varphi_0 \otimes \varphi_1$, we have
\begin{equation}
  \label{eq:If01}
  \Ic_f(h_0, \varphi, \vc) = I_0(h_0, \varphi_0, \Theta_1(\cdot, \varphi_1, \vc)),
\end{equation}
where $\Theta_a$ (with $a = 1$) is defined in \eqref{eq:Theta}.
\subsection{Fourier Transform and Siegel-Weil Formula}
\label{subsec:FTSW}
We  follow \cite{GQT14} to recall the Siegel-Weil formula needed for our purpose, which goes from the split orthogonal group to the symplectic group.
The range we need is in the 1st term range, and was originally proved in \cite{KR94}.
Let $\varphi = \varphi_\infty\varphi_f \in \Sc(V(\Ab))$ with $\varphi_\infty$ as in \eqref{eq:varphiinf} above.
For $(g, h) \in G(\Ab) \times H(\Ab)$, we have the theta function $\theta(g, h, \varphi)$, and are interested in the value of the convergent integral $I(h, \varphi)$ defined in \eqref{eq:Ic}.

For a rational quadratic space $(V, (,)_V)$, suppose $V = U^+ + U^- + \Vo$ with $U^+, U^-$ complementary totally isotropic subspaces and $\Vo = (U^+ + U^-)^\perp$.
Let $W = X + Y$ denote the symplectic space of rank 2 over $\Qb$ with the symplectic pairing $\langle, \rangle_W$. The rational vector space $\Wb := V \otimes W$ is then a symplectic space with respect to the pairing
\begin{equation}
  \label{eq:pairWb}
  \langle v_1 \otimes w_1,  v_2 \otimes w_2 \rangle_\Wb :=
  (v_1, v_2)_V \langle w_1, w_2\rangle_W.
\end{equation}
From this, we have the Fourier transform $\Fcr_{U^+}: \Sc(V(\Ab)) \to \Sc(((U^- \otimes W)+ \Vo)(\Ab))$ defined by
\begin{equation}
  \label{eq:Fc}
  \Fcr_{U^+}(\varphi)(\eta, \vo) := \int_{U^+(\Ab)} \varphi(u^+, \eta_1, \vo)\psi((u^+, \eta_2)_V) du^+
\end{equation}
with $\eta = (\eta_1, \eta_2) \in (U^-)^2(\Ab) \cong (U^- \otimes W)(\Ab)$ and $\eta_i \in U^-(\Ab)$.
Here $du^+$ is the Usual Haar measure on $U^+(\Ab) = \Ab$.
Note that we have $(u^+, \eta_1, \vo) \in (U^+\otimes X + U^-\otimes X + \Vo)(\Ab) = V(\Ab)$.
Note that on $\Sc((U^- \otimes W + \Vo)(\Ab))$, the  Weil representation $\omega$ acts as
\begin{equation}
  \label{eq:WeilFourier}
  \begin{split}
  (  \omega(g, 1)\phi)(\eta, \vo) &= \omega_{\Vo}(g)\phi(\eta g, \vo), g \in G(\Ab),\\
  (  \omega(1, a)\phi)(\eta, \vo) &= |\det(a)| \phi(a^{-1}\eta, \vo), a \in \GL(U^+)(\Ab), \\
  (  \omega(1, u)\phi)(\eta, \vo) &= \psi(\langle u(\eta), \eta \rangle /2) \phi(\eta, \vo), u \in N(U^+)(\Ab) \subset \mathrm{Hom}_\Qb(U^-, U^+)(\Ab),
\end{split}
\end{equation}
which makes $\Fcr_{U^+}$ an intertwining map.

For $V$ in \eqref{eq:V}, we can take $U^\pm = V^\pm$ and $\Vo$ trivial with $V^\pm$ defined in \eqref{eq:Vpm}.
Another possibility is to take $U^\pm = \ell^\pm$ and $\Vo = V_1 \oplus V_{-1}$, which will be used in calculating the Fourier expansion of the theta integral $I_0$ in \eqref{eq:I0}.
To simplify notations, we write
\begin{equation}
  \label{eq:Fc}
  \Fcr := \Fcr_{V^+},~
  \Fcr_1 := \Fcr_{\ell^+},
\end{equation}
and use them to represent the Fourier transform at the finite and infinite places as well.
For example, $\Fcr_1$ is given by
\begin{equation}
  \label{eq:Fc1}
              \Fcr_1(\varphi)    ((\eta_1, \eta_2), \nu, \lambda)
= \int_{\Ab} \varphi(b, \eta_1, \nu, \lambda)\psi(b\eta_2) db
\end{equation}
for $\varphi \in \Sc(V(\Ab))$.
As $\Fcr_1$ acts as $\Fcr_1' \otimes \mathrm{id}$ on $\Sc(V(\Ab)) = \Sc(V_0(\Ab)) \otimes \Sc(V_1(\Ab))$, we will abuse notation and write
 $\Fcr_1 = \Fcr_1'$, which acts on $\Sc(V_0(\Ab))$.

For a place $v \le \infty$ of $\Qb$ and corresponding local field $k = \Qb_v$, recall we have the Siegel-Weil section
\begin{align*}
  \Phi_v: \Sc((V^-\otimes W)(k)) &\to I^H_v(0)\\
  \phi_v&\mapsto (h \mapsto (\omega_v(h)\phi_v)(0)),
\end{align*}
where $I^H_v(s) = \Ind^{H(k)}_{P(k)}(|\cdot|^s)$ is the degenerate principal series.
The image of $\Phi_v$ is a submodule of $I^H_v(0)$ denoted by $R_v(W)$.
When $v < \infty$, it is known that (see \cite[Proposition 5.2(ii)]{GQT14})
$$
I^H_v(0) = R_v(W) \oplus (R_v(W) \otimes {\det}_{H}).
$$
It is clear that
\begin{equation}
  \label{eq:Ginv}
  \Phi_v(\omega(g) \phi_v) = \Phi_v(\phi_v)
\end{equation}
for any $g \in G(k)$.

Given any $\phi = \otimes_v \phi_v \in \Sc((V^-\otimes W)(\Ab))$, we denote $\Phi_s(\phi) \in I^H(s)$ the standard section satisfying $\Phi_0(\phi) = \otimes_{v} \Phi_v(\phi_v)$.
We can then form the Eisenstein series
$$
E^H_P(s, \phi)(h) := \sum_{\gamma \in P(\Qb) \backslash H(\Qb)} \Phi_s(\phi)(\gamma h),
$$
which has meromorphic continuation to $s \in \Cb$ and is holomorphic at $s = 0$.
The regularized Siegel-Weil formula by Kudla-Rallis gives then the following equality (see \cite[Theorem 7.3(ii)]{GQT14})
\begin{equation}
  \label{eq:SW}
2  I(h, \varphi) = E^H_P(0, \Fcr(\varphi))(h).
\end{equation}
As a special case of the proposition in section 2 of \cite{Moeglin97a}, following an argument in \cite{GPSR}, we have the following lemma.
\begin{lemma}
  \label{lemma:Ecompare}
  For any $h \in H(\Ab)$, we have
  \begin{equation}
    \label{eq:Ecomp}
    E^H_P(s, \phi)(h) = \sum_{\gamma_0 \in B(F) \backslash G(F),~ \gamma_1 \in H_1(\Qb)}
    \Phi_s(\phi)((\gamma_0, \gamma_1 )h).
  \end{equation}
\end{lemma}
\begin{proof}
  We will show that
$ P(\Qb)\backslash H(\Qb)\cong (B(F)\backslash G(F) ) \times H_1(\Qb)$
  %
with the map induced by \eqref{eq:act}.
First, we have $P(\Qb)\backslash H(\Qb) = (P \cap (H_0 \times H_1))(\Qb)\backslash (H_0 \times H_1)(\Qb)$.
  Let $H_{-1} \subset H_0$ denote the image of $\SO(V_{-1})$, which is isomorphic to $H_1$, and $P_0 := P \cap H_0$. Then  $P \cap (H_0 \times H_1) = P_0P_1^\Delta$ with $P_1^\Delta \cong H_1$ the image of the diagonal embedding of $H_1$ into $H_{-1} \times H_1$.
From this,   we obtain
$$
  (P \cap (H_0 \times H_1)) (\Qb)\backslash (H_0\times H_1)(\Qb)
=  (P_0P_1^\Delta) (\Qb)\backslash (H_0\times H_1)(\Qb)
  = (((P_0 H_{-1}) \backslash H_0) \times H_1 )(\Qb).
  $$
  Equation \eqref{eq:BG1} then finishes the proof.
\end{proof}

Suppose $\vc = \otimes_{p \le \infty} \vc_p$ is  an odd character of $H_1(\Ab)/H_1(\Qb)$ and
\begin{equation}
  \label{eq:chi}
\chi:= \vc \circ \Nm^- = \otimes_{v \le \infty} \chi_v
\end{equation}
a totally odd character of $\Ab_F^\times/F^\times$, which can be viewed as a character on $B_0(\Ab)$.
Denote
\begin{equation}
  \label{eq:I0chi}
  I^{G_0}(\chi) := \mathrm{Ind}_{B_0(\Ab)}^{G_0(\Ab)} \chi,~
  I^{G_0}_{ p}(\chi_p) := \mathrm{Ind}_{B_0(\Qb_p)}^{G_0(\Qb_p)} \chi_p,~ \chi_p := \bigotimes_{v \mid p} \chi_v.
\end{equation}
From \eqref{eq:diag2}, we see that
\begin{equation}
  \label{eq:Irel}
  I^{G_0}(\chi) = I(0, \chi),~
  I^{G_0}_p(\chi_p) = \bigotimes_{v \mid p} I_{v}(0, \chi_v)
\end{equation}
with $I(s, \chi)$ and $I_v(s, \chi_v)$ defined in \eqref{eq:Ischi}.
Using the formula \eqref{eq:SW} and Lemma \ref{lemma:Ecompare},
we can rewrite the function $\Ic(g_0, \varphi, \vc)$ in \eqref{eq:Ic} as
\begin{align*}
  2\Ic(g_0, \varphi, \vc)
  &=
\mathrm{CT}_{s=0}    \int_{[H_1]} \vc(h_1) E^H_P(s, \Fcr(\varphi))(g_0, h_1) dh_1
=      E^{G_0}_{B_0}(0, F_{\varphi, \vc})(g_0),
\end{align*}
for $g_0 \in G_0(\A)$, where $E^{G_0}_{B_0}(s', F_{\varphi, \vc})$ is the Eisenstein series for the standard section associated to
\begin{equation}
  \label{eq:Fvarphi}
  \begin{split}
      F_{\varphi, \vc}(g_0) &:=       F_{\varphi, \vc, 0}(g_0)
      \in \Ind^{G_0}_{B_0} \chi,\\
      F_{\varphi, \vc, s}(g_0) &:= \int_{H_1(\Ab)}\Phi_s(\Fcr(\varphi))(g_0, h_1) \vc(h_1) dh_1.
  \end{split}
\end{equation}
Note that $F_{\varphi, \vc, s}$ is not a standard section, i.e.\ it depends on $s$ when restricted to any open compact subgroup of $G_0(\hat\Qb)$.

If $\vc = \otimes_{p \le \infty} \vc_p$ and $\varphi = \otimes_{p \le \infty} \varphi_p$, then $F_{\varphi, \vc, s}$ is a product of local integrals.
\begin{equation}
  \label{eq:Fvarphi}
F_{\varphi_p, \vc_p, s}(g_{0, p}) := \int_{H_1(\Qb_p)}\Phi_s(\Fcr(\varphi_p))(g_{0, p}, h_1) \vc(h_1) dh_1,~
F_{\varphi_p, \vc_p} := F_{\varphi_p, \vc_p, 0}
\in I_{0, p}(\chi_p).
\end{equation}
Recall that $dh_1$ is normalized so that the maximal compact subgroup of $H_1(\Qb_p)$ has volume 1.
We have explicitly
\begin{equation}
\label{eq:Fexp}
F_{\varphi_p, \vc_p}(g)
 =
 \int_{H_1(\Qb_p) \times F_p \times \Q_p} \varphi_p((g, h_1)^{-1} (x, 0, \lambda, \lambda)) \vc_p(h_1) dx\, d\lambda\, dh_1
\end{equation}
with $d\lambda$ the self-dual measure on $F_p$ such that $\int_{\Oc_{F_p}} d\lambda = |D|_p^{1/2}$. 
From this, we see that
\begin{equation}
  \label{eq:Galois}
  \sigma_a^{}(|D|_p^{1/2}F_{\varphi_p, \vc_p}(g)) =
|D|_{p}^{1/2} F_{  \sigma_a^{}(\varphi_p), \vc_p}(g),~
  F_{\varphi_p, \vc_p}(t_0g) =
F_{  \varphi_p, \vc_p}(g)
\end{equation}
for all $a \in \Zb_p^\times$ and $t_0 \in T_0(\Zb_p)$.
At all but finitely many cases, the function $F_{\varphi_p, \vc_p, s}$ is given explicitly as follows.
\begin{lemma}
\label{lemma:3}
Suppose $p$ is unramified in $E$ and $\varphi_p$ is the characteristic function of the maximal lattice
$V_{\Zb} \otimes \Zb_p \subset V_p$.
Then
\begin{equation}
  \label{eq:Fsval}
  F_{\varphi_p, \vc_p, s}(g_p)
=
(1-p^{-2-2s})\prod_{v \mid p}  L(1 + s, \chi_v)
\end{equation}
for all $g_p \in G_0(\Zb_p)$.
\end{lemma}

\begin{proof}
  Since $\varphi_p$ is $G_0(\Zb_p)$-invariant, we can suppose $g_p = 1$.

  If $p$ is inert in $F$,
  then
$$\Phi_s(\Fcr(\varphi_p))(1, h_1) = \Phi_0(\Fcr(\varphi_p))(1, h_1) = 1 = \vc(h_1)$$
 for all $h_1 \in H_1(\Qb_p) = H_1(\Zb_p) = \Oc_{F_p}^1 \subset \Oc_{F_p}^\times$, and $  F_{\varphi_p, \vc_p, s}(g_p) = \int_{H_1(\Zb_p)} dh_1 = 1$.

 If $p$ is split in $F$,
 we have
 $F_p \cong \Qb_p^2$, $H_1(\Qb_p) = \{(\alpha, \alpha^{-1}) \in F_p: \alpha \in \Qb_p^\times\} \cong \Qb_p^\times$
and $\chi_v = \chi_{v'}$ is a character of $\Qb_p^\times$.
 Straightforward (though involved) calculations show that
$$
\Phi_s(\Fcr(\varphi_p))(1, h_1) =
\Phi_0(\Fcr(\varphi_p))(1, h_1) \min\{|h_1|_v, |h_1|_{v'}\}^s .
$$
  For $h_1 = (\alpha, \alpha^{-1})$ with $o(\alpha) = m$, we have
  \begin{equation}
    \label{eq:int1}
      \begin{split}
        \Phi_0(\Fcr(\varphi_p))(1, h_1)
    &= \int_{\Qb_p^3} \cha(\Zb_p^6)(a, 0, \lambda_1, \lambda_2, \alpha^{-1} \lambda_1, \alpha \lambda_2) da d\lambda_1 d\lambda_2\\
    &=
      \int_{p^{\max\{0, m\}}\Zb_p}d\lambda_1
      \int_{p^{\max\{0, -m\}}\Zb_p}      d\lambda_2
      = p^{-|m|}.
  \end{split}
  \end{equation}
  Since $p$ is  unramified in $E$, we have $\vc_p((\alpha, \alpha^{-1})) = \e^{o(\alpha)}$ with $\e := \vc_p((p, p^{-1})) = \chi_v(p) = \chi_{v'}(p)$.
Putting these together then gives us
  \begin{equation}
    \label{eq:Fvs1}
      \begin{split}
    F_{\varphi_p, \vc_p, v, s}(1)
    &=\int_{\Qb_p^\times} \min\{|\alpha|_p, |\alpha^{-1}|_p\}^{1+s}
      \e^{o(\alpha)} |\alpha|_p^s d^\times \alpha\\
    &= \sum_{m \in \Zb} \e^m p^{-|m|(1+s)}
      = L(1 + s, \chi_v)L(1 + s, \chi_{v'})(1-p^{-2-2s}).
  \end{split}
  \end{equation}
This finishes the proof.
\end{proof}

\subsection{Matching Global Sections}
\label{subsec:global-match}
The function $\Ic(g_0, \varphi^{(k, k')}, \vc)$ is a Hilbert modular form of weight $(k, k')$. We want to suitably choose $\vc$ and $\varphi_f$  and compare this function to a coherent Eisenstein series.

Let $\chi = \chi_{E/F}$ be a Hecke character associated to a quadratic extension $E/F$ with $E/\Qb$ biquadratic,
and $\vc: \Ab_F^\times/F^\times \to \Cb^\times$ the character satisfying \eqref{eq:vc}, whose kernel  in $H_1(\hat\Zb)$ is denoted by $K_\vc$.
Let $\alpha \in F^\times, W_\alpha$ be the same as in section \ref{subsec:Eisenstein}.
For our purpose, we will choose $\phi^{(k, k')}_\infty \in \Sc(W_\alpha(F \otimes \Rb))$ to be eigenfunctions of $K_\infty = \SL_2(\Rb)^2$ with weight $(k, k')$ and normalized to have
$$
\phi^{(k, k')}_\infty(0) = 1.
$$
The matching result we will prove is the following.
\begin{theorem}
  \label{thm:match}
  For $\alpha \in F^\times$ with $\Nm(\alpha) < 0$, given any $\phi_f \in \Sc(\hat W_{\alpha})$, there exists $\varphi_f \in \Sc(\hat V; \Qb^\mathrm{ab})$ such that $\omega_f(-1)\varphi_f = - \varphi_f$ for $-1 \in H_1(\hat\Qb)$,
  it is invariant with respect to the compact subgroup $G(\hat\Zb) \TT(\hat\Zb)K_\vc \subset G(\Ab) \times H_{}(\Ab)$,
  and satisfies
  \begin{equation}
    \label{eq:match}
\fac    F_{\varphi, \vc} =
    2  \Lambda(1, \chi) \lambda_\alpha(\phi) \in I(0, \chi).
  \end{equation}
Here $\varphi = \varphi_f \varphi^{(\e, -\e)}_{\infty}$ with $\e := \sgn(\alpha_1) = -\sgn(\alpha_2)$ and  $\varphi^{(\pm 1, \mp 1)}_{ \infty}$ defined in \eqref{eq:varphiinf}, and $\phi = \phi_f \phi^{(\e, -\e)}_\infty$.
In particular, we have the equality
\begin{equation}
  \label{eq:Ematch}
\fac    \Ic(g, \varphi, \vc) =  E^*(g, \phi).
\end{equation}
\end{theorem}
\begin{remark}
  \label{rmk:Lambda1}
  The constants $\Lambda(0, \chi) = \Lambda(1, \chi) = \frac{\sqrt{D_E/D}}{\pi^2} L(1, \chi)$ and $\sqrt{D_E}$ are in $\Qb^\times$.
\end{remark}
\begin{remark}
  \label{rmk:match}
  For $L \subset W_{\alpha}(\hat\Qb)$ a lattice and $\mu \in L^\vee/L$,
  suppose $\varphi_\mu \in \Sc(\hat V; \Qab)$ satisfies \eqref{eq:match} with
   $\phi_f = \phi_{L + \mu}$.
  Then it is easy to see that
  $$
  \sum_{\mu \in L^\vee/L} \Ic(g^\Delta_\tau, \varphi_\mu, \vc) \ef_\mu: \Hb \to \Cb[L^\vee/L]
  $$ is a (non-holomorphic) vector-valued modular form of weight 0 on $\SL_2(\Zb)$ with representation $\overline{\rho_L}$.
\end{remark}

\begin{remark}
  \label{rmk:decompose}
  If we decompose $V_{0} = U \oplus U^\perp$ with $U = \ell^+ + \ell^-$ the hyperbolic plane, then it is easy to see that $T_0 \subset \SO(U) \subset  H_0$. Therefore for any $\varphi \in \Sc(\hat V; \Qab)^{\TT(\hat\Zb)}$, we can write it as
  $$
\varphi = \sum_{j \in J} \varphi_{U, j} \otimes \varphi_{U^\perp, j}
$$
such that $\varphi_{U, j} \in \Sc(\hat U; \Qab)^{\TT(\hat\Zb)}$ and  $\varphi_{U^\perp, j} \in \Sc(\hat U^\perp; \Qab)^{\TT(\hat\Zb)}$ for all $j \in J$.
This in particular implies that $\varphi_{U^\perp, j}$ is $T(\hat\Zb)$-invariant, i.e.\ it is $\Qb$-valued by \eqref{eq:Galinv}.
\end{remark}
\begin{proof}[Proof of Theorem \ref{thm:match}]
  Suppose $\phi = \otimes_{v \le \infty}\phi_v$.
  By Theorem \ref{thm:localmatch}, there exists $\varphi_p \in \Sc(V_p; \Qip)^{}$ invariant with respect to $G(\Zb_p) \TT(\Zb_p)$ and satisfying \eqref{eq:matchp}.
Furthermore, $\varphi_p$ is the characteristic function of the maximal lattice in $V_p$ for all but finitely many $p$.
Therefore $\varphi_f := \bigotimes_{p < \infty} \varphi_p$ is in $\Sc(\hat V; \Qb^{\mathrm{ab}})^{(G\cdot\TT)(\hat\Zb)}$, and satisfies
$$
F_{\varphi_f, \vc_f} = \zeta(2)^{-1} {L(1, \chi)}{\sqrt{D_E/D}} \lambda_\alpha(\phi_f) = 6 \Lambda(1, \chi) \lambda_\alpha(\phi_f).
$$
Since $\vc_f(-1) = \sgn(-1) = -1$, the function $\omega_f(-1)\varphi_f$ with $-1 \in H_1(\hat\Qb)$ also satisfy these conditions, we can replace $\varphi_f$ by $(\varphi_f - \omega_f(-1)\varphi_f)/2$ so that $\omega_f(-1) \varphi_f = - \varphi_f$.
Furthermore, we have $F_{\omega_f(h) \varphi_f, \vc_f} = F_{\varphi_f, \vc_f}$ for all $h \in K_\vc$, and can therefore average over $K_\vc$ to ensure that $\varphi_f$ is $K_\vc$-invariant.

To prove \eqref{eq:Ematch}, it suffices to check that $F_{\varphi^{(\e, -\e)}_\infty, \vc_\infty}(g) = \pi^{-1} \lambda_\alpha(\phi_\infty^{(\e, -\e)})(g)$
for $g = (g_{\tau_1}, g_{\tau_2})$. 
Using
  \begin{align*}
    \Phi_\infty(\Fcr(\varphi^{(\e, -\e)}_\infty))(g, t)
    &=
      \Fcr(\omega(g, t)\varphi^{(\e, -\e)}_\infty)(0)
    =
      \int_{\Rb^3} (\omega(g, t)\varphi^{(\e, -\e)}_\infty)(a, 0, \lambda_1, \lambda_2, \lambda_1, \lambda_2) da d\lambda_1 d\lambda_2   \\
    &=
      \int_{\Rb^3} \varphi^{(\e, -\e)}_\infty
      \lp \frac{a - \lambda_1 u_1 - \lambda_2 u_2}{\sqrt{v_1v_2}}, 0, v_1\lambda_1/\sqrt{v_1v_2}, v_2\lambda_2/\sqrt{v_1v_2}, t^{-1}\lambda_1, t\lambda_2\rp da d\lambda_1 d\lambda_2   \\
    &=
      \int_{\Rb^2}
      (v_1\lambda_1 - v_2\lambda_2) (t^{-1}\lambda_1 - t^{}\lambda_2)
e^{-\frac{\pi}{v_1v_2} ((v_1\lambda_1 - v_2\lambda_2)^2 + v_1v_2((t^{-1}\lambda_1 + t\lambda_2)^2))}
      d\lambda_1 d\lambda_2\\
    &=
      \int_{\Rb^2}
      x
\frac{2x + (t^{-1}v_2 - tv_1)y}{v_1t + v_2t^{-1}}
e^{-\frac{\pi}{v_1v_2} (x^2 + v_1v_2y^2))}
      \frac{dx dy}{v_1t + v_2t^{-1}}   \\
&=
\pi^{-1} 2 {(v_1v_2)^{3/2}}{}
      \frac{t^2}{(v_1t^2 + v_2)^2},
  \end{align*}
where we have  used the change of variable $x = v_1\lambda_1 - v_2\lambda_2, y = t^{-1} \lambda_1 + t \lambda_2$,
 we obtain
  \begin{align*}
    F_{\varphi^{(\e, -\e)}_\infty, \vc_\infty}(g) &=    \int_{0}^\infty
    \Phi_\infty(\Fcr(\varphi_\infty))(g, t)   \frac{dt}{t}
=\pi^{-1} 2 {(v_1v_2)^{3/2}}{}
   \int_{0}^\infty
      \frac{t dt}{(v_1t^2 + v_2)^2}
=\pi^{-1}  \sqrt{ v_1v_2}.
  \end{align*}
  On the other hand, we have
  $$
  \lambda_{\alpha}(\phi^{(\e, -\e)}_\infty)(g)
%
  = \sqrt{v_1v_2}.
  $$
  This finishes the proof.
\end{proof}

The requirement that $\varphi_f$ in Theorem \ref{thm:match} is invariant with respect to $\TT(\hat\Zb)$ will be important to deduce important rationality results in section \ref{subsec:Millson}. We give a taste of such results in the following lemma.

\begin{lemma}
    \label{lemma:hatvarphirat}
    If $\varphi_0 \in \Sc(V_0; \Qab)$ is invariant with respect to $\TT(\hat\Zb) \subset (\GL_2 \times H_0)(\hat\Zb)$, then $\Fcr_1(\varphi_0) \in \Sc(((\ell^- \otimes W) + V_{-1})(\hat\Qb); \Qab)$ satisfies
    \begin{equation}
      \label{eq:Galois0}
      \sigma_a      \lp \Fcr_1(\varphi_0)    ((\eta_1, \eta_2), \nu) \rp
      =  \Fcr_1(\varphi_0)    ((a^{-1} \eta_1, \eta_2), \nu)
\end{equation}
for any $\sigma_a \in \mathrm{Gal}(\Qab/\Qb)$ associated to $a \in \hat\Zb^\times$ as in section \ref{subsec:WeilRep}.
In particular, we have
\begin{equation}
  \label{eq:Fc1rat}
       \Fcr_1(\varphi_0)    ((0, r), \nu) \in\Qb
\end{equation}
 for all $r \in \hat\Qb, \nu \in \hat F$.
\end{lemma}

\begin{proof}
  Using the expression for $\Fcr_1$ in \eqref{eq:Fc1}, we can write
  \begin{align*}
\sigma_a        &   \lp \Fcr_1(\varphi_0)    ((\eta_1, \eta_2), \nu) \rp
    = \sigma_a \lp \int_{\hat\Qb} \varphi_0(b, \eta_1, \nu)\psi_f(b\eta_2) db \rp
      = \int_{\hat\Qb} \sigma_a(\varphi_0(b, \eta_1, \nu))\psi_f(a b\eta_2) db\\
    &= \int_{\hat\Qb} \omega((t(a), 1))(\varphi_0)(b, \eta_1, \nu)\psi_f(a b\eta_2) db
= \int_{\hat\Qb} \omega((1, \iota(t(a^{-1}))))(\varphi_0)(b, \eta_1, \nu)\psi_f(a b\eta_2) db\\
    &= \int_{\hat\Qb} \varphi_0(ab, a^{-1}\eta_1, \nu)\psi_f(a b\eta_2) db =
       \Fcr_1(\varphi_0)    (( a^{-1}\eta_1,\eta_2), \nu).
  \end{align*}
  For the second step, we moved $\sigma_a$ inside the integral as $\varphi_0$ is a Schwartz function and the integral is a finite sum. The third and fourth steps used \eqref{eq:weilGL2} and the invariance of $\varphi_0$ under $(t, \iota(t)) \in \TT(\hat\Zb)$ respectively.
  Equation \eqref{eq:Fc1rat} now follows from \eqref{eq:Galois0} via \eqref{eq:weilGL2}.
\end{proof}
\subsection{Matching Local Sections I}
\label{subsec:local-match-I}
The goal of this section is to prove Theorem \ref{thm:localmatch}, the non-archimedean local counterpart of the matching result \ref{thm:match}.
For this purpose, we  fix a prime $p < \infty$ throughout this section.
The main input to Theorem \ref{thm:localmatch} is the following surjectivity result.

\begin{proposition}
  \label{prop:surj}
Let $\vc_p$ and $\chi_p$ be as in \eqref{eq:chi}.
Then  the following map
\begin{equation}
  \label{eq:beta}
  \begin{split}
    \beta: \Sc(V_{p}; \Cb)^{G(\Zb_p)} \subset \Sc(V_p; \Cb) &\to I_p^{G_0}(\chi_p)\\
    \varphi &\mapsto F_{\varphi, \vc_p}
  \end{split}
\end{equation}
is surjective.
Furthermore if $\Phi \in I_p^{G_0}(\chi_p)$ is valued in $\Qip$, then there exists $\varphi \in \Sc(V_p; \Qip)^{G(\Zb_p)}$ satisfying $\beta(\varphi) = \Phi$.
Here $\Qip \subset \Qab$ is the subfield defined in \eqref{eq:Qinfp}.
\end{proposition}

\begin{proof}
  Using \eqref{eq:Irel}, we can suppose $\Phi = \otimes_{v \mid p} \Phi_v$ with $\Phi_v \in I_v(0, \chi_v)$.
  Since $F_{\omega(g)\varphi, \vc_p} = F_{\varphi, \vc_p}$ for all $g \in G(\Zb_p)$
  and $\varphi \in \Sc(V_p; \Cb)$, it suffices to prove the surjectivity of $\beta$ on $\Sc(V_p; \Cb)$.
  To do this, we will use the  $m$-th Fourier coefficient of $\Phi_v \in I_v(0, \chi_v)$ for $m \in F_v$, which is defined by
\begin{equation}
  \label{eq:WmPhi}
  W_m(\Phi_v)
 := \int_{F_v} \Phi_v(w n(b)) \psi_v(-mb) db
\end{equation}
with $\psi_v$ an additive character of $F_v$.
For $m = (m_v)_{v \mid p} \in F_p$ and $\varphi \in \Sc(V_p; \Cb)$, we denote
\begin{equation}
  \label{eq:Wm}
  \begin{split}
    W_m(\varphi)
&:=
\prod_{v \mid p} W_{m_v}( (F_{\varphi, \vc_p})_v)\\
&=
 \int_{F_p\times H_1(\Qb_p)}
(\omega_p(wn(b), h_1) \Fcr(\varphi))(0)\psi_p(-mb) \vc_p(h_1) dh_1 \, db
  \end{split}
\end{equation}
with $\psi_p := \prod_{v \mid p} \psi_v$.
Now, the  $G(F_v)$-module $I_v(0, \chi_v)$ can be written as
$$
I_v(0, \chi_v) = \oplus_{\alpha \in F_v^\times / \Nm(E_v^\times)} R(W_\alpha)
$$
with $R(W_\alpha)$ the image of $\lambda_{\alpha, v}$ and irreducible.
So $I_v(0, \chi_v)$ is irreducible if and only if $\chi_v$ is trivial. Otherwise, it is the direct sum of two irreducible submodules.
Furthermore for $\Phi \in R(W_\alpha)$, the coefficient $W_m(\Phi)$ is zero unless $m/\alpha \in \norm_{E_{v}/F_{v}}E_{v}^\times$.
We then have two cases to consider, depending on whether $\chi_v = \chi_{v'}$ is trivial or not.

When $\chi_v = \chi_{v'}$ is trivial, Lemma \ref{lemma:calc} gives us $\varphi$ such  that $W_m( \varphi) \neq 0$ for some $m \in F_p^\times$.
So for $v \mid p$, the restriction of $\mathrm{im}(\beta) \subset I_p^{G_0}(\chi_p)$ to $G(F_v)$ gives a non-zero section in $I_v(0, \chi_v)$, and generates  a non-trivial, irreducible sub $G(F_v)$-module.
As $I_v(0, \chi_v)$ is irreducible, the map $\beta$ is surjective.
When $\chi_v = \chi_{v'}$ is non-trivial, we again apply Lemma \ref{lemma:calc} to obtain a submodule $R \subset I_v(0, \chi_v) = R(W_{\alpha_0}) \oplus R(W_{\alpha_1})$ from $\mathrm{im}(\beta)$ such that $\pi_i(R)$ is non-trivial with $\pi_i: I_v(0, \chi_v) \to R(W_{\alpha_i})$ the projection.
As $R(W_{\alpha_i})$ is irreducible, we have $\pi_i(R) = \pi_i(I_v(0, \chi_v))$.
Consider $R_i := \ker \pi_i \cap R$ as a submodule of the irreducible module $\ker \pi_i$.
As $R(W_{\alpha_0})$ and $R(W_{\alpha_1})$ are not isomorphic \cite[Proposition 3.4]{KR92}, $R_i$ cannot be trivial for both $i = 0, 1$, otherwise  $R \cong \pi_i(R) = R(W_{\alpha_i})$.
Thus $R_i = \ker \pi_i \subset R$ for an $i$, which implies $R = I_v(0, \chi_v)$ and proves surjectivity.

When $\Phi = \otimes_{v \mid p} \Phi_v$ has value in $\Qip$, we apply the surjectivity of $\beta$ and the discussion in section \ref{subsec:WeilRep} to choose $\varphi_j \in \Sc(V_p; \Qip)^{G(\Zb_p)}$ and $c_j \in \Cb$ such that
$$
\varphi := \sum_{j = 1}^J c_j \varphi_j \in \Sc(V_p; \Cb)
$$
satisfies $\beta(\varphi) = \Phi$ and $J$ is minimal.
Therefore, $F_{\varphi, \vc_p} = \sum_{j = 1}^J c_j F_{\varphi_j, \vc_p}$ is valued in $\Qip$.
By the minimality of $J$, the section $F_{\varphi_j, \vc_p}$ is not identically zero for all $j$.
Therefore, the set $\{1, c_1, \cdots, c_J\} \subset \Cb$ is linearly dependent over $\Qip$.
The minimality of $J$ then implies that $J = 1$ and $c_1 \in \Qip$, hence $\varphi \in \Sc(V_p; \Qip)^{G(\Zb_p)}$.
\end{proof}

Using this proposition, we can match any continuous function on $G_0(\Zb_p)$ via the map $\beta$. Furthermore, we can incorporate Galois action to obtain the following result.

\begin{proposition}
  \label{prop:extend}
  In the setting of Proposition \ref{prop:surj}, given any continuous function $\Phi: G_0(\Zb_p) \to \Cb$ satisfying
  \begin{equation}
    \label{eq:cond3}
    \Phi(m(a)n(b)k) = \chi(a) \Phi(k),~
  \end{equation}
  for all $m(a), n(b) \in B_0(\Zb_p), k \in G_0(\Zb_p)$,
     there exists $\varphi \in \Sc(V_p; \Cb)^{G(\Zb_p)}$ such that $F_{\varphi, \vc_p}(g) = \Phi(g)$ for all $g \in G_0(\Zb_p)$.
  Furthermore, if $\Phi$ takes values in $\Qip$ and satisfies
  \begin{equation}
    \label{eq:Galois1}
    \sigma_a(|D|_p^{-1/2} \Phi(t_0^{-1} g t_0)) = |D|_p^{-1/2} \Phi(g),~
  \end{equation}
with $    t_0 = \iota(t(a)) \in \tH_0(\Zb_p), t(a) = \smat{a}{}{}{1} \in T \subset \GL_2(\Zb_p)$
  for all $a \in \Zb_p^\times$ and $g \in G_0(\Zb_p)$, then $\varphi \in \Sc(V; \Qip)^{G(\Zb_p)}$ can be chosen to be $\TT(\Zb_p)$-invariant.
\end{proposition}

\begin{proof}
  A continuous function $\Phi$ on $G_0(\Zb_p)$ satisfying \eqref{eq:cond3} can be uniquely extended to a section $\tilde\Phi \in I^{G_0}_p(\chi_p)$ by setting%
  $$
\tilde\Phi(g) := \chi_p(a) \Phi(k)
$$
with $g = m(a)n(b)k$ the Iwasawa decomposition of $g$.
Therefore the first claim is a direct consequence of Proposition \ref{prop:surj}.

For the second claim, we take any $\varphi \in \Sc(V_p; \Qip)^{G(\Zb_p)}$ and observe that
\begin{align*}
  \frac{F_{\omega_p(t, t_0)\varphi_p, \vc_p}(g)}{  |D|_p^{1/2}}
  &=   \frac{ F_{\omega_p(t)\varphi_p, \vc_p}(gt_0)}{|D|_p^{1/2} }
    =   \frac{F_{\sigma_a(\varphi_p), \vc_p}(t^{-1}gt_0)}{|D|_p^{1/2} }
    =  \sigma_a^{}\lp \frac{ F_{\varphi_p, \vc_p}(t^{-1}gt)}{ |D|_p^{1/2}}\rp\\
 & =    \sigma_a(|D|_p^{-1/2} \Phi(t_0^{-1} g t_0)) = |D|_p^{-1/2} \Phi(g)
\end{align*}
for any $(t, t_0) \in \TT(\Zb_p)$ with $t = t(a), t_0 = \iota(t(a))$ and $g \in G_0(\Zb_p)$.
Here we used \eqref{eq:Galois} for the first line.
By averaging $\varphi$ over $\TT(\Zb_p)$, we can suppose that it is $\TT(\Zb_p)$-invariant. This finishes the proof.
\end{proof}

We are now ready to state and prove the local matching result.
This is just the Kudla matching principle \cite{Kudla03} in some sense.
\begin{theorem}
  \label{thm:localmatch}
  For any $\phi_v \in \Sc(W_\alpha(F_v))$ with $v \mid p$, %
   there exists $\varphi_p \in \Sc(V_p; \Qip)^{(G\cdot \TT)(\Zb_p)}$ such that
  \begin{equation}
  \label{eq:matchp}
F_{\varphi_p, \vc_p} = (1 - p^{-2})|D/D_E|_p^{1/2} \prod_{v \mid p}  L(1, \chi_v) \lambda_{\alpha, v}(\phi_v).
\end{equation}
In addition, if $p$ is unramified in $E$ and co-prime to $\alpha$, and  $\phi_v$ is the characteristic function of the maximal lattice in $W_\alpha(F_v)$, then we can choose $\varphi_p$ to be the  characteristic function of the maximal lattice in $V_p$.
\end{theorem}
\begin{proof}
  Suppose $p$ and $\alpha$ are co-prime, $E_p/\Qb_p$ is unramified and $\phi_v = \cha(\Oc_{E_v}), \varphi_p = \cha(\Oc_{F_p} \times \Zb_p^2 \times \Oc_{F_p})$.
  Then it is easy to check that $F_{\varphi_p, \vc_p}$ and $\prod_{v \mid p} \lambda_{\alpha, v}(\phi_v)$ are both right $G(\Oc_{F_p})$-invariant.
  Since they are both in $\prod_{v \mid p} I(0, \chi_v)$, we only need to check that
  $$
F_{\varphi_p, \vc_p}(1) = (1-p^{-2})\prod_{v \mid p} L(1, \chi_v) \lambda_{\alpha, v}(\phi_v)(1)
  $$
  by the Iwasawa decomposition of $G(F_p)$.
This is given precisely by Lemma \ref{lemma:3}, and proves \eqref{eq:matchp} for all but finitely many places.

When $\phi_v$ is $\Qb$-valued, we can use \eqref{eq:weilGL2} to check that
\begin{align*}
  \sigma_a^{}\lp  \prod_{v \mid p} \lambda_{\alpha, v}(\phi_v)(t^{-1}gt)\rp
  &=   \prod_{v \mid p}
    \sigma_a^{} \lp\omega_{\alpha, v}(t^{-1}gt)(\phi_v)(0)\rp
   =  \prod_{v \mid p}
    \lp\omega_{\alpha, v}(g)(\sigma_a(\phi_v))(0)\rp\\
 &  =  \prod_{v \mid p}
    \lp\omega_{\alpha, v}(g)(\phi_v)(0)\rp
  =   \prod_{v \mid p}
    \lambda_{\alpha, v}(\phi_v)(g)
\end{align*}
for any $(t, t_0) \in \TT(\Zb_p)$ with $t = t(a), t_0 = \iota(t(a))$ and $g \in G_0(\Zb_p)$.
Proposition \ref{prop:extend} combined with Remark \ref{rmk:Lambda1} then completes the proof.
\end{proof}

Finally, we record the two  local calculation lemmas used in proving Proposition \ref{prop:surj}.

\begin{lemma} \label{Fourier}
  Suppose $F_p/\Q_p$ is non-split with valuation ring $\Oc_p$, uniformizer  $\varpi$, residue field size $q$, and a non-trivial additive character $\psi$.
  For a   character  $\vc$    of $H_1(\Qb_p) = F_p^1 \subset\Oc_p^\times$, let
$$
n(\psi) :=\min\{n: \psi(\varpi^n \Oc_p) =1\},~
n(\vc) :=\min\{n\ge 0: \vc(K_n) =1\},
$$
be the conductors of $\psi$ and $\vc$ respectively,
where  $K_n:=F_p^1 \cap (1+ \varpi^n \Oc_p)$.
 Then
$$
\int_{F_p^1} \vc(x) \psi(mx) dx \ne 0
$$
for some $m \in \Oc_p^\times$ if and only if $n(\vc) \le n(\psi)$.
\end{lemma}
\begin{proof}
  Let
$$
f(m) =\begin{cases}
     \int_{F_p^1} \vc(x) \psi(mx) dx &\fff m \in \Oc_p^\times
     \\
      0  &\hbox{otherwise} .
      \end{cases}
$$
Then $f \in \Sc(F_p)$ and its Fourier transformation with respect to $\psi$ is
\begin{align*}
\hat f(m) &= \int_{F_p} f(n) \psi(-nm) dn
=\int_{F_p^1} \vc(x) \int_{\Oc_p^\times} \psi(n(x-m) dn  dx
 \\
 &=\int_{F_p^1} \vc(x) (\cha(m+\varpi^{n(\psi)}\Oc_p)(x)  -q^{-1}\cha(m+\varpi^{n(\psi)-1}\Oc_p)(x) ) dx.
\end{align*}
First assume that there is some $h_0\in F_p^1$ such that $h_0-m \in \varpi^{n(\psi)}\Oc_p$.  Then
\begin{align*}
\hat f(m) &= \vc(h_0) (\int_{K_{n(\psi)}} \vc(x) dx -q^{-1}  \int_{K_{n(\psi)-1}} \vc(x) dx)
\\
 &=\begin{cases}
   0 &\fff n(\vc) > n(\psi),
   \\
   \vc(h_0)  \vol(K_{n(\psi)}) &\fff n(\vc) = n(\psi),
   \\
   \vc(h_0)(\vol(K_{n(\psi)}) - q^{-1} K_{n(\psi)-1}) &\fff  n(\vc) < n(\psi).
   \end{cases}
\end{align*}
Next we assume that there no $h_0\in F_p^1$ such that $h_0-m \in \varpi^{n(\psi)}\Oc_p$ but some $h_0 \in F_p^1$ with $h_0-m \in \varpi^{n(\psi)-1}\Oc_p$. Then
$$
\hat f(m) = \begin{cases}
   0 &\fff n(\vc) \ge  n(\psi),
   \\
   -q^{-1} \vc(h_0)\vol(K_{n(\psi)-1}) &\fff n(\vc) < n(\psi).
   \end{cases}
$$
Finally, if there is no $h_0 \in F_p^1$ with $h_0-m \in \varpi^{n(\psi)-1}\Oc_p$, then $\hat f (m) =0$.  Now the lemma is clear.
\end{proof}

\begin{lemma}
  \label{lemma:calc}
  When $\chi_v$ is trivial, there exists $\phi \in \Sc(V_p)$ such that $F_{\phi, \vc_p}$ is non-trivial.
  When $\chi_v$ is non-trivial,
    then for any $\e = (\e_v)_{v \mid p}$ with $\e_v = \pm 1$
  there exists $\phi^\e \in \Sc(V_p)$ and $m^\e \in F_p^\times$ such that $W_{m^\e}(\phi^\e) \neq 0$ and $m^\e = (m^{\e_v}_v)_{v \mid p}$ with $\chi_v(m^{\e_v}_v) = \e_v$.
\end{lemma}

\begin{proof}
  When $\chi_v$ is trivial, the character $\vc_p$ of $H_1(\Qb_p)$ is also trivial. Suppose $\phi$ is the characteristic function of the maximal lattice in $V_p$, then integral in \eqref{eq:Fexp} is positive at $g = 1$, which means $F_{\phi, \vc_p}$ is non-trivial.

  Suppose now that $\chi_v$, hence $\vc_p$, is non-trivial.
  We can suppose that $n(\psi) = 0$.
When $\phi = \phi_0 \otimes \phi_1$ with $\phi_i \in \Sc(V_{i, p})$, we can apply \eqref{eq:Fexp} to write
  \begin{equation*}
      W_m(\phi)
 = \int_{F_p\times H_1(\Qb_p) \times F_p \times \Q_p}
 \phi_0((wn(b))^{-1}\cdot \kzxz {x} {\lambda} {\lambda'} {0})
 \phi_1(h_1^{-1} \lambda)
 \psi_p(-mb) \vc_p(h_1) dx\, d\lambda\, dh_1 \, db.
\end{equation*}
We first assume  that $p$ is non-split and use the  notation in Lemma \ref{Fourier}.
In this case, there is a unique place $v$ of $F$ above $p$, and $\e = \pm 1$.
For $n \ge \max\{n(\vc)+1, 1\}$ and  $\beta \in \Oc_p^\times$, let
$
\phi_1 = \phi_{1, \beta} = \cha(\beta+p^n \OO_p)
$
and
$$
\phi_0=\cha \kzxz {\Z_p} {\varpi^n\OO_p} {\varpi^n\OO_p} {1+ p^n\Z_p}.
$$
Then
\begin{align*}
&W_m(\phi_0 \otimes \phi_1)
=\int_{F_p\times F_p^1 \times F_p \times \Q_p}
\phi_0 \kzxz {b\lambda  + b'\lambda'  + b b' x} {-\lambda'-b x} {-\lambda-b'x} {x}
 \phi_1(h^{-1} \lambda)
                 \vc(h)\psi(-m b) dx\, d\lambda\, dh\, db\\
  &=\int_{F_p^1} c(h) \vc(h) dh,
\end{align*}
where
\begin{align*}
c(h)&:=\int_{ \beta h + \varpi^n \OO_p } \int_{-\lambda'+\varpi^n \OO_{\mathfrak p}}  \int_{1+p^n \Z_p} \psi(-m b) dx\, db\, d\lambda\\
&=p^{-n(1+f)}\cha(\varpi^{-n} \OO_p)(m) \int_{\beta h+ \varpi^n \OO_p} \psi(m\lambda')  d\lambda
=C \cha(\varpi^{-n} \OO_p)(m) \psi(m\beta' h^{-1})
\end{align*}
for some nonzero constant $C$. Here $f=1$ or $2$ depending on whether $F/\Q$ is ramified or inert at $p$.
Using $\vc(h) = \vc(h^{-1})$, we have
\begin{equation} \label{eq:Whittaker}
W_m(\phi_0 \otimes \phi_1) = C \cha(\varpi^{-n} \OO_p)(m) \int_{F_p^1} \vc(h) \psi(m\beta' h) dh.
\end{equation}
If  $F_p/\Qb_p$ is inert, then $\chi_p$ is ramified and non-trivial when restricted to $\Oc_p^\times$.
By Lemma \ref{Fourier}, we can find $m_0 \in \varpi^{ - n} \Oc_p^\times$ such that
$$
W_{m_0/\beta'}(  \phi_0 \otimes \phi_{1, \beta})
= C \int_{F_p^1} \vc(h) \psi(m_0 h) dh
\ne 0
$$
as $ n(\psi(\varpi^{ - n} \cdot)) = n  \ge n(\vc)$.
We can choose $\beta = \beta^\pm$ such that $\chi(m_0/(\beta^\pm)') = \pm 1$.
Then taking $\phi^{\pm} = \phi_{0} \otimes \phi_{1, \beta^\pm}$ proves the Lemma.
If $F_p/Q_p$ is ramified, then $E_v/F_v$ is inert and $\chi_v(\varpi) = -1, \chi_v \mid_{\Oc_p^\times} = 1$.
Again by Lemma \ref{Fourier}, we can find  $m_j \in \varpi^{- n + j} \Oc_p^\times$ for $j = 0, 1$ such that
$$
W_{m_j}(  \phi_0 \otimes \phi_{1, 1})
= C \int_{F_p^1} \vc(h) \psi(m_j h) dh
\ne 0
$$
as $ n(\psi(m_j \cdot)) = n - j  \ge n(\vc)$.
Therefore, $\phi^\pm = \phi_0 \otimes \phi_{1, 1}$ satisfies the Lemma.

Finally, we come to the case when $p = v_1v_2$ splits and  $\eta := \chi_{v_1} = \chi_{v_2}$ is non-trivial.
In this case,
$F_p=F_{v_1} \times F_{v_2} =\Qb_p^2$ and $\eta = \vc_p$ is a character of  $\Qb_p^\times \cong H_1(\Qb_p)$.
For $m \in F_p$, we write $m=(m_1, m_2)$ with $m_j \in \Qb_p$ and
\begin{align*}
  V_{0, p} &\cong M_2(\Q_p) \\
  \smat{a}{\lambda}{\lambda'}{b} &\mapsto
  \smat{a}{\lambda_1}{\lambda_2}{b}
\end{align*}
So we take $\phi_1 = \phi_{1, 1} \otimes \phi_{1,2}$ with $\phi_{1, j} \in \Sc(\Qb_p)$ and $\phi_{0} \in \Sc(M_2(\Q_p))$.
Simple calculation gives us
\begin{align*}
W_{m}(\phi_0 \otimes \phi_1)
=
  & \int_{\Q_p^2 \times \Q_p^\times \times \Q_p^2 \times \Q_p}
    \phi_{1, 1}( h^{-1}\lambda_1) \phi_{1, 2}(h \lambda_2)
    \phi_0 \kzxz { b_1 \lambda_1 + b_2 \lambda_2 + b_1 b_2 x} {-\lambda_2  - b_1 x} {-\lambda_1 - b_2 x} {x}
 \\
  &\quad \cdot \eta(h) \psi(-m_1b_1 -m_2 b_2) dx \, d\lambda_1\, d\lambda_2 \, {d^\times h} \, db_1 \, db_2.
\end{align*}
Taking $\phi_{1, j} =\cha(1+p^n\Z_p)$, $n \ge\max\{ 1, n(\eta)\}$,  and
$$
\phi_{0} =\cha (\kzxz {\Z_p} {p^n\Z_p} {p^n\Z_p} {1+ p^n\Z_p}),
$$
The same calculation as above gives
\begin{align*}
W_{m}(\phi_1 \otimes \phi_0)
  &= C \int_{\Q_p^\times} \eta(x) \psi(m_2 x + m_1 x^{-1})  \cha( p^{-n}\Z_p^2)(m_1, m_2)
   \cha( p^{-n}\Z_p^2)(m_1 x^{-1}, m_2 x)   \frac{dx}{|x|}
\end{align*}
for some nonzero constant $C$.

When  $\eta$ is ramified, we take $n=n(\rho)$, $m_1=p^l$  with  $l \ge n$ and $m_2 =m_0 p^{-n}\in  p^{-n}\Z_p^\times$ and obtain (write $o(x) =\hbox{ord}_{p}x$)
\begin{align*}
W_{m}(\phi_1 \otimes \phi_0)
  &= C \int_{0 \le o(x) \le n+l} \eta(x) \psi(m_2 x) \psi(m_1 x^{-1}) \frac{dx}{|x|}
  \\
  &= C \sum_{ 0 \le i \le n}\int_{p^i \Z_p^\times} \eta(x) \psi(m_2 x) \frac{dx}{|x|}
    + C \sum_{1 \le i \le l} \int_{ p^{-n-i} \Z_p^\times} \eta(x)^{-1} \psi( m_1 x)  \frac{dx}{|x|}
\\
  &=C \left[\eta(m_0)^{-1} \int_{\Z_p^\times}\eta(x) \psi(p^{-n} x) dx +\eta(p)^{n+l} \int_{\Z_p^\times} \eta^{-1}(x) \psi(p^{-n} x) dx\right]
  \\
   &\ne 0,
\end{align*}
for some $m_2$.

When $\eta$ is unramified, we have $n(\eta) = 0$ and take $n = 1$.
For $m_j \in  p^{-1} \Z_p^\times$, $j=1, 2$, we have
\begin{align*}
&W_{(m_1, m_2)}(\phi_0 \otimes \phi_1)
= C \int_{\Z_p^\times} \psi(m_2 x + m_1 x^{-1}) dx.
\end{align*}
If we sum this over $m_j \in p^{-1}\Zb_p^\times$, then the result is non-zero.
So there exists $m_j \in p^{-1}\Zb_p^\times$ such that $W_{(m_1, m_2)}(\phi_0 \otimes \phi_1) \neq 0$.
For  $m_j \in \Z_p^\times$, $j=1, 2$, we have
\begin{align*}
W_{(m_1, m_2)}(\phi_0 \otimes \phi_1)
&= C \left[ \int_{\Z_p^\times} dx + \eta(p)  \int_{p\Z_P^\times} \psi(m_1 x^{-1}) \frac{dx}{|x|}
   + \eta(p)^{-1} \int_{p^{-1} \Z_p^\times} \psi(m_2 x) \frac{dx}{|x|}\right]
   \\
   &= C(1-p^{-1} +p^{-1} + p^{-1}) \ne 0,
   \end{align*}
   as $\eta(p) =-1$.
Replacing $\phi_1$ by $\phi_1'=\cha(1+p^n \Z_p, p + p^n \Z_p)$, the same calculation gives
$$
W_{(m_1, m_2)}(\phi_0 \otimes \phi_1')\ne 0
$$
when $m_j \in \Z_p^\times$ and $m_{3-j} \in p^{-1}\Z_p^\times$.
This completes the proof.
\end{proof}

\subsection{Matching Local Sections II}
\label{subsec:local-match-II}
In order to give the factorization result, we also need a matching result involving the following local sections with the $s$ parameter.
When $p = vv'$ splits in $F$, we define a slightly modified section
\begin{equation}
  \label{eq:Fvs}
  F_{\varphi_p, \vc_p, v, s}(g_p) := \int_{H_1(\Qb_p)}\Phi_0(\Fcr(\varphi_p))(g_p, h_1) \vc(h_1) |h_1|_v^s dh_1.
\end{equation}
This function  on $G_0(\Qb_p)$ depends on the choice of $v \mid p$, and has the following property.
\begin{lemma}
  \label{lemma:Is}
  When $|s| < 1$, the integral in \eqref{eq:Fvs} converges absolutely and defines a rational function in $p^{s}$ defined over $\Qip(\varphi_p)$.
  Furthermore, when restricted to the first (resp.\ second) components of $G_0(\Qb_p) \cong G(F_p) \cong G(F_v) \times G(F_{v'})$,
  it defines  a section in $I(s, \chi_v)$ (resp.\ $I(-s, \chi_{v'})$).
\end{lemma}
\begin{proof}
  For the first claim, one can suppose $\omega(g_p)\Fcr_1(\varphi_p)$ is the characteristic function of
  $$
  C_1 \times (p^{a_1}\Zb_p + r_1) \times (p^{a_2}\Zb_p + r_2) \times (p^{b_1}\Zb_p + t_1) \times (p^{b_2}\Zb_p + t_2),
  $$
  with $C_1 \subset \Qb_p^2$ a compact subset and $a_j, b_j \in \Zb, r_j, t_j \in \Qb_p$.
  As  in Lemma \ref{lemma:3}, the integral defining $F_{\varphi_p, \vc_p, v, s}(g_p)$ is given by
  \begin{align*}
    F_{\varphi_p, \vc_p, v, s}(g_p)
    &=     \int_{\Qb_p^\times}\int_{\Qb_p^4}(\omega(g_p)\Fcr_1(\varphi_p))(0, 0, \lambda_1, \lambda_2, \alpha \lambda_1, \alpha^{-1} \lambda_2) d\lambda_1 d\lambda_2
      \vc((\alpha, \alpha^{-1})) |\alpha|_p^s d^\times \alpha.
  \end{align*}
  Suppose $(0, 0) \in C_1$, otherwise the integral vanishes identically.
  When $|\alpha|_p \ge p^N$ for $N$ sufficiently large, we have $|\alpha^{-1} \lambda_2|_p$ very small for all $\lambda_2 \in p^{a_2}\Zb_p + r_2$. Therefore,
  when $t_2 \not\in p^{b_2}\Zb_p$ or $r_1 \not\in p^{a_1}\Zb_p$, the integral over those $\alpha$ with $|\alpha|_p \ge p^N$ is zero.
  When $t_2 \in p^{b_2}\Zb_p$ and $r_1 \in p^{a_1}\Zb_p$, we have
  \begin{align*}
    &\quad \int_{|\alpha|_p \ge p^N}\int_{\Qb_p^4}(\omega(g_p)\Fcr_1(\varphi_p))(0, 0, \lambda_1, \lambda_2, \alpha \lambda_1, \alpha^{-1} \lambda_2) d\lambda_1 d\lambda_2
      \vc((\alpha, \alpha^{-1})) |\alpha|_p^s d^\times \alpha\\
    &=\vol(C_1)
      \sum_{n \ge N} \vc((p, p^{-1}))^{-n} p^{ns}
      \int_{\Zb_p^\times}  \vc((a, a^{-1}))\\
&\quad \times      \vol(p^{a_1} \Zb_p \cap p^n(p^{b_1} + a^{-1} t_1))
      \vol((p^{a_2} \Zb_p + r_2) \cap p^{-n+b_2}\Zb_p)      d^\times a\\
    &=\vol(C_1) \vol(p^{a_2} \Zb_p + r_2)  \vol(p^{b_1} + a^{-1} t_1)
\int_{\Zb_p^\times}  \vc((a, a^{-1}))           d^\times a
      \sum_{n \ge N} (\vc((p, p^{-1}))^{} p^{-1+s})^n    ,
  \end{align*}
  which converges  when $|s| < 1$ and defines a rational function in $p^s$.
  The same argument takes care of the case when $|\alpha|_p$ is sufficiently small. This proves the first claim.



  %
  For the second claim,  it is clear from the definition that $F_{\varphi_p, \varrho_p, v, s}$ is locally constant as $\varphi_p$ is a Schwartz function.
For the transformation property, we have
  \begin{equation}
    \label{eq:ma}
    (    \omega(m(\alpha)g, h) \Fcr(\varphi_p))(0)
    = |\alpha \alpha'|_v (\omega(g, h (\alpha'/\alpha)) \Fcr(\varphi))(0)
  \end{equation}
  for $\alpha = (\alpha_1, \alpha_2) \in F_p^\times = (\Qb_p^\times)^2$.
  A change of variable plus $\vc(\alpha/\alpha') =  \chi_v(\alpha) \chi_{v'}(\alpha)$ and $|\alpha|_v = |\alpha_1|_p, |\alpha'|_v = |\alpha_2|_p$ then finishes the proof.
\end{proof}
Now, we will extend the matching result in Theorem \ref{thm:localmatch} to standard sections.
\begin{theorem}
  \label{thm:localmatchs}
  In the setting of Theorem \ref{thm:localmatch}, suppose $p = vv'$ splits and let $\lambda_{\alpha, v, s}(\phi_v) \in I_v(s, \chi_v)$ denote the standard section associated to $\lambda_{\alpha, v}(\phi_v) \in I_v(0, \chi_v)$ for $\phi_v \in \Sc(W_\alpha(F_v))$.
For any $r \in \Nb$, there exists  $\varphi^{}_p \in \Sc(V_p; \Qip)^{(G\cdot \TT)(\Zb_p)}$ such that
  \begin{equation}
    \label{eq:matchps}
    F_{\varphi^{}_p, \vc_p, v, s}(g) = \mathcal{L}(s) \lambda_{\alpha, v, s}(\phi_v)(g_v) \lambda_{\alpha, v, -s}(\phi_{v'})(g_{v'}) + O(s^r).
  \end{equation}
  for all $g = (g_v, g_{v'}) \in G_0(\Qb_p)$, where $\mathcal{L}(s) := (1 - p^{-2}) |D_E|_p^{-1}  L(1+s, \chi_v) L(1-s, \chi_{v'})$.
  \end{theorem}
\begin{remark}
  \label{rmk:standardsec}
    If $\vc_p$ is unramified and $\varphi_p$ is the characteristic function of the maximal lattice in $V_p$, then
  \begin{equation}
    \label{eq:localint}
    F_{\varphi_p, \vc_p, v, s}(1) = L(1 + s, \chi_v)L(1 - s, \chi_{v'})(1-p^{-2}).
  \end{equation}
by a similar calculation as in Lemma \ref{lemma:3}, and $F_{\varphi_p, \vc_p, v, s}/(L(1+s, \chi_v)L(1-s, \chi_{v'})(1-p^{-2}))$ is already the standard section $\lambda_{\alpha, v, s}(\phi_v)(g_v) \lambda_{\alpha, v, -s}(\phi_{v'})(g_{v'})$.
\end{remark}

\begin{proof}
  For   $ \varphi \in \Sc(V_p; \Cb)$ and $g \in G_0(\Qb_p)$, we write
  \begin{align*}
    F_{\varphi, \vc_p, v, s}(g)
    &= \sum_{n \ge 0} \frac{a_n(\varphi, g)}{n!} (-\log p \cdot s)^n,\\
    a_n(\varphi, g)
    &:= (- \log p)^{-n}\partial_s^{n} \lp
  F_{\varphi, \vc_p, v, s}(g)
      \rp\mid_{s = 0}
      =
      \int_{H_1(\Qb_p)}\Phi_0(\Fcr(\varphi_p))(g_p, h_1) \vc(h_1) \ord_v(h_1)^n  dh_1.            \end{align*}
    It is easy to check from definition that $F_{\varphi, \vc_p, v, s}: G_0(\Zb_p) \to \Cb$ satisfies \eqref{eq:cond3}, hence so does the function $a_n(\varphi, g)$ for all $n \ge 0$.
Now we define
$$
\varphi^{(n)} := \frac{1}{(-2)^{n}n!}(\omega((p, p^{-1})^2) - 1)^n \varphi  \in \Sc(V_p; \Cb)
$$
with $(p, p^{-1}) \in H_1(\Qb_p)$.
An easy induction shows that
$
F_{\varphi^{(n)}, \vc_p, v, s} = \tfrac{1}{n!}(\tfrac{1 - p^{2s}}{2})^n F_{\varphi^{}, \vc_p, v, s},
$
which implies
\begin{equation}
  \label{eq:varphin}
  a_{n'}(\varphi^{(n)}, g) =
  \begin{cases}
    0 & \text{ if } n' < n,\\
 F_{\varphi, \vc_p}& \text{ if } n' = n.
  \end{cases}
\end{equation}
When $\varphi \in \Sc(V_p; \Qip)$ is $(G\cdot \TT)(\Zb_p)$-invariant, so is the function $\varphi^{(n)} \in \Sc(V_p; \Qip)$.
Furthermore, the function $a_n(\varphi, \cdot): G_0(\Zb_p) \to \Qip$ satisfies conditions \eqref{eq:cond3} and \eqref{eq:Galois1}.
By Proposition \ref{prop:extend}, there exists $\varphi_{n} \in \Sc(V_p; \Qip)^{(G\cdot \TT)(\Zb_p)}$ such that
\begin{equation}
  \label{eq:varphi_n}
F_{\varphi_{n}, \vc_p}(k) = a_n(\varphi, k)
\end{equation}
for all $k \in G_0(\Zb_p)$.

Now we prove the theorem by induction on $r$.
The case $r = 1$ is just the content of Theorem \ref{thm:localmatch}.
Note that $|D|_p = 1$ when $p$ splits in $F$.
  Now suppose we have $\varphi$ satisfying \eqref{eq:matchps} for some $r \ge 1$.
  As
  $$
  \Phi_s :=
  \lambda_{\alpha, v, s}(\phi_v)(g_v) \lambda_{\alpha, v, -s}(\phi_{v'})(g_{v'}) \in I(s, \chi_v) I(-s, \chi_{v'})$$
  is a standard section, it satisfies
$
\Phi_s(k) = \Phi_0(k)
$ when $k \in G_0(\Zb_p)$.
So in that case, we have
\begin{align*}
  &    F_{\varphi^{}, \vc_p, v, s}(k)
    - \mathcal{L}(s)
    \Phi_s(k)
    =
    \lp a_r(\varphi, k) - c_r \Phi_0(k)\rp
    \frac{(-\log p \cdot s)^r}{r!} + O(s^{r+1}),
\end{align*}
with $c_r :=     (-\log p)^{-r} \partial^r_s \mathcal{L}(s)\mid_{s = 0}$ rational. %
If we set
$$
\tilde \varphi := \varphi - \varphi^{(r)}_r - c_r \varphi^{(r)} \in \Sc(V_p, \Qip)^{(G\cdot \TT)(\Zb_p)},
$$
then equations \eqref{eq:varphin} and \eqref{eq:varphi_n} give us
\begin{align*}
  &\quad    F_{\tilde\varphi^{}, \vc_p, v, s}(k)
-  \mathcal{L}(s) \Phi_s(k)\\
&    =
    \lp a_r(\varphi, k) - F_{\varphi_r, \vc_p}(k) + c_r F_{\varphi, \vc_p}(k) - c_r \Phi_0(k)\rp
    \frac{(-\log p \cdot s)^r}{r!} + O(s^{r+1}) = O(s^{r+1}).
\end{align*}
So $\tilde\varphi$ satisfies the claim for $r +1$. This completes the proof.
\end{proof}

Now, we can state a consequence of the matching result in Theorem \ref{thm:localmatchs}.
\begin{proposition}
  \label{prop:FCmatch}
  For matching sections $\varphi_p \in \Sc(V_p; \Qip)^{G(\Zb_p)}$ and $\{\phi_v \in \Sc(W_\alpha(F_v)): v \mid p\}$ as in Theorem \ref{thm:localmatchs} with $r = 2$, we have
  \begin{equation}
    \label{eq:FCmatch}
    \prod_{v \mid p} W^*_{t, v}( \phi_v) =
    \sum_{n \in \Zb}     \int_{H_1(\Qb_p)} \vc_p(h) \Fcr_1(\varphi_{p})((0, p^n), t'/p^{n}, - h^{-1} t/p^{n}) dh
  \end{equation}
  for all $t \in F_p^\times$.
  Furthermore when $p = vv'$ is a split prime and $W_{t, v}^*(\phi_v) = 0$, then we have
  \begin{equation}
    \label{eq:FCmatchs}
\frac{    W^{*, \prime}_{t, v}(\phi_v)      W^{*}_{t, v'}(\phi_{v'})}{\log p}
= \sum_{n \in \Zb}     \int_{H_1(\Qb_p)} \vc_p(h) \Fcr_1(\varphi_{p})((0, p^n), t'/p^{n}, -h^{-1} t/p^{n}) o_v(h) dh.
  \end{equation}
\end{proposition}


\begin{remark}
  \label{rmk:vv'}
  Since $o_{v'}(h) = - o_v(h)$ for all $h \in H_1(\Qb_p)$, the left hand side of \eqref{eq:FCmatchs} gets a minus sign if $o_v(h)$ is replaced by $o_{v'}(h)$ on the right hand side.
\end{remark}

\begin{proof}
  To prove \eqref{eq:FCmatch}.   we apply the definition of $W^*_{t, v}$ in \eqref{eq:W*0} and Theorem \ref{thm:localmatch} to obtain
  \begin{align*}
    \prod_{v \mid p} W^*_{t, v}( \phi_v)
    &=
|D/D_E|_p^{1/2}      \prod_{v \mid p} L(1, \chi_v)
      \int_{F_v}\lambda_{\alpha, v}(\phi_v) (wn(b_v)) \psi_v(-tb_v) db_v\\
    &= (1-p^{-2})^{-1}  \int_{F_p} F_{\varphi_p, \vc_p}(wn(b)) \psi_p(-\tr(tb)) db\\
 &     =(1-p^{-2})^{-1}     \int_{H_1(\Qb_p)}    \vc(h_1)
 \int_{F_p}
\Fcr(\omega_p((wn(b), h_1))\varphi_p)(0) \psi_p(-\tr(tb))
   db dh_1 .
  \end{align*}
  Now the right hand side of \eqref{eq:FCmatch} can be rewritten as
  \begin{align*}
    \text{ Right hand side of }
    &\eqref{eq:FCmatch}
    =
      \int_{H_1(\Qb_p)} \vc_p(h_1)
      \sum_{n \in \Zb}
      \Fcr_1( \omega_p(h_1)\varphi_{p})((0, p^n), t'/p^{n}, -t/p^{n}) dh_1\\
    &=
      \int_{H_1(\Qb_p)} \vc_p(h_1)
      \sum_{n \in \Zb}
      (1-p^{-1})^{-1} \int_{p^n \Zb_p^\times}      \Fcr_1( \omega_p(h_1)\varphi_{p})((0, u), t'/u, -t/u) d^\times u dh_1    \\
    &=
      (1-p^{-1})^{-1}      \int_{H_1(\Qb_p)} \vc_p(h_1)
      \int_{\Qb_p^\times}      \Fcr_1( \omega_p(h_1)\varphi_{p})((0, u), t'/u, -t/u) d^\times u dh_1.
  \end{align*}
  For the second line, we have used  $\omega(m(a))\varphi_p = \varphi_p$ for all $a \in \Zb_p^\times$ since $\varphi_p$ is $G(\Zb_p)$-invariant.
  Equation \eqref{eq:FCmatch} now follows from applying Proposition \ref{prop:Key} to $\varphi = \omega_p(h_1)\varphi_p$.

  We now prove \eqref{eq:FCmatchs}.
Let $ \Phi_s \in I(s, \chi_v)I(-s, \chi_{v'})$ be the standard sections extending $ \lambda_{\alpha, v}(\phi_v)\lambda_{\alpha, v'}(\phi_{v'})$.
By Theorem \ref{thm:localmatchs} with $r = 2$, we have
  $$
  F_{\varphi_p, \vc_p, v, s} = (1-p^{-2})|D_E|_p^{-1}L(1, \chi_v)L(1, \chi_{v'}) \lambda_{\alpha, v, s}(\phi_v)\lambda_{\alpha, v', -s}(\phi_{v'}) + O(s^2).
  $$
  Therefore using $W^{*}_{t, v}(\phi_v) = 0$, we have
  \begin{align*}
    &W^{*, \prime}_{t, v}(\phi_v)
    W^{*}_{t, v'}(\phi_{v'})
    = \partial_s
\lp      W^{*}_{t, v}(1, s, \phi_v)      W^{*}_{t, v'}(1, -s, \phi_{v'}) \rp \mid_{s = 0}\\
    &= |D_E|_p^{-1}      L(1, \chi_v)      L(1, \chi_{v'})
\partial_s
      \lp \int_{F_p}
      \Phi_s((wn(b_v), wn(b_{v'})))
      \psi_{v}(-tb_v) \psi_{v'}(-tb_{v'}) db_v db_{v'} \rp \mid_{s = 0}\\
&=
(1-p^{-2})^{-1} \partial_s
\lp \int_{F_p}  F_{\varphi_p, \vc_p, v, s}(wn(b))\psi_p(-\tr(tb)) db \rp \mid_{s = 0}\\
&= (1-p^{-2})^{-1}
\partial_s
\lp \int_{H_1(\Qb_p)}    \vc(h_1) |h_1|_v^s
 \int_{F_p}
\Fcr(\omega_p((wn(b), h_1))\varphi_p)(0) \psi_p(-\tr(tb))
   db dh_1 \rp \mid_{s = 0}\\
&=
\log p (1-p^{-2})^{-1}
 \int_{H_1(\Qb_p)}    \vc(h_1) \ord_v(h_1)
 \int_{F_p}
\Fcr(\omega_p((wn(b), h_1))\varphi_p)(0) \psi_p(-\tr(tb))
   db dh_1.
  \end{align*}
  Applying Proposition \ref{prop:Key} and
  continuing as in the second half of the proof of \eqref{eq:FCmatch} then proves \eqref{eq:FCmatchs}.
\end{proof}

We end this section with the following technical result used in the proof of the previous proposition.
\begin{proposition}
  \label{prop:Key}
  For any $\varphi \in \Sc(V_p; \Cb)^{G(\Zb_p)}$ and $t \in F_p^\times$, we have
\begin{equation}
  \label{eq:Key}
(1+1/p)^{-1} \int_{F_p}
  (\omega_p((w n(\beta))\Fcr(\varphi))(0) \psi_p(-\tr(t\beta)) d\beta
    =
  \int_{\Qb_p^\times} \Fcr_1(\varphi)((0, u), t'/u, -t/u) d^\times u.
\end{equation}
\end{proposition}
\begin{proof}
Since the left hand side is essentially the Fourier transform of $(\omega(wn(\beta))\Fcr(\varphi_p))(0)$ as a function of $\beta \in F_p$, it suffices to calculate the inverse Fourier transform of the right hand side, though we need to be careful about the singularity of right hand side when $t \not\in F_p^\times$.
To take care of this, we define
\begin{equation}
  \label{eq:Gt}
  G_\e(t, \varphi) :=
  \begin{cases}
  \int_{\Qb_p^\times} \Fcr_1(\varphi)((0, u), t'/u, -t/u) d^\times u,& \text{ if } |t| > \e,\\
    0,& \text{ otherwise.}
  \end{cases}
\end{equation}
for  $\e > 0$ and $t \in F_p$.
Note that $|t| := \min\{|t_1|, |t_2|\}$ when $t = (t_1, t_2) \in F_p = \Qb_p^2$,
Given any fixed $t \in F_p^\times$, the limit $\lim_{\e \to 0} G_\e(t, \varphi)$ exists and is the right hand side of \eqref{eq:Key}.
Also for any fixed $\e > 0$, the function $G_\e(t, \varphi)$ is a Schwartz function on $F_p$.
Its inverse Fourier transform is given by
\begin{align*}
\hat G_\e(\beta, \varphi) &:=
\int_{F_p} G_\e(t, \varphi) \psi_p(\tr(t'\beta')) dt\\
&=    \int_{F_p\backslash D_\e}
   \int_{\Qb_p^\times} \Fcr_1(\varphi)((0, u),  t'/u, -t/u) d^\times u \psi_p(\tr(t'\beta')) dt\\
& =   \int_{\Qb_p^\times} \int_{F_p\backslash D_{\e/|u|}} \Fcr_1(\varphi)((0, u),  \tilde t', -\tilde t)
    \psi_p(\tr(\tilde t' u \beta')) |u|^2 d\tilde t {d^\times u},
   \end{align*}
where $D_\e \subset F_p$ is the $\e$-neighborhood of $0$.
Note that
\begin{align*}
   \Fcr_1(\varphi)((0, u),  \tilde t', -\tilde t)     \psi_p(\tr(\tilde t' u \beta'))
&= \Fcr_1(\varphi_\beta)((0, u),  - \tilde t, -\tilde t),
\end{align*}
where $\varphi_\beta := \omega(w_2 n(\beta))\varphi$ and $w_i \in H(\Qb)$ is defined in \eqref{eq:wi}.
Therefore, $\hat G_\e(\beta, \varphi)$ is given by
   \begin{align*}
\hat G_\e(\beta, \varphi)
& =  \int_{\Qb_p^\times} \int_{F_p\backslash D_{\e/|u|}}\Fcr_1(\varphi_\beta)((0, u),  - \tilde t, -\tilde t) |u|^2 d\tilde t {d^\times u}\\
&    =  \int_{\Qb_p^\times}
\lp
 \Fcr(\varphi_\beta)((0, u), 0)
-
\int_{D_{\e/|u|}}  \phi_\beta ((0, u),  s,  s) ds
 \rp
|u|^2 d^\times u
  \end{align*}
  with $s = -\tilde t$ and $ds = d \tilde t$.
  Using the $G(\Zb_p)$-invariance of $\varphi$, we have
  $$
  \Fcr(\varphi_\beta)((0, p^n u), 0) = \Fcr(\varphi_\beta)((0, p^n ), 0)
  = \Fcr(\varphi_\beta)(a, 0)
  $$
for all $u \in \Zb_p^\times, n \in \Zb$ and $a \in (p^n\Zb_p)^2  - (p^{n+1}\Zb_p)^2 $.
  Applying this, we can evaluate the first part as
  \begin{align*}
&  \int_{\Qb_p^\times}
\Fcr(\varphi_\beta)((0, u), 0) |u|^2 d^\times u
= \sum_{n \in \Zb}   \int_{\Zb_p^\times} \Fcr(\varphi_\beta)((0, p^n u), 0) |p^n u|^2 d^\times u\\
    &    = \sum_{n \in \Zb}  \Fcr(\varphi_\beta)((0, p^n), 0)
p^{-2n}(1-1/p)
    = (1+1/p)^{-1} \sum_{n \in \Zb}  \Fcr(\varphi_\beta)((0, p^n), 0)
      \int_{(p^n\Zb_p)^2 - (p^{n+1}\Zb_p)^2} da\\
&=
(1+1/p)^{-1} \sum_{n \in \Zb} \int_{(p^n\Zb_p)^2 - (p^{n+1}\Zb_p)^2}  \Fcr(\varphi_\beta)(a, 0) da
= (1+1/p)^{-1}  \int_{\Qb_p^2}  \Fcr(\varphi_\beta)(a, 0) da\\
&= (1+1/p)^{-1} \Fcr(\omega(w_1)\varphi_\beta)(0)
= (1+1/p)^{-1} (\omega(wn(\beta))\Fcr(\varphi))(0).
 \end{align*}
Then for any fixed $t \in F_p^\times$, we have
\begin{align*}
  G_\e(t, \varphi)
  &= \int_{F_p}   G_\e(\beta, \varphi)\psi_p(- \tr(\beta t)) d\beta\\
  &= (1+1/p)^{-1} \int_{F_p} ((\omega(wn(\beta))\Fcr)(\varphi))(0) \psi_p(- \tr(\beta t)) d\beta
    - E_\e(\varphi),\\
  E_\e(\varphi)
  &:=
\int_{F_p}  \psi_p(- \tr(\beta t))
\int_{\Qb_p^\times}    |u|^2
    \int_{D_{\e/|u|}}
   \Fcr_1(\varphi)((0, u),  -s', s)     \psi_p(-\tr( s u \beta)) ds
    d^\times u d\beta
\end{align*}
Since $\Fcr_1(\varphi)$ is a Schwartz function, we can replace the domain $\Qb_p^\times \times D_{\e/|u|}$ with a compact open subset independent of $\beta$, and interchange the order of integration to compute the integral over $\beta$ first, which yields
$$
E_\e(\varphi) = \int_{\Qb_p^\times}    |u|^2
\int_{D_{\e/|u|}}
\int_{F_p}  \psi_p(- \tr(\beta (t + su))) d\beta
\Fcr_1( \varphi) ((0, u),  -s',  s) ds
    d^\times u .
    $$
    When $\e$ is sufficiently small, we have $t \not\in D_\e$ and $t + su \neq 0$, in which case $E_\e = 0$.
    This finishes the proof.
\end{proof}

\section{
  Doi-Naganuma Lift of the Deformed Theta Integral}
\label{sec:DNMod}
In this section, we will define and study the properties of the function $\tI$ discussed in the introduction. In particular, we will calculate its Fourier coefficients and images under lowering differential operators.
The actions of differential operators follow from those on the theta kernel, which are given in section \ref{subsec:L}.
The Fourier expansion computations are carried out in section \ref{subsec:FEtI}, with the main result being Proposition \ref{prop:tI+}.
Section \ref{subsec:Millson} contains rationality results about theta lifts that will be needed to handle the error term mentioned in section \ref{subsec:idea}.

Choose $\varphi^{(1, 1)} :=  \varphi_f\varphi^{(1, 1)}_{\infty}  \in \Sc(V(\Ab))^{K_\vc}$ with $\varphi^{(1, 1)}_\infty := \varphi^{(1, 1)}_{0, \infty} \otimes \varphi^+_\infty \in \Sc(V_0(\Rb)) \otimes \Sc(V_1(\Rb))$  and
\begin{equation}
  \label{eq:tv}
  \varphi^{(1, 1)}_{0, \infty} (a, b, \nu, \nu') :=
  -i(a-b + i(\nu + \nu'))
  e^{-\pi (a^2 + b^2 + \nu^2 + (\nu')^2)}
  \in \Sc(V_0(\Rb)).
\end{equation}
For $\tvc_C$ as in \eqref{eq:tvc}, we can define
\begin{equation}
  \label{eq:tI}
  \tI(g_0) :=
  \Ic(g_0, \varphi^{(1, 1)}, \tvc_C)
= \int_{[H_1]}  \int_{[G]} \theta(g, (g_0, h_1), \varphi^{(1, 1)}) dg \tvc_C(h_1)dh_1.
\end{equation}
We will now analyze various properties of this integral.

\subsection{Lowering Operator Action}
\label{subsec:L}
To calculate the action of differential operators on $\tI$, it suffices to understand the effect on $\varphi_0$ via the Weil representation, which can be done in the Fock model. For this, we follow the appendices in \cite{FM06} and \cite{Li18} (see also \cite{KM90}).

We identify $(V_0(\Rb), Q) = (M_2(\Rb), \det) \cong \Rb^{2, 2}$ with the basis
\begin{equation}
  \label{eq:vs}
  \begin{split}
      v_1 &:=  \frac{1}{\sqrt{2}} \pmat{1}{0}{0}{1},
  v_2 := \frac{1}{\sqrt{2}} \pmat{0}{-1}{1}{0},
  v_3 :=  \frac{1}{\sqrt{2}} \pmat{-1}{0}{0}{1},
  v_4 :=  \frac{1}{\sqrt{2}} \pmat{0}{1}{1}{0},
  \end{split}
\end{equation}
which identifies $\Sc(\Rb^{2, 2}) \cong \Sc(V_0(\Rb))$ and gives us
\begin{equation}
  \label{eq:bases-change}
  \pmat{a}{\nu_1}{\nu_2}{b} = \frac{a + b}{\sqrt{2}} v_1 +
  \frac{-a + b}{\sqrt{2}} v_3 +
  \frac{-\nu_1 + \nu_2}{\sqrt{2}} v_2 +
  \frac{\nu_1 + \nu_2}{\sqrt{2}} v_4.
\end{equation}
The \textit{polynomial Fock space} is the subspace $\mathbb{S}(\Rb^{2, 2}) \subset \Sc(\Rb^{2, 2})$ spanned by  functions of the form $\prod_{1 \le j \le 4} D_j^{r_j} \varphi^\circ$ for $r_j \in \Nb_0$, where $\varphi^\circ \in \Sc(\Rb^{2, 2})$ is the Gaussian
$$
\varphi^\circ(x_1, x_2, x_3, x_4) := e^{-\pi(x_1^2 + x_2^2 + x_3^2 + x_4^2)}
$$
and $D_r$ are operators on $\Sc(\Rb^{2, 2})$ defined by
\begin{equation}
  \label{eq:Dj}
D_r := \partial_{x_r} - 2\pi x_r,~ 1 \le r \le 4.
\end{equation}
There is an isomorphism $\iota: \mathbb{S}(\Rb^{2, 2}) \to \Ps(\Cb^4) = \Cb[\zf_1, \zf_2, \zf_3, \zf_4]$ such that $    \iota( \varphi^\circ) = 1$, $D_r$ acts as $(-1)^{\lfloor (r -1)/2\rfloor} i \zf_r $.
We now set
 \begin{equation}
   \label{eq:vwf}
\vf:= \zf_1 + i \zf_2,~ \wf := \zf_3 - i\zf_4 .
\end{equation}
Then using \eqref{eq:bases-change}, the Schwartz functions $\varphi^{(\e, \e')}_{0, \infty} \in \Sc(V_0(\Rb))$ in \eqref{eq:varphiinf} and \eqref{eq:tv} become
\begin{equation}
  \label{eq:iotavarphi0}
  \begin{split}
    \iota(\varphi^{( 1, - 1)}_{0, \infty})
    &=     i\sqrt{2} \iota( ( x_1 + i x_2) \varphi^\circ)
= -\frac{i \sqrt{2}}{4\pi} \iota ( (D_1 + iD_2) \varphi^{\circ})
=    \frac{\sqrt{2}}{4\pi} \vf,\\
    \iota(\varphi^{( -1,  1)}_{0, \infty})
    &= - \frac{\sqrt{2}}{4\pi} \bvf,~
\iota(\varphi^{( 1,  1)}_{0, \infty}) = - \frac{\sqrt{2}}{4\pi} \wf.
  \end{split}
\end{equation}

Let $(W, \langle, \rangle)$ be the $\Rb$-symplectic space of dimension 2, and
$\Wb := V_0(\Rb) \otimes W$ the symplectic space with the skew-symmetric form $(,)\otimes \langle, \rangle$.
The Lie algebra $\spf(\Wb \otimes \Cb)$ acts on $\Sc(V_0)$ through the infinitesimal action induced by $\omega$, which we also denote by $\omega$.
In  $\spf(\Wb \otimes \Cb)$, we have the subalgebra $\spf(W \otimes \Cb) \times \mathfrak{o}(V_0 \otimes \Cb)$.
Through $\iota$, the elements $L, R \in \slf_2(\Cb) \cong \spf(W \otimes \Cb)$ defined in \eqref{eq:RL} act on $\Cb[\zf_1, \zf_2, \zf_3, \zf_4]$ as (see \cite[Lemma A.2]{FM06})
\begin{equation}
  \label{eq:fock1}
      \omega(L) = - 8\pi  \partial_{\vf} \partial_{\overline{\vf}}+    \frac{1}{8\pi} \wf \overline{\wf},~
  \omega(R) = - 8\pi \partial_{\wf} \partial_{\overline{\wf}} + \frac{1}{8\pi} \vf \overline{\vf}.
\end{equation}
Using the isomorphism
\begin{align*}
  \slf_2(\Cb)^2 &\to \mathfrak{o}(V_0 \otimes \Cb)  \\
(A, B) &\mapsto (v \mapsto A v + v B^t),
\end{align*}
we see that the elements $L_1 = (L, 0), L_2 = (0, L_2), R_1 = (R, 0), R_2 = (0, R)$ in $  \slf_2(\Cb)^2$ act on $\Cb[\zf_1, \zf_2, \zf_3, \zf_4]$ through $\iota$ as (see \cite[Lemma A.1]{FM06})
\begin{equation}
  \label{eq:fock}
  \begin{split}
    \omega(L_1) &= 8\pi  \partial_{\vf} \partial_{{\wf}} -    \frac{1}{8\pi} \overline{\vf} \overline{\wf},~
  \omega(R_1) =  8\pi \partial_{\overline{\vf}} \partial_{\overline{\wf}} - \frac{1}{8\pi} \vf \wf,\\
    \omega(L_2) &=  8\pi  \partial_{\overline{\vf}} \partial_{{\wf}} -    \frac{1}{8\pi} {\vf} \overline{\wf},~
  \omega(R_2) = 8\pi \partial_{{\vf}} \partial_{\overline{\wf}} - \frac{1}{8\pi} \overline{\vf} \wf,
  \end{split}
\end{equation}
For convenience, we slightly abuse notation and write $L, R, L_j, R_j$ for their corresponding actions on $\Ps(\Cb^4)$.
When we consider the decomposition $V_0 = V_{00} \oplus U_{D}$ in \eqref{eq:V0decomp}, the map $\iota$ induces $\mathbb{S}(V_{00}(\Rb)) \cong \Ps(\Cb^3) = \Cb[\zf_1, \zf_3, \zf_4]$ and  $\mathbb{S}(U_{D}(\Rb)) \cong \Ps(\Cb) = \Cb[\zf_2]$.
For $a, b, c \in \Nb_0$, we also define $\varphi^{(a, b)}_{00, \infty} \in \mathbb{S}(V_{00}(\Rb))$ and $\varphi^c_{D, \infty} \in \mathbb{S}(\Rb)$ by
\begin{equation}
  \label{eq:varphi00}
  \varphi^{(a, b)}_{00, \infty} := \lp - \frac{\sqrt{2}}{4\pi}\rp^{a+b}  \iota^{-1}(\zf_1^a \wf^{b}),~
  \varphi^{(c)}_{D, \infty}:=  \lp - \frac{\sqrt{2}i}{4\pi}  \rp^c \iota^{-1}(\zf_2^c).
\end{equation}

For $r \in \Nb_0$, we have the operators $\RC_{r}, \tRC_{r}$ defined in \eqref{eq:RC} that also act on $ \Ps(\Cb^4)$.
They are related by the following lemma.
\begin{lemma}
\label{lemma:RC}
In the notations above, we have
\begin{equation}
  \label{eq:RL}
  \begin{split}
    (    L_1 + L_2) \RC_r(\wf) &= -L (\tRC_r ( \vf +   (-1)^r  \bvf)),\\
    (\tRC_r - (R_1 + R_2)^r) ( \vf) &= (-8\pi)^{1-2r} 2^r R(\wf^r  p_r(\vf, \bvf)),\\
(\tRC_r - (R_1 + R_2)^r) ( \bvf) &= (-8\pi)^{1-2r} 2^r R(\wf^r p_r(\bvf, \vf)),~
\end{split}
\end{equation}
where $p_r(X, Y) := - (\tQ_r(X, Y) - (X + Y)^r)/Y \in \Qb[X, Y]$ for all $r \in \Nb_0$.
\end{lemma}

\begin{proof}
  It is easy to check that
  \begin{equation}
    \label{eq:RCact}
  \begin{split}
  \RC_r(\wf) &= (-8\pi)^{-2r}2^r \wf^{r+1} Q_r(\vf, \bvf),\\
    \tRC_r(\vf) &= (-8\pi)^{-2r}2^r \wf^{r} \tQ_r(\vf, \bvf) \vf,\\
    \tRC_r(\bvf) &= (-8\pi)^{-2r}2^r \wf^{r} \tQ_r(\vf, \bvf) \bvf.
  \end{split}
  \end{equation}
  This leads directly to the second equation in \eqref{eq:RL} from the definition.
  To prove the first equation, it is enough to verify
  $$
  (L_1 + L_2)\wf^{r+1} Q_r(\vf, \bvf)
  =
  -  L(  Q_r(\vf, \bvf)(\vf + \bvf)\wf^{r}),$$
  which follows from \eqref{eq:Qdiff}.
\end{proof}

\begin{proposition}
  \label{prop:diffop}
  Let $\phi_f \in \Sc(W_\alpha(\hat F))$ and $\varphi_f \in \Sc(\hat V; \Qab)$ be matching sections as in Theorem \ref{thm:match}, and denote $\e := \sgn(\alpha_1) = - \sgn(\alpha_2)$.
  Then  for $\tI$ defined in \eqref{eq:tI}, we have
  \begin{equation}
    \label{eq:Ldiff}
    \begin{split}
\fac          (L_1 + L_2) \RC_r &\tI(    g_{}) %
=
- (-4\pi)^{-r} (R_1 + R_2)^r (E^*(g_{}, \phi^{(1, -1)}) - (-1)^r E^{*}(g_{}, \phi^{(-1, 1)}))\\
 & -2\log \ve_\vc
\fac \tRC_r
  \Ic_f(g_{}, \varphi^{(1, -1)}_{} - (-1)^r \varphi^{(-1, 1)}_{}, \vc),
    \end{split}
  \end{equation}
  where $\Ic_f$ is defined in \eqref{eq:If} and $\varphi^{(\pm 1, \mp 1)} = \varphi_f \otimes \varphi^{(\pm 1, \mp 1)}_{\infty} \in \Sc(V(\Ab))$ is defined in \eqref{eq:varphiinf}.
\end{proposition}
\begin{proof}
  Suppose $\varphi_f = \varphi_{0, f} \otimes \varphi_{1, f}$ and denote $\varphi^\pm_1 = \varphi_{1, f} \varphi^\pm_\infty, \varphi_{0}^{(k, k')} = \varphi_{0, f}\varphi^{(k, k')}_{0, \infty}$. Then
  \begin{align*}
&    (L_1 + L_2) \RC_r (\tI(    g_{}))
    =
      \int_{[G]}    (L_1 + L_2) \RC_r (\theta_0(g', g, \varphi^{(1, 1)}_{0}))
      \vt_1(g', \varphi_1^+, \tvc_C) dg'\\
    &=
     \int_{[G]}    L (\tRC_r (\theta_0(g' , g_{}, \varphi^{(1, -1)}_{0}) - (-1)^r \theta_0(g' , g_{}, \varphi^{(-1, 1)}_{0})))
      \vt_1(g', \varphi_1^+, \tvc_C) dg'
  \end{align*}
  by Lemma \ref{lemma:RC} and (\ref{eq:iotavarphi0}).
Now applying Stokes' theorem and Theorem \ref{thm:modified} gives us
\begin{align*}
    \int_{[G]}   & L( \tRC_r (\theta_0(g'  , g_{}, \varphi_{0}^{\e, -\e})))
        \vt_1(g',  \varphi_1^+, \tvc_C) dg'
  =
-    \int_{[G]}     \tRC_r (\theta_0(g'  , g_{}, \varphi_{0}^{\e, -\e}))
    L        \vt_1(g',  \varphi_1^+, \tvc_C) dg'    \\
  &=
-   \int_{[G]}     \tRC_r (\theta_0(g'  , g_{}, \varphi_{0}^{\e, -\e}))
    \lp        \vt_1(g',  \varphi_1^-, \vc)
+ 2\log \ve_\vc \Theta_1(g', \varphi_1^-, \vc)
    \rp
    dg'.
\end{align*}
with $\e = \pm 1$.
Since $R \vt_1(g', \varphi_1^-, \vc) = 0$, we can apply Stokes' theorem, Lemma \ref{lemma:RC} and Theorem \ref{thm:match} to obtain
\begin{align*}
      \int_{[G]}&     \tRC_r (\theta_0(g'  , g_{}, \varphi_{0}^{\e, -\e}))
        \vt_1(g',  \varphi_1^-, \vc) dg'\\
        &=
(-4\pi)^{-r} (R_1 + R_2)^{r}    \int_{[G]}     (\theta_0(g'  , g_{}, \varphi_{0}^{\e, -\e}))
                                                         \vt_1(g',  \varphi_1^-, \vc) dg'\\
&  = (-4\pi)^{-r} (R_1 + R_2)^{r} \Ic(g_{}, \varphi^{(\e, -\e)}, \vc)
  =\facinv (-4\pi)^{-r}  (R_1 + R_2)^{r} E^*(g_{}, \phi^{(\e, -\e)}).
\end{align*}
Putting these together finishes the proof.
\end{proof}

To understand the first term on the right hand side of \eqref{eq:Ldiff}, recall
the decomposition for $\theta_0(g', g^\Delta, \varphi_0)$ in \eqref{eq:theta0decomp} when $\varphi_{0, \infty} \in \mathbb{S}(V_0)$.
This allows us to define
\begin{equation}
  \label{eq:RC'}
  (  \RC^{', \Delta}_{r', (k_1,k_2)}\theta_0)(g', g_{00}^\Delta, \varphi_0)  := (-4\pi)^{-r'} Q_{r', (k_1, k_2)}(R_1', R_2')
\lp  \theta_{00}(g_1', g_{00}, \varphi_{00})   \theta_{D}(g_2', \varphi_{D}) \rp \mid_{g'_1 = g'_2 = g'}
\end{equation}
for $\varphi_0 = \varphi_{00} \otimes \varphi_{D}$ with $\varphi_{00} \in \Sc(V_{00}(\Ab)), \varphi_D \in \Sc(U_D(\Ab))$, and $R'_j$ the raising operator on $g'_j$.
In the Fock model, $R'_1$, resp.\ $R'_2$, acts on $\Cb[\zf_1, \zf_3, \zf_4]$, resp.\ $\Cb[\zf_2]$, as
\begin{equation}
  \label{eq:R'j}
  \omega(R'_1 ) =
   - 8\pi \partial_{\wf} \partial_{\overline{\wf}} + \frac{1}{8\pi} \zf_1^2,~
  \omega(R'_2 ) = \frac{1}{8\pi} \zf_2^2.
\end{equation}
This definition also extends by linearity to all $\varphi_0 \in \Sc(V_0(\Ab))$ satisfying $\varphi_{0, \infty} \in \mathbb{S}(V_0(\Rb))$.
We now record the following lemma.

\begin{lemma}
  \label{lemma:RC'}
  For $r \in \Nb_0$, denote $r_0 := \lfloor r/2\rfloor$.
  Then
  \begin{equation}
    \label{eq:RC't}
    \begin{split}
          (    \tRC_r \theta_0)&(g', g_{00}^\Delta, \varphi^{(1, -1)}_0 + (-1)^r \varphi^{(-1, 1)}_0 )\\
&    =
-2^{2r_0 - r + 1}(    \RC^{', \Delta}_{r_0, (-r + 1/2, r - 2r_0 + 1/2)} \theta_0)(g', g_{00}^\Delta, \varphi_{0, f}(\varphi^{(1, r)}_{00, \infty} \otimes  \varphi^{r - 2r_0}_{D, \infty}) )
    \end{split}
  \end{equation}
\end{lemma}

\begin{proof}
Suppose $\varphi_{0, f} = \varphi_{00, f} \otimes \varphi_{D, f}$.
Then equations in \eqref{eq:iotavarphi0} implies
$$
(    \tRC_r \theta_0)(g', g_{00}^\Delta, \varphi^{(1, -1)}_0 - (-1)^r \varphi^{(-1, 1)}_0 )
=
 \frac{\sqrt{2}}{4\pi}
\theta_0(g', g_{00}^\Delta, \varphi_{0, f}\iota^{-1}\tRC_r(\vf  + (-1)^r \bvf)).
$$
From \eqref{eq:RCact} and the definition of $P_r$ in \eqref{eq:Pr}, we have
\begin{align*}
  \tRC_r(\vf + (-1)^r \bvf))
  &=
(-4\pi)^{-2r}2^{-r} \wf^{r} \tQ_r(\vf, \bvf)( \vf  + (-1)^r \bvf)
    = (-4\pi)^{-2r}2^{-r} \wf^{r} Q_r(\vf, \bvf)( \vf  + \bvf)\\
  &= (-4\pi)^{-2r}2 \wf^{r} \zf_1^{r+1} P_r(i\zf_2/\zf_1)\\
  &= (-4\pi)^{-2r}    2 \wf^{r} \zf_1
    (-1)^{r_0} \sum_{s = 0}^{r_0} \binom{r_0 - r - 1/2}{r_0 - s} \binom{r - r_0 - 1/2}{s} \zf_1^{2s} (i\zf_2)^{r-2s}\\
  &= (-4\pi)^{r_0 -2r}2 i^r (-2)^{r_0} Q_{r_0, (-r+1/2, r-2r_0 + 1/2)}(R_1', R_2') \zf_1 \wf^r \zf_2^{r - 2r_0}.
\end{align*}
Substituting the definition \eqref{eq:varphi00} finishes the proof.
\end{proof}

The following technical lemma concerns a change of regularized integrals and follows from the proof of Lemma 5.4.3 in \cite{Li18}.
\begin{lemma}
  \label{lemma:switch}
  Given $\varphi_{i, f} \in \Sc(\hat V_i)$ with $i = 0, 1$,  let $\Gamma \subset \PSL_2(\Zb) \subset G_{00}(\Qb)$ be a congruence subgroup that acts trivially on $\varphi_{0, f}$.
  For any $a \ge 1, b \ge 0$ and $f \in M^!_{-2b}(\Gamma)$, we have
  \begin{align*}
    \int^{\mathrm{reg}}_{\Gamma \backslash \Hb} y^b f(z)
    \int^{\mathrm{}}_{[\SL_2]}
    &
\theta_1(g, \xi, \varphi_1^-)    \theta_0(g, g_{z}^\Delta, \varphi_{0, f}
  (\varphi^{(a, b)}_{00, \infty} \otimes  \varphi^{b-a+1}_{D, \infty}) )
dg  d\mu(z)\\
&=    \int^{\mathrm{reg}}_{[\SL_2]} \theta_1(g, \xi, \varphi_1^-)\int^{\mathrm{reg}}_{\Gamma \backslash \Hb} y^b f(z)
    \theta_0(g, g_{z}^\Delta, \varphi_{0, f}
  (\varphi^{(a, b)}_{00, \infty} \otimes  \varphi^{b-a+1}_{D, \infty}) )
d\mu(z)  dg.
  \end{align*}
\end{lemma}

\subsection{Fourier Expansion of $\tI$}
\label{subsec:FEtI}
To evaluate the Fourier expansion of $\tI$ in \eqref{eq:tI}, we change to a mixed model of the Weil representation using the partial Fourier transform $\Fcr_1$ defined in \eqref{eq:Fc}.

Throughout the section, we write
\begin{equation}
  \label{eq:g0}
g_0 = (g_{z_1}, g_{z_2}) \in G_0(\Rb)
\end{equation}
 for $(z_1, z_2) \in \Hb^2$ with $z_i = x_i + iy_i$
and $g_\taub \in G(\Rb)$  with $\taub = \ub + i\vb \in \Hb$, then \eqref{eq:WeilFourier} implies
        \begin{equation}
          \label{eq:omega0}
\begin{split}
  \omega_{0, \infty}(g_0)&       \Fcr_{1}(\varphi_{0, \infty})((0, r), \nu)
      =  \sqrt{y_1y_2} \ebf(r (x_2 \nu + x_1\nu')) \Fcr_{1}(\varphi_{0, \infty})((0, r \sqrt{y_1y_2}), \nu \sqrt{y_2/y_1}),\\
      \omega_{0, \infty}(g_\taub)
      & \Fcr_{1}(\varphi_{0, \infty})( (0, r), \nu)
      =
        \sqrt{\vb} \ebf(-\ub\nu\nu') \Fcr_{1}(\varphi_{0, \infty})((0, r/\sqrt{\vb}), \sqrt{\vb} \nu).
        \end{split}
        \end{equation}
        Also when $\varphi_{0, \infty} = \varphi^{(k, k')}_{0, \infty}$ with $k, k' = \pm 1$ as given in \eqref{eq:varphiinf} and \eqref{eq:tv}, we have
\begin{equation}
  \label{eq:FTv}
  \Fcr_{1}(    \varphi^{(k, k')}_{0, \infty})((0, r), \nu) =
  \varphi^{(k, k')}_{0, \infty}(i r, 0, \nu) e^{-2\pi r^2}
\end{equation}
with $r \in \Rb, \nu = (\nu_1, \nu_2) \in \Rb^2$.
After applying Poisson summation and unfolding, we can rewrite the theta kernel $\theta_0(g, g_0, \varphi_0)$ as
\begin{align*}
  \theta_0(g, g_0, \varphi_0)
  &=    \sum_{\lambda \in V_0(\Qb)} \omega_0(g, g_0) \varphi_0(\lambda)
    =    \sum_{\substack{\nu \in V_{-1}(\Qb)\\ \eta \in \Qb^2}}
  \omega_{-1}(g) \omega_0(g_0) \Fcr_1(\varphi_0)(\eta g, \nu ) \\
  &=     \sum_{\nu \in V_{-1}(\Qb)}
    \omega_{-1}(g) \omega_0(g_0)
    \lp      \Fcr_1(\varphi_0)((0, 0), \nu) +
    \sum_{\substack{r \in \Qb^\times\\ \gamma \in \Gamma_\infty\backslash \SL_2(\Zb)}}
  \Fcr_1(\varphi_0)((0, r)\gamma g, \nu)
    \rp.
\end{align*}
For a bounded, integrable function $f$ on $[G]$ such that $\theta_0(g, g_0, \varphi_0) f(g)$ is
right $\SL_2(\hat\Zb)\SO_2(\R)$-invariant,
we have
$$
I_0(g_0, \varphi_0, f)
= \int_{[G]} \theta_0(g, g_0, \varphi_0) f(g) dg
= \facinv \int_{\SL_2(\Zb) \backslash \Hb} \theta_0(g_\taub, g_0, \varphi_0) f(g_\taub) d\mu(\taub),
$$
which    can be written as
$   I_0(g_0, \varphi_0, f) =   I^0_0(g_0, \varphi_0, f) +   I^+_0(g_0, \varphi_0, f)$ with (see e.g.\ equation (4.2) in \cite{Kudla16})
\begin{equation}
  \label{eq:I02}
  \begin{split}
I^0_0(g_0, \varphi_0, f)  &:=
\facinv         \sum_{\nu \in V_{-1}(\Qb)}
\int_{\SL_2(\Zb)\backslash\Hb}         \omega_{-1}(g_\taub) \omega_0(g_0) \Fcr_1(\varphi_0)((0, 0), \nu) f(g_\taub) d\mu(\taub), \\
I^+_0(g_0, \varphi_0, f)
& =
\facinv
\sum_{\substack{\nu \in V_{-1}(\Qb),~ r \in \Qb^\times\\ \gamma \in \Gamma_\infty\backslash \SL_2(\Zb)}}
\int_{\SL_2(\Zb) \backslash \Hb} f(g_{\gamma \taub})
\omega_{-1}(g_{\gamma \taub}) \omega_0(g_0)
  \Fcr_1(\varphi_0)((0, r)g_{\gamma\taub}, \nu)  d\mu(\taub)
\\
&=
    \sum_{\nu \in V_{-1}(\Qb),~ r \in \Qb^\times}
    \Fcr_{1}(\varphi_{0, f})((0, r), \nu)
    \Ff_{r, \nu}(\varphi_{0, \infty})(z_1, z_2, f )\\
\Ff_{r, \nu} ( \varphi)(z_1, z_2, f)
  &:=    \facinv \int_{\Gamma_\infty\backslash \Hb}
(    \omega_0(g_\taub, g_{0}) \Fcr_{1}(\varphi))((0, r), \nu )f(g_\taub) d\mu(\taub).
    \end{split}
  \end{equation}
  Using the $\SL_2(\hat\Zb)$-invariance of $\theta_0(g, g_0, \varphi_0)f(g)$, we can rewrite for any $N \in \Nb$
  \begin{equation}
    \label{eq:Fb2}
\Ff_{r, \nu} ( \varphi)(z_1, z_2, f)
  :=    N^{-1} \facinv \int_{\Gamma_\infty^N \backslash \Hb}
(    \omega_0(g_\taub, g_{0}) \Fcr_{1}(\varphi))((0, r), \nu )f(g_\taub) d\mu(\taub)
  \end{equation}
with $\Gamma_\infty^N := \{n(Nb): b \in \Zb\} \subset \Gamma_\infty$.

For our purpose, we are interested in the case when $f(g) = \vartheta_1(g, \varphi_1, \rho)$ with $\rho \in\{ \vc, \tvc_C\}, \varphi_1 = \varphi^\pm_1$ and $\varphi_0 = \varphi_0^{(k, k')}$ for $k, k' = \pm 1$.
In that case, we have $I_0(g_0, \varphi_0, f) = \Ic(g_0, \varphi, \rho)$ with $\varphi = \varphi_0 \otimes \varphi_1$, and denote
\begin{equation}
  \label{eq:Ic+}
  \Ic^+(g_0, \varphi, \rho) := I^+_0(g_0, \varphi_0, \vartheta(\cdot, \varphi_1, \rho)),~
  \Ic^0(g_0, \varphi, \rho) := I^0_0(g_0, \varphi_0, \vartheta(\cdot, \varphi_1, \rho)).
\end{equation}
This can be extended by $\Qb$-linearity to all $\varphi = \varphi_f \varphi_\infty \in \Sc(V(\Ab))$ with $\varphi_f \in \Sc(\hat V)$ and $\varphi_\infty \in \mathbb{S}(V(\Rb))$.

The  constant term $I^0_0$ in the Fourier expansion of $I_0(g_0, \varphi_0, f)$ is independent of $x_1, x_2$ and can be evaluated by the change of model in section \ref{subsec:FTSW}.
For our purpose, we will state a decay result needed to prove Theorem \ref{thm:main-O22}
\begin{lemma}
  \label{lemma:decay}
  Suppose there is $s \in \Rb$ such that $|f(g_\tau)| \ll v^s$ for all $\tau$ in the usual fundamental domain of $\SL_2(\Zb)\backslash \Hb$.
  Then
  $$
\lim_{y \to \infty} y^{-c} \lp\partial_{y_1}^a \partial_{y_2}^b \frac{I^0_0(g_0, \varphi_0, f)}{\sqrt{y_1y_2}} \rp\mid_{y_1 = y_2 = y} = 0
  $$
  for any $a, b, c \in \Nb_0$ satisfying $a + b + c \ge 1$.
  When $a = b= c = 0$, the limit exists.
\end{lemma}
\begin{remark}
  \label{rmk:input-decay}
  It is easy to check that $f(g) = \vartheta(g, \varphi_1, \tvc_C)$ fulfills the condition in the lemma.
\end{remark}
\begin{proof}
  Let $\Fc$ denote the fundamental domain. Then we can use \eqref{eq:omega0} to obtain
  \begin{align*}
    \frac{    I^0_0(g_0, \varphi_0, f)}{\sqrt{y_1y_2}}
    &=
\facinv         \sum_{\nu \in F} \Fcr_1(\varphi_{0, f})((0, 0), \nu)
         \lp \nu \sqrt{y_2/y_1} + \nu' \sqrt{y_1/y_2}\rp\\
&\times      \int_{\Fc} \ebf(-\ub \nu\nu') e^{-\pi \vb      (\nu^2 y_2^2 + (\nu')^2 y_1^2)/(y_1y_2)} f(g_\taub) \frac{d\ub d\vb}{\vb}.
  \end{align*}
  Since $\Fcr_1(\varphi_{0, f})$ is a Schwartz function, we can suppose the sum over $\lambda \in F$ is replaced by a sum over the translate of a lattice.
  For the integral on the second line, we can trivially estimate it by
  $$
  \int^\infty_{\sqrt{3}/2}  e^{-\pi \vb      (\nu^2 y_2^2 + (\nu')^2 y_1^2)/(y_1y_2)} \vb^s\frac{ d\vb}{\vb}.
  $$
  From this, we see that $|   y^{-1}I^0_0((g_z, g_z), \varphi_0, f)|$ is bounded independent of $y$, and the second claim holds. This also gives the first claim for $a = b = 0$ and $c \ge 1$.
  The other cases follows from first applying $\partial_{y_1}^a\partial_{y_2}^b$ to $   \frac{    I^0_0(g_0, \varphi_0, f)}{\sqrt{y_1y_2}}$, then conducting the same estimate.
\end{proof}

We will now evaluate the non-constant term $\Ic^+$.
Let $\vc$ and $\tvc_C$ be as in \eqref{eq:tvc}, and $K_\vc \subset H_1(\hat\Zb), \Gamma_\vc \subset H_1^+(\Qb)$ be as in \eqref{eq:Kvc} and \eqref{eq:eK1} respectively.
For $f(g) = \vartheta_1(g, \varphi^-_1, \vc)$ and $\varphi_{0, \infty} = \varphi^{(\pm 1, \mp 1)}_{0, \infty}$ defined in \eqref{eq:varphiinf}, we can apply \eqref{eq:omega0}, \eqref{eq:FTv} and \eqref{eq:Fb2} to obtain
\begin{align*}
  \Ff_{r, \nu} ( \varphi^{(\pm 1, \mp 1)}_{0, \infty})
  (z_1, z_2, f)
&=\sgn(\nu)    2\sqrt{y_1y_2} \ebf(r ((x_2 \mp i y_2) \nu + (x_1 \pm i y_1)\nu'))\\
&  \times
  \vol(K_\vc)
  \sum_{\xi \in C} \vc(\xi)
  \sum_{\substack{\beta \in \Gamma_\vc \backslash F^\times\\  \beta\beta' = \nu\nu' < 0}}
\sgn(\beta) \varphi_{1, f}(\xi^{-1} \beta)
\end{align*}
if $\mp r \nu > 0$, and zero otherwise.
After the change of variable $t = r \nu'$, we have
\begin{align*}
  \Ic^+(g_0, \varphi^{(1, -1)}, \vc)
  &=
\facinv    \sqrt{y_1y_2}
    \sum_{t \in F^\times,~ t_1 > 0 > t_2}
c_t(\varphi_{f}, \vc)    \ebf(t_1 z_1 + t_2 \overline{z_2}),\\
  \Ic^+(g_0, \varphi^{(-1, 1)}, \vc)
  &=
\facinv    \sqrt{y_1y_2}
    \sum_{t \in F^\times,~ t_2 > 0 > t_1}
c_t(\varphi_{f}, \vc)    \ebf(t_1 \overline{z_1} + t_2 {z_2}),\\
  c_t(\varphi_{f}, \vc)
  &:= 2 \vol(K_\vc)
  \sum_{\xi \in C} \vc(\xi)
    \sum_{r \in \Qb^\times}
  \sum_{\substack{\beta \in \Gamma_\vc \backslash F^\times\\  \Nm(\beta) = \Nm(t)/r^2 < 0}}
    \sgn(\beta_1 t_2/r)
\Fcr_{1}(\varphi_{f})((0, r), t/r,  \xi^{-1} \beta)     .
\end{align*}
In the case of $\Ic^+(g_0, \varphi, \tvc_C)$, we have the following result.

\begin{proposition}
  \label{prop:tI+}
  Given $\phi_f \in \Sc(\hat W_\alpha)$, let  $\varphi_f
  \in \Sc(\hat V)^{G(\hat\Zb)\TT(\hat\Zb)K_\vc}$ be   a matching section as in Theorem \ref{thm:match}.
  For $\tvc_C$ as in \eqref{eq:tvc}, we have
  \begin{equation}
    \label{eq:I+FE}
\frac{   \Ic^+(g_0, \varphi_f \varphi^{(1, 1)}_\infty, \tvc_C)}{ \sqrt{y_1y_2}}
    = \facinv \sum_{t \in F^\times,~ t \gg 0 }
  \tc_t(\varphi_{f}, \vc)    \ebf(t_1 z_1 + t_2 {z_2})
  +\sum_{t \in F^\times} e_t(\varphi_f; y_1, y_2) \ebf(t_1z_1 + t_2z_2)
  \end{equation}
  where $\tc_t(\varphi_f, \vc) \in \Cb$ and $e_t(\varphi_f; \cdot): \Rb_{> 0}^2 \to \Cb$ are given in \eqref{eq:tc} and \eqref{eq:et} below.
  There is a constant $M \in \Nb$ such that $\tc_t(\varphi_f, \vc) = 0$ when $ t \not\in M^{-1}\Oc$.

  Let
  \begin{equation}
    \label{eq:SC}
  S_C := \{v \text{ prime in }F: \ord_v(\xi) \neq 0 \text{ for some }\xi \in C\}
  \end{equation}
  be a finite set of primes, and
  Then there exists  $\kappa \in \Nb$ and $\beta(t, \phi_f) \in F^\times$ such that
  $\tc_t(\varphi_f, \vc)  = -\frac{2}{\kappa}\log |\beta(t, \phi_f)/\beta(t, \phi_f)'|$  and
  \begin{equation}
    \label{eq:betafac}
    \kappa^{-1}    \ord_\pf(\beta(t, \phi_f)/ \beta(t, \phi_f)') =
    \begin{cases}
      \tW_t(\phi_f)      ,&\hbox{if }     \Diff(W_\alpha, t) = \{\pf\},\\
-\tW_t(\phi_f)      ,&\hbox{if }     \Diff(W_\alpha, t) = \{\pf'\},\\
      0,&\hbox{otherwise, }
    \end{cases}
  \end{equation}
when $\Diff(W_\alpha, t) \subset S_C^c$  with %
  $\tW_t$ defined in %
  \eqref{eq:tW}. %
Furthermore, the function  $e_t$ satisfies
  \begin{equation}
    \label{eq:et0}
    \lim_{y \to \infty}
    y^{-c} \lp \partial^a_{y_1}\partial^b_{y_2}  e_t(y_1, y_2) \rp\mid_{y_1 = y_2 = y} = 0
  \end{equation}
  for all $a, b, c \in \Nb_0$.
\end{proposition}

\begin{remark}
  \label{rmk:tI}
  Note that we have
  $$
\tI(g_0) = \Ic(g_0, \varphi^{(1, 1)}, \tvc_C) = \Ic^0(g_0, \varphi^{(1, 1)}, \tvc_C) + \Ic^+(g_0, \varphi^{(1, 1)}, \tvc_C).
  $$
  from \eqref{eq:tI} and \eqref{eq:Ic+}.
\end{remark}
\begin{proof}
   Suppose $\varphi_f = \varphi_{0, f} \otimes \varphi_{1, f}$, as the general case follows by linearity.
   We first prove \eqref{eq:I+FE}.
   Using \eqref{eq:I02}, it is enough to evaluate $\Ff_{r, \nu}(\varphi^{(1, 1)}_{0, \infty})(z_1, z_2, \vartheta_1(\cdot, \varphi^+_1, \tvc_C))$.
   If we set $t := r\nu'$, then we have by (\ref{eq:omega0}), (\ref{eq:FTv}), and Theorem  \ref{thm:modified}
  \begin{align*}
    \Ff_{r, \nu'}(\varphi^{(1, 1)}_{0, \infty})(z_1, z_2, \vartheta_1(\cdot, \varphi^+_1, \tvc_C))
    &= \lp e_{ r, \nu'}(\varphi_{1, f};y_1, y_2)%
      +  \tc_{r, \nu'}(\varphi_{1, f}, \vc) \rp \sqrt{y_1y_2}\ebf( t_1 z_1 + t_2 z_2))
  \end{align*}
  with $\tc_{r, \nu'}(\varphi, \vc)$ and $e_{r, \nu}(\varphi; y, y')$ given by
  %
  \begin{align*}
    \tc_{r, \nu'}(\varphi, \vc)
    &:=
      \begin{cases}
4              \vol(K_\vc)
      \displaystyle\sum_{\substack{\beta \in \Gamma_\vc \backslash F^\times\\  \beta\beta' = \nu\nu' > 0\\ \xi \in C}}              \vc(\xi)
      \sgn(r\beta) \varphi_{}(\xi^{-1} \beta)
\log (\ve_{\vc})    \left\{\log_{\ve_{\vc}} \sqrt{|\beta/\beta'|} \right\},      & t \gg 0\\
      0,& \text{ otherwise.}
      \end{cases}\\
    e_{r, \nu}(\varphi; y_1, y_2)
    &:=
-\frac{\vol(K_\vc)}{\sqrt{y_1y_2}}
      \sum_{\substack{\beta \in \Gamma_\vc \backslash F^\times\\  \beta\beta' = \nu\nu'\\ \xi \in C}}
 \vc(\xi)
    \varphi_{}(\xi^{-1}\beta)
\lp \delta_{\nu\nu' < 0} \sgn(\beta)
    e^*_{r, \nu}(y_1, y_2)
    +
    \frac{\log \ve_{\vc}}{\sqrt{\pi}}
    e^\dagger_{r, \nu, \beta}(y_1, y_2)
    \rp.
  \end{align*}
Here we have set
  \begin{align*}
    e^*_{r, \nu}(y_1, y_2)
    &:=     \int^\infty_0 \Gamma(0, 4\pi |\nu\nu'| \vb)
K_{r, \nu}(\vb, y_1, y_2) \frac{d\vb}{\vb},\\
      e^\dagger_{r, \nu, \beta}(y_1, y_2)
    &:=
      \int^\infty_0
      \sum_{\e \in \Gamma_1}
      \sgn(\beta\e - \beta'\e')
      \Gamma(1/2, \pi (\beta\e  - \beta'\e')^2\vb)
      \sqrt{\vb} K_{r, \nu}(\vb, y_1, y_2) \frac{d\vb}{\vb},\\
    K_{r, \nu}(\vb, y_1, y_2)
    &:=     e^{-\pi \lp \frac{A}{\sqrt{\vb}} - B\sqrt{ \vb}   \rp^2}
      \lp \frac{A }{\sqrt{\vb}} + B \sqrt{\vb} \rp,~
    A:= r \sqrt{y_1y_2},~ B := \frac{\nu y_1 + \nu' y_2}{\sqrt{y_1y_2}}.
  \end{align*}
  So if we set
  \begin{align}
    \label{eq:tc}
    \tc_t(\varphi_f, \vc) &:=
    \sum_{r \in \Qb^\times}    \Fcr_1(\varphi_{0, f})((0, r), t'/r) \tc_{r, t'/r}(\varphi_{1, f}, \vc),\\
    \label{eq:et}
    e_t(\varphi_f; y_1, y_2) &:=
    \sum_{r \in \Qb^\times}    \Fcr_1(\varphi_{0, f})((0, r), t'/r) e_{r, t'/r}(\varphi_{1, f};y_1, y_2),
  \end{align}
  then equation \eqref{eq:I+FE} holds by \eqref{eq:I02}.
  Since $\varphi_f$ is a Schwartz function, the sum defining $\tc_t$ is finite, and equals to 0 when $t \not\in M^{-1} \Oc$ for some $M \in \Nb$ depending only on $\varphi_f$.

  To prove \eqref{eq:betafac}, notice that
  \begin{align*}
    \tc_t(\varphi_f, \vc) &=
              4\vol(K_\vc)
    \sum_{r \in \Qb^\times}
\displaystyle\sum_{\substack{\beta \in \Gamma_\vc \backslash F^\times\\
        1 \le |\beta/\beta'| \le \ve_{\vc}^2\\
    \beta\beta' = tt'/r^2 > 0\\ \xi \in C}}              \vc(\xi)
\Fcr_1(\varphi_{ f})((0, r), t'/r, \xi^{-1} \beta)      \sgn(r\beta)
    \log_{} \sqrt{|\beta/\beta'|}
  \end{align*}
  By Theorem \ref{thm:localmatch} and Lemma \ref{lemma:hatvarphirat}, there exists $c \in \Nb$ such that $2c \Fcr_1(\varphi_f)((0, r), \nu, \lambda) \in \Zb$ for all $r \in \hat\Qb$ and $\nu, \lambda \in \hat F$.
  Then we can write
  $$
  \tc_t(\varphi_f, \vc) = -\frac{2}{\kappa}\log \left| \frac{\beta(t, \varphi_f)}{\beta(t, \varphi_f)'}\right|,~
  \beta(t, \varphi_f) :=
    \prod_{r \in \Qb^\times}
    \displaystyle\prod_{\substack{\beta \in \Gamma_\vc \backslash F^\times\\
        1 \le |\beta/\beta'| \le \ve_{\vc}^2\\  \beta\beta' = tt'/r^2 > 0\\ \xi \in C}}
(r \beta)^{
-\vol(K_\vc)    \vc(\xi)
\kappa \Fcr_1(\varphi_{ f})((0, r), t'/r, \xi^{-1} \beta)      \sgn(r\beta) }
  $$
  For   any split rational prime $p = \pf \pf'$ with any $\pf \not\in S_C$, we have
  \begin{align*}
      \kappa^{-1} &\ord_{\pf} \beta(t, \varphi_f)
  =
-\vol(K_\vc)    \sum_{r \in \Qb^\times}
    \sum_{\substack{\beta \in \Gamma_\vc \backslash F^\times\\
    \beta\beta' = tt'/r^2 > 0\\ \xi \in C}}
    \vc(\xi)
     \Fcr_1(\varphi_{ f})((0, r), t'/r, \xi^{-1} \beta)      \sgn(r\beta)
    \ord_{\pf} (\xi^{-1} r\beta)\\
&= - \vol(K_\vc)    \sum_{r \in \Qb^\times}
    \sum_{\substack{h \in \Gamma_\vc\backslash H_1(\Qb)^+\\
    \xi \in C}}
    \vc(\xi)
     \Fcr_1(\varphi_{ f} - \omega_f(-1)\varphi_f)((0, r), t'/r, \xi^{-1} h^{-1} t/r)      \sgn(h^{-1} t)
    \ord_{\pf} (\xi^{-1} h^{-1} t)\\
&= 2 \sum_{r \in \Qb^\times}
    \int_{H_1(\hat\Qb)}
    \vc(h_1)
     \Fcr_1(\varphi_{ f})((0, r), t'/r, -h_1^{-1}  t/r)
    \ord_{\pf} (h_1^{-1} t') dh_1
  \end{align*}
  since $t \gg 0, \ord_\pf(\beta) = \ord_{\pf}(\xi^{-1} \beta)$
  and $\vc_f(h) = \sgn(h)$, where $h_1 = -\xi h \in H_1(\hat\Qb)$ and $\vc(h_1) = - \vc(\xi)$.
The last step follows from \eqref{eq:C}.

Notice that the quantity above factors as the following product of sums of local integrals
\begin{align*}
  \kappa^{-1} \ord_{\pf} \beta(t, \varphi_f)
  &=    2 \prod_{\ell < \infty,~ \ell \neq p}
    \lp \sum_{n \in \Zb}
    \int_{H_1(\Qb_\ell)}
    \vc_\ell(h_{1, \ell})
     \Fcr_1(\varphi_{ f, \ell})((0, \ell^n), t'/\ell^n, -h_{1, \ell}^{-1}  t/\ell^n)
           dh_{1, \ell}\rp\\
&  \times \lp \sum_{n \in \Zb}
    \int_{H_1(\Qb_p)}
    \vc_p(h_{1, p})
     \Fcr_1(\varphi_{ f, p})((0, p^n), t'/p^n, -h_{1, p}^{-1}  t/p^n)
           \ord_{\pf} (h_{1, p}^{-1} t')
           dh_{1, p}\rp.
\end{align*}
Applying \eqref{eq:FCmatch} turns the first line on the right hand side into $2 \prod_{v < \infty,~ v \nmid p} W^*_{t, v}(\phi_v)$.
If this is non-zero, then $\mathrm{Diff}(W_\alpha, t)$ is either $\{\pf\}$ or $\{\pf'\}$ as it has odd size.
If $\mathrm{Diff}(W_\alpha, t) = \{\pf\}$, then $W^*_{t, \pf}(\phi_\pf) = 0$ and Proposition \ \ref{prop:FCmatch} tell us that the second line on the right hand side becomes
\begin{align*}
  \frac{    W^{*, \prime}_{t, \pf}(\phi_\pf)      W^{*}_{t, \pf'}(\phi_{\pf''})}{\log p}
  +      W^{*}_{t, \pf}(\phi_\pf)      W^{*}_{t, \pf'}(\phi_{\pf'}) \ord_{\pf}(t')
  =   \frac{    W^{*, \prime}_{t, \pf}(\phi_\pf)      W^{*}_{t, \pf'}(\phi_{\pf'})}{\log p}
\end{align*}
as $\ord_{\pf} (h_{1, p}^{-1} t') = \ord_{\pf} (h_{1, p}^{-1}) + \ord_{\pf}( t')$.
This gives us
$$
\kappa^{-1} \ord_{\pf} \beta(t, \varphi_f)
= \tW_t(\phi_f)/2.
$$
Repeat the above argument together with Remark \ref{rmk:vv'}, we obtain
$\kappa^{-1} \ord_{\pf'} \beta(t, \varphi_f)= -\tW_t(\phi_f)/2$.
Putting this together gives us \eqref{eq:betafac}, where the case with $\mathrm{Diff}(W_\alpha, t) = \{\pf'\}$ is obtained similarly.

  Now we will prove \eqref{eq:et0}.
  Since $\varphi$ has compact support, the summation over $\xi$ and $\beta$ in $e_{r, \nu}$ and the summation over $r$ in \eqref{eq:et} are finite sums, it suffices to establish \eqref{eq:et0} with $e_t$ replaced by $e^*_{r, \nu}$ and $e^\dagger_{r, \nu, \beta}$ with $r > 0$.
For any fixed $C, \e > 0$, $s \in \Rb$ and $a, b, c \in \Nb_0$, we have
\begin{align*}
  \left|
y^c \partial^a_{y} \partial^b_{y'}
\int^{A^{1 - \e}}_0
e^{-C \vb} \vb^s
K_{r, \nu}(\vb, y, y') \frac{d\vb}{\vb} \right|
&\ll
\int^{A^{1 - \e}}_0
K_{r, \nu}(\vb, y, y') \frac{d\vb}{\vb}
\ll e^{- A^{\e/2}} ,\\
  \left|
y^c \partial^a_{y} \partial^b_{y'}
\int^\infty_{A^{1 - \e}}
e^{-C \vb} \vb^s
K_{r, \nu}(\vb, y, y') \frac{d\vb}{\vb} \right|
&\ll
\int^\infty_{A^{1 - \e}}
e^{-C\vb/2} d\vb
\ll e^{- A^{1/2-\e/2}}
\end{align*}
when $B$ is in a compact subset of $\Rb$ and $A > 0$ is sufficiently large.
Furthermore, it is easy to see that there exists $C > 0$ such that
  $$
|\Gamma(0, 4\pi |\nu\nu'| \vb)| \ll \vb^{-1} e^{-C \vb},~
\left|  \sum_{\ve \in \Gamma_1}
      \sgn(\beta\ve - \beta'\ve')
      \Gamma(1/2, \pi (\beta\ve  - \beta'\ve')^2\vb) \right|
    \ll \vb^{-1/2} e^{-C \vb}
      $$
for all $\vb > 0$.
Combining these then proves \eqref{eq:et0}.
\end{proof}

\subsection{Rationality of Theta Lifts}
\label{subsec:Millson}
Recall that the rational quadratic space $V_\alpha$ is the restriction of scalars of the $F$-quadratic space $W_\alpha$.
The following result shows that the Millson theta lift preserves rationality.
\begin{proposition}
  \label{prop:millson}
  Let $f = \sum_{\mu \in L_{}^\vee/L_{}} f_\mu \ef_\mu \in M^!_{-2r, \rho_{L_{}}}$ be weakly holomorphic for some $r \in \Nb$ and lattice $L_{} \subset V_{\alpha}$.
  For $\varphi^{(1, r)}_{} = \varphi^{(1, r)}_{00, \infty} \varphi_f$ with $\varphi_{  f} \in \Sc(\hat V_{00}; \Qab)^{\TT(\hat\Zb)}$, let $\Gamma(N) \subset \SL_2(\Zb)$ be a congruence subgroup contained in $\ker(\rho_L)$ that fixes $\varphi_f$. The following regularized integral\footnote{The regularization is the same as in \cite{BF04} or \cite{ANS18}.}
  \begin{equation}
    \label{eq:Millson}
 I^M_{}(\tau, f_\mu, \varphi_{f}) :=
 \frac{1}{[\mathrm{SL}_2(\Zb): \Gamma(N)]}
 \int^{\mathrm{reg}}_{\Gamma(N) \backslash\Hb } y^r f_\mu(z)
 \theta_{00}(g_\tau, h_z, \varphi^{(1, r)}_{}) d\mu(z)
                 \end{equation}
                                 defines a weakly holomorphic modular form of weight $-r + 1/2 < 0$.
  Suppose $f$  has rational Fourier coefficients at the cusp $\infty$,  so does $I^M(\tau, f_\mu, \varphi_f)$ for all  $\mu \in L_{}^\vee/L_{}$.
\end{proposition}
\begin{remark}
  \label{rmk:r=0}
  When $r = 0$ and $f$ has vanishing constant term, the same proof shows that the weakly holomorphic modular form $I^M(\tau, f_\mu, \varphi_f)$ has rational Fourier coefficients up to algebraic multiples of weight $1/2$ unary theta series.
\end{remark}
\begin{proof}
  We will use the Fourier expansion of Millson theta lift calculated in Theorem 5.1 of \cite{ANS18}, which we now recall.
Fix an orientation on $V_{00}(\Rb)$.
  For an isotropic line $\ell \subset V_{00}$, let $G_{00, \ell} \subset G_{00}$ be its stabilizer and $\gamma_\ell \in \SL_2(\Zb)$ such that $\gamma_\ell^{-1} G_{00, \ell} \gamma_\ell = G_{00, \ell_\infty}$ with $\ell_\infty = \Qb  v_\infty$ and $v_\infty := \smat{1}{0}{0}{0}$.
  Denote  $c_\ell(m, \mu)$  the $m$-th Fourier coefficient of $f_\mu \mid_{-2r} \gamma_\ell$.
  If $x \in V_{00}(\Qb)$ satisfies $\sqrt{-Q(x)} = d \in \Qb_{> 0}$, then
  $x^\perp$ is a hyperbolic plane spanned by two isotropic lines $\ell_x$ and $\ell_{-x}$ such that $x, \gamma_{\ell_x}v_\infty,\gamma_{\ell_{-x}}v_\infty$ is positively oriented.
We can then define  $r_{x} \in \Qb$ by
$$
\gamma_{\ell_{ x}}^{-1} \cdot x = -d \pmat{2r_{ x}}{1}{1}{0}.
$$

  Suppose $r \ge 1$.
  From \cite[Theorem 5.1]{ANS18}, we know that $[\SL_2(\Zb): \Gamma(N)] \cdot I^M_{}$ is weakly holomorphic of weight $-r + 1/2 < 0$ with principal part
  given by   \footnote{Up to a sign $(-1)^{r+1}$ depending on the orientation of $V_{00}(\Rb)$.}
  $$
  \sum_{d > 0} \frac{\ebf(-d^2 \tau)}{2d^{1+r}}
  \sum_{\substack{x \in \Gamma_{L}\backslash V_{00, -d^2}(\Qb)\\
        w \in \Qb_{ < 0}}}
w^k
  \varphi_{   f}(x)
(  c_{\ell_x}(w, \mu)   \ebf(r_{x} w)
+ (-1)^{r+1}  c_{\ell_{-x}}(w, \mu)  \ebf(r_{-x} w)) \in \Qab.
  $$
  Note that the inner sum above vanishes for $d$ sufficiently large by Proposition 4.7  in \cite{BF04}, and $I^M$ is uniquely determined by its principal part because its weight is negative.
 Now we can enlarge $N$ such that
 $      N r_{\pm x} w \in \Zb $ whenever $ c_{\ell_{\pm x}}(w, \mu)  \neq 0$.
Then given a prime $p \nmid N$, for an element $x \in \Gamma_{L}\backslash V_{00, -d^2}(\Qb)$ to have a representative $\tx \in V_{00}$  such that both $t(p) \cdot \tx$ and $\tx$ are both $p$-integral is equivalent to find a $p$-integral representative $\smat{A}{B}{B}{C}$ with $p \nmid A$.
  Note that the set
  $$
  S_d(\varphi) := \{x \in \Gamma_{L}\backslash V_{00, -d^2}(\Qb): \varphi_{}(x) \neq 0
  \}
  $$
is a finite set for any $\varphi \in \Sc(\hat V_{00})$.

  For any $\sigma_a \in \Gal(\Qb^\ab/\Qb)$ with $a \in \hat\Zb^\times$, we have $t(a) \in T(\hat\Zb) \subset \GL_2(\hat\Zb)$.
  Choose an odd prime $p \nmid N$ such that $a \equiv p \bmod{N}$ and every $ x\in S_d(\varphi_{  f})$ has a $p$-integral representative $\tx \in V_{00}(\Qb)$ such that $t(p) \cdot \tx$ is $p$-integral.
Let $\tilde{S}_d(\varphi_{  f})$ be such a set of representatives.

  Denote $\tx' := t(p)\cdot \tx$ for $\tx \in \tilde{S}_d(\varphi_{  f})$, which is $p$-integral
  and satisfies  $\ell_{\pm \tx' } = t(p) \cdot \ell_{\pm \tx}$ and
  $$
  \gamma_{\ell_{\pm \tx'}} \equiv
  t(p) \gamma_{\ell_{\pm \tx}} t(p)^{-1} \equiv
  t(a) \gamma_{\ell_{\pm \tx}} t(a)^{-1}
  \bmod{N},~ r_{\pm \tx'}w - p r_{\pm \tx}w \in \Zb.
  $$
when $c_{\ell_{\tx}}(w)$ or  $c_{\ell_{-\tx}}(w)$ is non-zero.
By equation \eqref{eq:weilGL2}, $\Gamma(N) \subset \ker(\rho_L)$ and the fact that $f$ has rational Fourier coefficients, we then have
  $$
  \sigma_a( f \mid_{-2r} \gamma_{\ell_{\pm \tx}})
  =\rho_L(t(a) \gamma_{\ell_{\pm \tx}}) f
  =\rho_L(t(a) \gamma_{\ell_{\pm \tx}} t(a)^{-1})f
  =\rho_L(\gamma_{\ell_{\pm \tx'}})f
  =  f \mid_{-2r}  \gamma_{\ell_{\pm \tx'}}.
  $$
These imply that
  \begin{equation}
    \label{eq:Galc}
    \sigma_a(c_{\ell_{\pm \tx}}(w, \mu)) = c_{\ell_{\pm \tx'}}(w, \mu),~
    \sigma_a(\ebf(r_{\tx} w)) =     \ebf(p r_{\tx} w)     =     \ebf(r_{\tx'} w)
  \end{equation}
for all $d > 0$, $\tx \in \tilde S_d(\varphi_f)$ and $w \in \Qb_{< 0}$.
Finally, we have $\varphi_{  f}(\tx) =         \varphi_{  f}(t(p)^{-1}\tx')$, which implies
  \begin{equation}
    \label{eq:isom}
    \begin{split}
      \varphi_{  f}(\tx)) =
      \varphi_{  f}(t(a)^{-1}\tx') =
      \omega(\iota(t(a)))\varphi_{  f})(\tx').
    \end{split}
  \end{equation}
  since $\varphi_{  f} \in \Sc(L_{})$ and $p$ is co-prime to the level of $L_{}$.
  Here $\iota$ is the map defined in \eqref{eq:GSpin}.
  The map $\tx \mapsto \tx'$ then gives a bijection between $S_d(\varphi_{  f})$
  and $S_d(\omega(\iota(t(a)))\varphi_{  f})$.
  From this, we then obtain
\begin{equation}
  \label{eq:saact}
\sigma_a (I^M_{}(\tau, f_\mu, \varphi_{  f})) =
I^M_{}(\tau, f_\mu, \omega(\iota(t(a))) \sigma_a( \varphi_{  f})) =
I^M_{}(\tau, f_\mu, \omega(t(a), \iota(t(a)))  \varphi_{  f}).
\end{equation}
As $(t(a), \iota(t(a))) \in \TT(\hat\Zb)$, and  $\varphi_{f}$ is $\TT(\hat\Zb)$-invariant, the modular form $I^M(\tau, f_\mu, \varphi_f)$ has rational Fourier coefficients.
\end{proof}

From Propositions \ref{prop:mixmock} and \ref{prop:millson}, we can deduce the following result.

\begin{proposition}
  \label{prop:rational}
  Let $r \in \Nb_0$ and $f \in M^!_{-2r, \rho_L}$ as in Proposition \ref{prop:millson}. %
For all $\mu \in L^\vee/L$ and congruence subgroup $\Gamma(N) \subset \ker(\rho_L)$ fixing $\varphi_\mu \in \Sc(\hat V; \Qab)^{(G \cdot \TT)(\hat\Zb)}$, which is a matching section of $\phi_{\mu}$ as in Theorem \ref{thm:match},
  the regularized integral
\begin{equation}
  \label{eq:ratconst}
  c_\mu(f):= \sqrt{D}^{-r}
  \fac
\frac{1}{  [\SL_2(\Zb): \Gamma(N)]}
\int^{\mathrm{reg}}_{\Gamma(N)\backslash \Hb}
v^{r}
f_\mu(\tau) \tRC_r \Ic_f(g_\tau^\Delta, \vc, \varphi_{ \mu}^{(1, -1)} - (-1)^r \varphi_{\mu}^{(-1, 1)} )d\mu(\tau)
\end{equation}
is a rational number.
\end{proposition}

\begin{proof}
Since $\varphi_\mu$ is $G(\hat\Zb)$-invariant, we can rewrite the constant $c_\mu := c_\mu(f)$ as
  \begin{align*}
\sqrt{D}^{r}c_\mu
    = &
     \int^{\mathrm{reg}}_{\Gamma_L\backslash \Hb}
                           v^r f_\mu(\tau)
       \lim_{T' \to \infty}      \int_{\Fcr_{T'}}
\int_{H_1(\Qb)\backslash H_1(\hat\Qb)}
    \theta(g_{\tau'}, (g^\Delta, h_1),\varphi_\mu^r)
    \vc(h_1) dh_1
    d\mu(\tau')                   d\mu(\tau)
  \end{align*}
with $\varphi^r_\mu :=  \tRC_r(\varphi_\mu^{(1, -1)} - (-1)^r \varphi_\mu^{(-1, 1)}) \in \Sc(V(\Ab))$.
Using Lemma \ref{lemma:switch}, we can switch the regularized integral in $g$ with the limit in $T'$.
Then by the rational decomposition $V = V_{00} \oplus U_{D} \oplus V_1$
, we can write
$$
\varphi_{\mu} = \sum_{j \in J}
\varphi_{00,  \mu, j} \otimes \varphi_{D,  \mu, j} \otimes \varphi_{1,  \mu, j},~
$$
with $ \varphi_{00, \mu, j} \in \Sc(\hat V_{00}; \Qab)^{\TT(\hat\Zb)},
\varphi_{1,  \mu, j} \in \Sc(\hat V_1)$ and $\varphi_{D,  \mu, j} \in \Sc(\hat U_D)$.
The constant $c_\mu$ can then be rewritten as
$$
{c}_\mu = \sum_{j \in J}  c_{\mu, j},
$$
where $c_{\mu, j}$ is defined by %
  \begin{align*}
    c_{\mu, j} &:=
 \sqrt{D}^{-r} \vol(K_\vc) \lim_{T' \to \infty}
        \int_{\Fcr_{T'}}
\Theta_{1}(g_{\tau'}, \varphi_{1, \mu, j}^-, \vc)\\
&        \times \int^{\mathrm{reg}}_{\Gamma(L) \backslash \Hb}                             v^r f_\mu(\tau)
    \theta_0(g_{\tau'}, g^\Delta, \tRC_r(\varphi_{0, \mu, j}^{(1, -1)} - (-1)^r \varphi_{0, \mu, j}^{(-1, 1)}))
                      d\mu(\tau)                   d\mu(\tau').
  \end{align*}
    Now with  Lemma \ref{lemma:RC'}, we obtain
  \begin{align*}
    \frac{c_{\mu, j}}{2^{2r_0 - r + 1}}
    &=
\vol(K_\vc)      \lim_{T' \to \infty}
        \int_{\Fcr_{T'}}
(v')^{-1/2} \Theta_{1}(g_{\tau'}, \varphi_{1, \mu, j}^-, \vc)
      G_{\mu, j}(\tau')       d\mu(\tau'),
  \end{align*}
  where $G_{\mu, j}$   is a weakly holomorphic modular form of weight $-1$ defined by
  $$
  G_{\mu, j}(\tau') :=     \sqrt{v'}        \sqrt{D}^{ - 3r} \RC^{', \Delta}_{r_0, (-r + 1/2, r - 2r_0 + 1/2)}
  \lp \theta_D(g_{\tau_2'}, \varphi^{r - 2r_0}_{D, \mu, j}) I^M(\tau_1', f_\mu, \varphi_{00, \mu, j})\rp\mid_{\tau_1' = \tau_2' = \tau'},
  $$
  and has rational Fourier coefficient at the cusp $\infty$ by Proposition \ \ref{prop:rational}.
  As $\varphi_\mu$ is $\SL_2(\hat\Zb)$-invariant, the function
  $(v')^{-1/2} \sum_{j \in J} \Theta_{1}(g_{\tau'}, \varphi_{1, \mu, j}^-, \vc)      G_{\mu, j}(\tau')$ is $\SL_2(\Zb)$-invariant in $\tau'$.
  Applying Proposition \ \ref{prop:mixmock} and Stokes' Theorem, we then have
  \begin{align*}
    \frac{c_{\mu}}{2^{2r_0 - r + 1}}
    &=
\vol(K_\vc)      \lim_{T' \to \infty}
        \int_{\Fcr_{T'}}
      \sum_{j \in J}
      L_{\tau'} \lp \sqrt{v'} \tilde\Theta_{1, C}(g_{\tau'}, \varphi_{1, \mu, j}^-, \vc) \rp
      G_{\mu, j}(\tau', j)       d\mu(\tau')\\
    &=
-\vol(K_\vc)      \sum_{j \in J} \mathrm{CT}
\lp      \sqrt{v'}\tilde\Theta^+_{1, C}(g_{\tau'}, \varphi_{1, \mu, j}^-, \vc)
      G_{\mu, j}(\tau', j)       \rp \in \Qb.
  \end{align*}
  This finishes the proof.
\end{proof}

\section{Proofs of Theorems }
\label{sec:ThmPfs}
In this section, we will prove Theorem \ref{thm:factor}.
First, we state and prove the case for $\mathrm{O}(2, 2)$.

\begin{theorem}
  \label{thm:main-O22}
  Let $F$ be a real quadratic field and $W = W_\alpha$ an $F$-quadratic space and $W_{\alv}$ its neighborhood quadratic space as in section \ref{subsec:CM}.
Suppose $\alpha_1 < 0 < \alpha_2$.
For  $r \in \Nb_0$ and a lattice $L \subset W_\Qb$, suppose
  $$
  f = \sum_{m \in \Qb,~ \mu \in L^\vee/L} c(m, \mu) q^m \ef_\mu \in M_{-2r, \rho_L}^!
  $$
  is  a weakly holomorphic modular form with rational Fourier coefficients.
  Furthermore, suppose it has vanishing constant term when $r = 0$.
Then  there exists $\kappa, M \in \Nb$ depending on $f$ such that
  \begin{equation}
    \label{eq:main1}
    \begin{split}
\kappa&\lp           \Phi^r_f(Z(W_\alpha)) - (-1)^r \Phi^r_f(Z(W_{\alv})) \rp\\
          &=
- \frac{\deg(Z(W ))}{\Lambda(0, \chi)}
          \sum_{m > 0,~ \mu \in L^\vee/L} c(-m, \mu)m^r \sum_{\substack{\lambda \in F^\times \cap M^{-1}\Oc\\ \lambda\gg 0\\ \tr(\lambda) = m}}
P_r\lp \frac{\lambda - \lambda'}{m}\rp
\log\left| \frac{\beta(\lambda, \phi_{\mu} )}{\beta(\lambda, \phi_{\mu})'}\right|,
    \end{split}
  \end{equation}
  where $\beta(t, \phi_f) \in F^\times$ is non-zero which has the property
  \begin{equation}
    \label{eq:betafac2}
        \kappa^{-1}    \ord_\pf(\beta(t, \phi_f)/\beta(t, \phi_f)') =
    \begin{cases}
     \tW_t(\phi_f)      ,&\hbox{if }     \mathrm{Diff}(W , t) = \{\pf\},\\
-\tW_t(\phi_f)      ,&\hbox{if }     \mathrm{Diff}(W , t) = \{\pf'\},\\
      0,&\hbox{otherwise. }
    \end{cases}
  \end{equation}
\end{theorem}

\begin{proof}
  By the Siegel-Weil formula in \eqref{eq:SWCM}, we have
  $$
  \Phi^r_f(Z(W )) =
\frac{C'}{(-4\pi)^r}  \cdot       \int^{\reg}_{\SL_2(\Zb)\backslash \Hb}
  v^r    \sum_{\mu \in L^\vee/L} f_\mu(\tau)
      R^r_\tau  E^*(g_\tau^\Delta, \phi_{\mu}^{(\sgn(\alpha), -\sgn(\alpha))})
      d\mu(\tau)
  $$
  with $C' := \deg(Z(W ))/(2\Lambda(0, \chi)) \in \Qb^\times$.
For each $\mu \in L^\vee/L$, let $\varphi_{\mu} \in \Sc(\hat V)^{(G\cdot \TT)(\hat\Zb)}$ be a matching section of $\phi_{\mu}$ as in Theorem \ref{thm:match}.
Then we can apply Proposition \ref{prop:diffop} to obtain
  \begin{align*}
    \Phi^r_f&(Z(W_{\alv} )) - (-1)^r \Phi^r_f(Z(W_{\alpha} ))
    \\
&= \frac{C'}{(-4\pi)^r} \cdot     \int^{\reg}_{\SL_2(\Zb)\backslash \Hb}
      \sum_{\mu \in L^\vee/L} f_\mu(\tau)
      R^r_\tau \lp E^*(g_\tau^\Delta, \phi_{\mu}^{(1, -1)}) -(-1)^r E^*(g_\tau^\Delta, \phi_{\mu}^{(-1, 1)}) \rp d\mu(\tau)\\
&= - {C'} \cdot \fac      \int^{\reg}_{\SL_2(\Zb)\backslash \Hb}
      \sum_{\mu \in L^\vee/L} f_\mu(\tau)
  L_\tau (\RC_r \tI)(g_\tau^\Delta, \tvc_C, \varphi^{(1, 1)}_\mu) d\mu(\tau) \\
            &+ \frac{C' 2 \log \ve_\vc}{[\SL_2(\Zb) : \Gamma(N)]}       \sum_{\mu \in L^\vee/L}
\fac \int^\reg_{\Gamma(N) \backslash \Hb}              f_\mu(\tau)  \tRC_r \Ic_f(g_\tau^\Delta, \vc, \varphi^{(1, -1)}_\mu - (-1)^r\varphi^{(-1, 1)}_\mu) d\mu(\tau)\\
&= - {C'} \cdot \fac      \int^{\reg}_{\SL_2(\Zb)\backslash \Hb}
      \sum_{\mu \in L^\vee/L} f_\mu(\tau)
  L_\tau (\RC_r \tI)(g_\tau^\Delta, \tvc_C, \varphi^{(1, 1)}_\mu) d\mu(\tau)
            + {2 C' \sqrt{D}^r \log \ve_\vc
   \sum_{\mu \in L^\vee/L} c_\mu(f) }.
  \end{align*}
  By Proposition \ref{prop:rational}, we know that $c_\mu(f) \in \Qb$ for all $\mu \in L^\vee/L$.
  For the other term, we can apply Stokes' theorem to obtain
  \begin{align*}
 {C'}& \cdot \fac \lim_{T \to \infty}      \int^{}_{\Fcr_T}
      \sum_{\mu \in L^\vee/L} f_\mu(\tau)
  L_\tau (\RC_r \tI)(g_\tau^\Delta, \tvc_C, \varphi^{(1, 1)}_\mu)d \mu(\tau)\\
 &   =
     {C'} \cdot \fac \lim_{v \to \infty}
         \int_0^1 \sum_{\mu \in L^\vee/L} f_\mu(\tau)
(\RC_r    \tI)(g_\tau^\Delta, \tvc_C, \varphi^{(1, 1)}_\mu)    du \\
    &=  C' \sum_{m > 0,~ \mu \in L^\vee/L} c(-m, \mu)
\sum_{\lambda \in F^\times,~ \lambda \gg 0,~ \tr(\lambda) = m}
m^r      P_r\lp \frac{\lambda - \lambda'}{m}\rp
\tc_\lambda(\varphi_{\mu}, \vc)
  \end{align*}
  For the last step, we have applied Remark \ref{rmk:tI} to replace $\tI$ with $\Ic^0 + \Ic^+$, used Lemma \ref{lemma:decay} to see that $\Ic^0$ contribute nothing, and substitute in  the Fourier coefficients of $\RC_r\Ic^+$ in terms of   $\tc_\lambda( \varphi_{\mu}, \vc)$,  the $\lambda$-th Fourier coefficient of $\tI(g_\tau^\Delta, \varphi^{(1, 1)}_\mu, \tvc_C)$.
  Note that $m^r P_r((\lambda - \lambda')/m)$ appears by \eqref{eq:RCexp} and is a rational multiple of  $\sqrt{D}^r$.
As the sum above is finite, we can choose $C$ such that $S_C^c$ contains $\Diff(W , t)$ for all the $t$ that appears in this sum.
  Finally, the knowledge about the factorization of these coefficients in Proposition \ref{prop:tI+} finishes the proof.
\end{proof}

\begin{corollary}
  \label{cor:1-CM}
  In the setting of Theorem \ref{thm:main-O22}, suppose that $Z_f$ does not intersect with  $Z(W )$ when $r = 0$.
Then  there exists $\kappa \in \Nb$ and $\gamma(\lambda, \phi_\mu) \in F^\times$ such that
  \begin{equation}
    \label{eq:main2}
    \begin{split}
\kappa            \Phi^r_f(Z(W ))
          &=
-2 \frac{\deg(Z(W ))}{\Lambda(0, \chi)}
\sum_{m > 0,~ \mu \in L^\vee/L} c(-m, \mu) m^r
\sum_{\substack{\lambda \in F^\times \cap M^{-1}\Oc\\ \lambda\gg 0\\ \tr(\lambda) = m}}
P_r\lp \frac{\lambda - \lambda'}{m}\rp
\log|\gamma(\lambda, \phi_{\mu})|
    \end{split}
  \end{equation}
and %
  \begin{equation}
    \label{eq:gammafac}
        \kappa^{-1}    \ord_\pf(\gamma(t, \phi_f)) =
    \begin{cases}
      \tW_t(\phi_f)      ,&\hbox{if }     \mathrm{Diff}(W , t) = \{\pf\},\\
      0,&\hbox{otherwise. }
    \end{cases}
  \end{equation}
\end{corollary}
\begin{remark}
  The constant $\frac{\deg(Z(W ))}{\Lambda(0, \chi)}$ can be explicitly given when $X_K = X_0(1)^2$ (see \cite[Remark 3.6]{Li21a}).
\end{remark}
\begin{proof}
  By Theorem 5.10 in \cite{BEY21},
  we have
  \begin{align*}
(           \Phi^r_f(Z(W_\alpha )) &+ (-1)^r \Phi^r_f(Z(W_{\alv} )) )\\
 &         =
\frac{\deg(Z(W ))}{\Lambda(0, \chi)}
          \sum_{m > 0,~ \mu \in L^\vee/L} c(-m, \mu)m^r \sum_{\substack{\lambda \in F^\times\\ \lambda\gg 0\\ \tr(\lambda) = m}}
P_r\lp \frac{\lambda - \lambda'}{m}\rp
a_\lambda(\phi_\mu),
  \end{align*}
  where $a_t(\phi_\mu)$ is $t$-th the Fourier coefficient of the holomorphic part of the incoherent Eisenstein series, and given in \eqref{eq:at}.
    Adding this to \eqref{eq:main1} and applying \eqref{eq:betafac2} finishes the proof.
  \end{proof}

  Now we can prove Theorem \ref{thm:factor}.

  \begin{proof}[Proof of Theorem \ref{thm:factor}]
    Write $\V = \Vc \oplus W_\Qb$ and    $L_W := L \cap W_\Qb,~ \Lc := L \cap \Vc     $.
    Then $\Lc \oplus L_W \subset L$ is a full sublattice, and we can write
    $$
    \langle f(\tau), \overline{\Theta_L(\tau, Z(W))}\rangle_L
    =
    \langle \tf(\tau, \tau),
\overline{\Theta_{L_W}(\tau, Z(W))     }   \rangle_{L_W}
    $$
    with $\tf(\tau_1, \tau_2) :=    \langle \tr^L_{\Lc \oplus L_W}(f(\tau_1)), \overline{\Theta_{\Lc}(\tau_2)}\rangle_{\Lc}$.
    Using \eqref{eq:Rtf}, we have
    \begin{align*}
          (4\pi)^{-r}&    \langle R^r_{\tau} f(\tau), \overline{\Theta_L(\tau, Z(W))}\rangle_L
=    (4\pi)^{-r} (R_{\tau_1}^r
    \langle f(\tau_1), \overline{\Theta_L(\tau, Z(W))}\rangle_L  \mid_{\tau_1 = \tau}  \\
      &=    \langle    (4\pi)^{-r} (R_{\tau_1}^r (\tf)^\Delta,
\overline{\Theta_{L_W}(\tau, Z(W))     }   \rangle_{L_W}\\
        &= \sum_{\ell = 0}^r c^{(r; k_1, k_2)}_\ell (4\pi)^{-r+\ell} \langle R_{\tau}^{r-\ell}\tf_\ell , \overline{\Theta_{L_W}(\tau, Z(W))     }   \rangle_{L_W},
    \end{align*}
    where $k_1 = -2r + 1 - \frac{n}{2}, k_2 = \frac{n}{2}-1$ and
    $$
    \tf_\ell :=
    \RC_{\ell, (k_1, k_2)}(\tf)^\Delta \in M^!_{-2r+2\ell, L_W}
    $$
    has rational Fourier coefficients.
    Therefore, we have
    \begin{equation}
      \label{eq:n2decomp}
      \Phi^r_f(Z(W))
      =  \sum_{\ell = 0}^r c^{(r; k_1, k_2)}_\ell \Phi^{r-\ell}_{\tf_\ell}(Z(W)).
    \end{equation}
If $f$ has the Fourier expansion
$$
f(\tau) = \sum_{\nu \in L^\vee/L,~ n \in \Zb + Q(\nu)} c(n, \nu) q^{n} \ef_\nu,
$$
then the $(m, \mu)$-th Fourier coefficient of $\tf_\ell$, denoted by $c_\ell(m, \mu)$, can be expressed as
$$
c_\ell(m, \mu) = \sum_{\lambda_\circ \in \Lc^\vee} Q_{\ell, (k_1, k_2)}(m-Q(\lambda_\circ), Q(\lambda_\circ)) c(m- Q(\lambda_\circ), (\lambda_\circ, \mu)).
$$
with $Q_{\ell, (k_1, k_2)}(X, Y) \in \Qb[X, Y]$ defined in \eqref{eq:Qs}.
In particular when $Z(W) \cap Z_f = \emptyset$, we have $c_r(0, \mu) = 0$ for all $0 \le \ell \le r$ and $\mu \in L^\vee/L$ as $c(-Q(\lambda_\circ), (\lambda_\circ, \mu)) = 0$ for all $\lambda_\circ \in \Lc^\vee$ and $\mu \in L_W^\vee/L_W$ by \eqref{eq:singint}.

    By Corollary \ref{cor:1-CM}, we can write
    \begin{align*}
      \kappa    \Phi^r_f(Z(W)) =
-2 \frac{\deg(Z(W))}{\Lambda(0, \chi)}
&    \sum_{\ell = 0}^r c^{(r; k_1, k_2)}_\ell
    \sum_{m > 0,~ \mu \in L_W^\vee/L_W} c_\ell(-m, \mu)m^{r-\ell}\\
&\times    \sum_{\substack{\lambda \in F^\times \cap M^{-1}\Oc\\ \lambda\gg 0\\ \tr(\lambda) = m}}
P_{r-\ell}\lp \frac{\lambda - \lambda'}{m}\rp
\log|\gamma_\ell(\lambda, \phi_{\mu})|
    \end{align*}
    with
    $\gamma_\ell(\lambda, \phi_\mu) \in F^\times$ having factorization as in \eqref{eq:gammafac} independent of $\ell$, and express $\Phi^r_f(Z(W))$ as in \eqref{eq:sumalg} such that
    \begin{equation}
      \label{eq:n2fac}
      \begin{split}
  \kappa^{-1} \ord_\pf(a_j) =
  -2 \frac{\deg(Z(W))}{\Lambda(1, \chi)}
&  \sum_{\substack{m > 0\\ \mu \in L_W^\vee/L_W\\ \lambda_\circ \in \Lc^\vee}}
  c(-m- Q(\lambda_\circ), (\lambda_\circ, \mu))\\
&\times  \sum_{\substack{0 \le \ell \le r\\ r - \ell \equiv j \bmod{2}}} c^{(r; k_1, k_2)}_\ell
  Q_{\ell, (k_1, k_2)}(-m-Q(\lambda_\circ), Q(\lambda_\circ))  \\
&\times    \sum_{\substack{\lambda \in F^\times \cap M^{-1}\Oc\\ \lambda\gg 0\\ \tr(\lambda) = m\\\Diff(W, \lambda) = \{\pf\} }}
\frac{1}{\sqrt{D}^{j \bmod{2}}} P_{r-\ell}\lp \frac{\lambda - \lambda'}{m}\rp \tW_\lambda(\phi_\mu)
\end{split}
\end{equation}
for all prime $\pf$ of $F$.
\end{proof}

Finally, we prove Theorem \ref{thm:main}.

\begin{proof}[Proof of Theorem \ref{thm:main}]
  By the main result in \cite{Li23}, we have $\tkappa \in \Nb$ and Galois equivariant maps $\tilde\alpha_j: T_W(\hat\Qb) \to E^{\mathrm{ab}}$ satisfying
\begin{equation}
  \label{eq:diffalg}
\Phi_{f}^r([z_0, h]) - \Phi_{f}^r([z_0, h'])
=
\frac{1}{\tkappa}
\lp
\log \left|\frac{\tilde\alpha_1(h)}{\tilde\alpha_1(h')}\right|
+ \sqrt{D} \log \left|\frac{\tilde\alpha_2(h)}{\tilde\alpha_2(h')}\right|
\rp
\end{equation}
for all $h, h' \in T_W(\hat\Qb)$.
Furthermore when $n = 2$, we have $\tilde\alpha_j = 1$ for $j \equiv r \bmod{2}$.
  Setting $\alpha_j(h) := a_j \prod_{[z_0, h'] \in Z(W) \backslash [z_0, h]} \frac{\tilde\alpha_j(h)}{\tilde\alpha_j(h')}$ and $\kappa := \tkappa |Z(W)|$ and applying equations \eqref{eq:diffalg} and \eqref{eq:sumalg} proves the first two claims. Combining with Corollary \ref{cor:1}, we see that Conjecture \ref{conj:GZ} holds.
\end{proof}

Conflicts of interest: none.

\bibliography{RETL}{}

\newcommand{\etalchar}[1]{$^{#1}$}
\begin{thebibliography}{BvdGHZ08}

\bibitem[AGHMP18]{AGHMP18}
Fabrizio Andreatta, Eyal~Z. Goren, Benjamin Howard, and Keerthi Madapusi~Pera.
\newblock Faltings heights of abelian varieties with complex multiplication.
\newblock {\em Ann. of Math. (2)}, 187(2):391--531, 2018.

\bibitem[ANS18]{ANS18}
Claudia Alfes-Neumann and Markus Schwagenscheidt.
\newblock On a theta lift related to the {S}hintani lift.
\newblock {\em Adv. Math.}, 328:858--889, 2018.

\bibitem[AS64]{AS64}
Milton Abramowitz and Irene~A. Stegun.
\newblock {\em Handbook of mathematical functions with formulas, graphs, and
  mathematical tables}, volume~55 of {\em National Bureau of Standards Applied
  Mathematics Series}.
\newblock For sale by the Superintendent of Documents, U.S. Government Printing
  Office, Washington, D.C., 1964.

\bibitem[Bei87]{Bei87}
A.~A. Beilinson.
\newblock Height pairing between algebraic cycles.
\newblock In {\em {$K$}-theory, arithmetic and geometry ({M}oscow,
  1984--1986)}, volume 1289 of {\em Lecture Notes in Math.}, pages 1--25.
  Springer, Berlin, 1987.

\bibitem[BEY21]{BEY21}
Jan~Hendrik Bruinier, Stephan Ehlen, and Tonghai Yang.
\newblock C{M} values of higher automorphic {G}reen functions for orthogonal
  groups.
\newblock {\em Invent. Math.}, 225(3):693--785, 2021.

\bibitem[BF04]{BF04}
Jan~Hendrik Bruinier and Jens Funke.
\newblock On two geometric theta lifts.
\newblock {\em Duke Math. J.}, 125(1):45--90, 2004.

\bibitem[BHK{\etalchar{+}}20]{BHKRY20a}
Jan~H. Bruinier, Benjamin Howard, Stephen~S. Kudla, Michael Rapoport, and
  Tonghai Yang.
\newblock Modularity of generating series of divisors on unitary {S}himura
  varieties.
\newblock {\em Ast\'{e}risque}, (421, Diviseurs arithm\'{e}tiques sur les
  vari\'{e}t\'{e}s orthogonales et unitaires de Shimura):7--125, 2020.

\bibitem[BKY12]{BKY12}
Jan~Hendrik Bruinier, Stephen~S. Kudla, and Tonghai Yang.
\newblock Special values of {G}reen functions at big {CM} points.
\newblock {\em Int. Math. Res. Not. IMRN}, (9):1917--1967, 2012.

\bibitem[Blo84]{Bloch84}
Spencer Bloch.
\newblock Height pairings for algebraic cycles.
\newblock In {\em Proceedings of the {L}uminy conference on algebraic
  {$K$}-theory ({L}uminy, 1983)}, volume~34, pages 119--145, 1984.

\bibitem[Bor98]{Borcherds98}
Richard~E. Borcherds.
\newblock Automorphic forms with singularities on {G}rassmannians.
\newblock {\em Invent. Math.}, 132(3):491--562, 1998.

\bibitem[Bor99]{Borcherds99}
Richard~E. Borcherds.
\newblock The {G}ross-{K}ohnen-{Z}agier theorem in higher dimensions.
\newblock {\em Duke Math. J.}, 97(2):219--233, 1999.

\bibitem[Bru02]{Bruinier02}
Jan~H. Bruinier.
\newblock {\em Borcherds products on {O}(2, {$l$}) and {C}hern classes of
  {H}eegner divisors}, volume 1780 of {\em Lecture Notes in Mathematics}.
\newblock Springer-Verlag, Berlin, 2002.

\bibitem[BvdGHZ08]{123}
Jan~Hendrik Bruinier, Gerard van~der Geer, G\"{u}nter Harder, and Don Zagier.
\newblock {\em The 1-2-3 of modular forms}.
\newblock Universitext. Springer-Verlag, Berlin, 2008.
\newblock Lectures from the Summer School on Modular Forms and their
  Applications held in Nordfjordeid, June 2004, Edited by Kristian Ranestad.

\bibitem[BY06]{BY06}
Jan~Hendrik Bruinier and Tonghai Yang.
\newblock C{M}-values of {H}ilbert modular functions.
\newblock {\em Invent. Math.}, 163(2):229--288, 2006.

\bibitem[BY07]{BY07}
Jan~Hendrik Bruinier and Tonghai Yang.
\newblock Twisted {B}orcherds products on {H}ilbert modular surfaces and their
  {CM} values.
\newblock {\em Amer. J. Math.}, 129(3):807--841, 2007.

\bibitem[BY09]{BY09}
Jan~Hendrik Bruinier and Tonghai Yang.
\newblock Faltings heights of {CM} cycles and derivatives of {$L$}-functions.
\newblock {\em Invent. Math.}, 177(3):631--681, 2009.

\bibitem[BY11]{BY11}
Jan~H. Bruinier and Tonghai Yang.
\newblock C{M} values of automorphic {G}reen functions on orthogonal groups
  over totally real fields.
\newblock In {\em Arithmetic geometry and automorphic forms}, volume~19 of {\em
  Adv. Lect. Math. (ALM)}, pages 1--54. Int. Press, Somerville, MA, 2011.

\bibitem[CL20]{CL20}
Pierre Charollois and Yingkun Li.
\newblock Harmonic {M}aass forms associated to real quadratic fields.
\newblock {\em J. Eur. Math. Soc. (JEMS)}, 22(4):1115--1148, 2020.

\bibitem[DN70]{DN70}
Koji Doi and Hidehisa Naganuma.
\newblock On the functional equation of certain {D}irichlet series.
\newblock {\em Invent. Math.}, 9:1--14, 1969/70.

\bibitem[FM06]{FM06}
Jens Funke and John Millson.
\newblock Cycles with local coefficients for orthogonal groups and
  vector-valued\ {S}iegel modular forms.
\newblock {\em Amer. J. Math.}, 128(4):899--948, 2006.

\bibitem[GKZ87]{GKZ87}
B.~Gross, W.~Kohnen, and D.~Zagier.
\newblock Heegner points and derivatives of {$L$}-series. {II}.
\newblock {\em Math. Ann.}, 278(1-4):497--562, 1987.

\bibitem[Gou72]{Gould72}
Henry~W. Gould.
\newblock {\em Combinatorial identities}.
\newblock Henry W. Gould, Morgantown, W.Va., 1972.
\newblock A standardized set of tables listing 500 binomial coefficient
  summations.

\bibitem[GPSR87]{GPSR}
Stephen Gelbart, Ilya Piatetski-Shapiro, and Stephen Rallis.
\newblock {\em Explicit constructions of automorphic {$L$}-functions}, volume
  1254 of {\em Lecture Notes in Mathematics}.
\newblock Springer-Verlag, Berlin, 1987.

\bibitem[GQT14]{GQT14}
Wee~Teck Gan, Yannan Qiu, and Shuichiro Takeda.
\newblock The regularized {S}iegel-{W}eil formula (the second term identity)
  and the {R}allis inner product formula.
\newblock {\em Invent. Math.}, 198(3):739--831, 2014.

\bibitem[GZ85]{GZ85}
Benedict~H. Gross and Don~B. Zagier.
\newblock On singular moduli.
\newblock {\em J. Reine Angew. Math.}, 355:191--220, 1985.

\bibitem[GZ86]{GZ86}
Benedict~H. Gross and Don~B. Zagier.
\newblock Heegner points and derivatives of {$L$}-series.
\newblock {\em Invent. Math.}, 84(2):225--320, 1986.

\bibitem[Hec27]{Hecke26}
E.~Hecke.
\newblock Zur {T}heorie der elliptischen {M}odulfunktionen.
\newblock {\em Math. Ann.}, 97(1):210--242, 1927.

\bibitem[HP17]{HP17}
Philipp Habegger and Fabien Pazuki.
\newblock Bad reduction of genus 2 curves with {CM} jacobian varieties.
\newblock {\em Compos. Math.}, 153(12):2534--2576, 2017.

\bibitem[HY11]{HY11}
Benjamin Howard and Tonghai Yang.
\newblock Singular moduli refined.
\newblock In {\em Arithmetic geometry and automorphic forms}, volume~19 of {\em
  Adv. Lect. Math. (ALM)}, pages 367--406. Int. Press, Somerville, MA, 2011.

\bibitem[HY12]{HY12}
Benjamin Howard and Tonghai Yang.
\newblock {\em Intersections of {H}irzebruch-{Z}agier divisors and {CM}
  cycles}, volume 2041 of {\em Lecture Notes in Mathematics}.
\newblock Springer, Heidelberg, 2012.

\bibitem[JL70]{JL70}
H.~Jacquet and R.~P. Langlands.
\newblock {\em Automorphic forms on {${\rm GL}(2)$}}.
\newblock Lecture Notes in Mathematics, Vol. 114. Springer-Verlag, Berlin-New
  York, 1970.

\bibitem[KM90]{KM90}
Stephen~S. Kudla and John~J. Millson.
\newblock Intersection numbers of cycles on locally symmetric spaces and
  {F}ourier coefficients of holomorphic modular forms in several complex
  variables.
\newblock {\em Inst. Hautes \'Etudes Sci. Publ. Math.}, (71):121--172, 1990.

\bibitem[KR92]{KR92}
Stephen~S. Kudla and Stephen Rallis.
\newblock Ramified degenerate principal series representations for {${\rm
  Sp}(n)$}.
\newblock {\em Israel J. Math.}, 78(2-3):209--256, 1992.

\bibitem[KR94]{KR94}
Stephen~S. Kudla and Stephen Rallis.
\newblock A regularized {S}iegel-{W}eil formula: the first term identity.
\newblock {\em Ann. of Math. (2)}, 140(1):1--80, 1994.

\bibitem[Kud78]{Kudla78}
Stephen~S. Kudla.
\newblock Theta-functions and {H}ilbert modular forms.
\newblock {\em Nagoya Math. J.}, 69:97--106, 1978.

\bibitem[Kud94]{Kudla94}
Stephen~S. Kudla.
\newblock Splitting metaplectic covers of dual reductive pairs.
\newblock {\em Israel J. Math.}, 87(1-3):361--401, 1994.

\bibitem[Kud97]{Kudla97}
Stephen~S. Kudla.
\newblock Central derivatives of {E}isenstein series and height pairings.
\newblock {\em Ann. of Math. (2)}, 146(3):545--646, 1997.

\bibitem[Kud03]{Kudla03}
Stephen~S. Kudla.
\newblock Integrals of {B}orcherds forms.
\newblock {\em Compositio Math.}, 137(3):293--349, 2003.

\bibitem[Kud16]{Kudla16}
Stephen~S. Kudla.
\newblock Another product for a {B}orcherds form.
\newblock In {\em Advances in the theory of automorphic forms and their
  {$L$}-functions}, volume 664 of {\em Contemp. Math.}, pages 261--294. Amer.
  Math. Soc., Providence, RI, 2016.

\bibitem[Li16]{Li16}
Yingkun Li.
\newblock Real-dihedral harmonic {M}aass forms and {CM}-values of {H}ilbert
  modular functions.
\newblock {\em Compositio Mathematica}, 152(6):1159--1197, 2016.

\bibitem[Li21]{Li21a}
Yingkun Li.
\newblock Singular units and isogenies between {CM} elliptic curves.
\newblock {\em Compos. Math.}, 157(5):1022--1035, 2021.

\bibitem[Li22]{Li18}
Yingkun Li.
\newblock Average {CM}-values of higher green's function and factorization.
\newblock {\em American Journal of Math.}, 144(5):1241--1298, 2022.

\bibitem[Li23]{Li23}
Yingkun Li.
\newblock Algebraicity of higher {G}reen functions at a {CM} point.
\newblock {\em Invent. Math.}, 234(1):375--418, 2023.

\bibitem[LS22]{LS22}
Yingkun Li and Markus Schwagenscheidt.
\newblock Mock modular forms with integral {F}ourier coefficients.
\newblock {\em Adv. Math.}, 399:Paper No. 108264, 30 pp., 2022.

\bibitem[McG03]{McGraw03}
William~J. McGraw.
\newblock The rationality of vector valued modular forms associated with the
  {W}eil representation.
\newblock {\em Math. Ann.}, 326(1):105--122, 2003.

\bibitem[Mel08]{Mellit08}
Anton Mellit.
\newblock Higher {G}reen's functions for modular forms.
\newblock arXiv:0804.3184, 2008.

\bibitem[M{\oe}g97]{Moeglin97a}
Colette M{\oe}glin.
\newblock Non nullit\'{e} de certains rel\^{e}vements par s\'{e}ries th\'{e}ta.
\newblock {\em J. Lie Theory}, 7(2):201--229, 1997.

\bibitem[Nik79]{Ni79}
V.~V. Nikulin.
\newblock Integer symmetric bilinear forms and some of their geometric
  applications.
\newblock {\em Izv. Akad. Nauk SSSR Ser. Mat.}, 43(1):111--177, 238, 1979.

\bibitem[Ral84]{Rallis}
S.~Rallis.
\newblock On the {H}owe duality conjecture.
\newblock {\em Compositio Math.}, 51(3):333--399, 1984.

\bibitem[Sch09]{Sch09}
Nils~R. Scheithauer.
\newblock The {W}eil representation of {${\rm SL}_2(\Bbb Z)$} and some
  applications.
\newblock {\em Int. Math. Res. Not. IMRN}, (8):1488--1545, 2009.

\bibitem[Via11]{Via11}
Maryna Viazovska.
\newblock {CM} values of higher {G}reen's functions.
\newblock arXiv:1110.4654, 2011.

\bibitem[Xue10]{Xue10}
Hui Xue.
\newblock Gross-{K}ohnen-{Z}agier theorem for higher weight forms.
\newblock {\em Math. Res. Lett.}, 17(3):573--586, 2010.

\bibitem[YY19]{YY19}
Tonghai Yang and Hongbo Yin.
\newblock Difference of modular functions and their {CM} value factorization.
\newblock {\em Trans. Amer. Math. Soc.}, 371(5):3451--3482, 2019.

\bibitem[Zha97]{Zhang97}
Shouwu Zhang.
\newblock Heights of {H}eegner cycles and derivatives of {$L$}-series.
\newblock {\em Invent. Math.}, 130(1):99--152, 1997.

\end{thebibliography}
\bibliographystyle{alpha}

\end{document}